\documentclass{amsart}
\usepackage{graphicx}
\usepackage{amsmath}
\usepackage{amssymb}
\usepackage{amsthm}
\usepackage{esint}
\usepackage{mathrsfs}
\usepackage{amsfonts}
\allowdisplaybreaks[1]

\newtheorem{theorem}{Theorem}[section]
\newtheorem{lem}[theorem]{Lemma}

\theoremstyle{Corollary}
\newtheorem{cor}[theorem]{Corollary}

\numberwithin{equation}{section}

\begin{document}

\title[Invariant surface area functionals]{On the invariant surface area functionals in 3-dimensional CR geometry}

\author{Pak Tung Ho}
\address{Department of Mathematics, Tamkang University Tamsui, New Taipei City 251301, Taiwan}

\email{paktungho@yahoo.com.hk}


\date{11th of November, 2025.}

\begin{abstract}
Cheng, Yang, and Zhang have studied 
two  invariant surface area functionals 
in 3-dimensional CR manifolds. 
They 
deduced the Euler–Lagrange equations of the associated energy functionals when the 3-dimensional CR manifold
has constant Webster curvature and vanishing torsion. 
In this paper, 
we deduce the 
Euler–Lagrange equations of the energy functionals
in a more general 
3-dimensional CR manifold. 
Moreover, we study 
the invariant area functionals
on the disk bundle, 
on the Rossi sphere, 
and on  $3$-dimensional tori. 
In particular, we show that the Clifford torus 
is a minimizer for $E_1$ on the Rossi sphere $S^3_t$ when $t=-4+\sqrt{15}$. 
Also, by computing the second variation formula, 
we show that the Clifford torus is not a minimizer 
 for $E_1$ on the Rossi sphere $S^3_t$ when  $t>-4+\sqrt{15}$.

\end{abstract}

\maketitle

\section{Introduction}
 
In \cite{C}, Cheng introduced 
two CR invariant surface area functional 
$dA_1$ and $dA_2$. 
To recall them, let us review some basic notions 
in a pseudohermitian $3$-manifold.
These notions could be found in \cite{Lee}
and \cite{T}
for example. 
For a pseudohermitian $3$-manifold $(M,J,\theta)$, 
$\theta$ is the contact form
and $\xi=\ker\theta$ is the contact bundle. 
Then $J:\xi\to\xi$ is an endomorphism 
such that $J^2=-id$. 
On $(M,J,\theta)$, 
there is a canonical connection $\nabla$, which is called 
Tanaka-Webster connection. 
Associated to this connection, we have 
torsion $A_{11}$ and (Tanaka-)Webster curvature $W$. 

Let $\Sigma$ be a surface in $(M,J,\theta)$. 
We recall (c.f. \cite{CHMY}) that a point $p\in \Sigma$
is \textit{singular} 
if its tangent plane $T_p\Sigma$ coincides 
with the contact plane $\xi_p$ at $p$. 
We call a surface nonsingular if 
it does not contain any singular point. 
We recall in \cite{CHMY}
that a moving frame associated to $\Sigma$ can be chosen. 
Associated to $\Sigma$, 
we have a special frame $e_1$, $e_2:=Je_1$ such that 
$e_1\in T\Sigma\cap \ker\theta$
and has unit length 
with respect to the Levi metric $\frac{1}{2}d\theta(\cdot,J\cdot)$.
We denote the coframe dual to $e_1, e_2$ and $T$ 
by $e^1,e^2$ and $\theta$. 
A \textit{deviation function} $\alpha$ on $\Sigma$ is defined as 
$$T+\alpha e_2\in T\Sigma.$$
The \textit{$p$-mean curvature} or \textit{horizontal mean curvature}
$H$ of $\Sigma$ is defined as 
$$\nabla_{e_1}e_1=H e_2.$$

The two invariant surface area functionals
are (c.f. \cite{CYZ})
\begin{equation}\label{0.5}
E_1(\Sigma)=\int_\Sigma dA_1~~\mbox{ and }~~E_2(\Sigma)=\int_\Sigma dA_2,
\end{equation}
where 
\begin{equation}\label{0.4}
dA_1=\Big|e_1(\alpha)+\frac{1}{2}\alpha^2-\mbox{Im} A_{11}+\frac{1}{4}W+\frac{1}{6}H^2\Big|^{\frac{3}{2}}\theta\wedge e^1
\end{equation}
and 
\begin{equation}\label{0.3}
\begin{split}
dA_2 &=\Bigg[(T+\alpha e_2)(\alpha)+\frac{2}{3}\left(e_1(\alpha)+\frac{1}{2}\alpha^2-\mbox{Im} A_{11}+\frac{1}{4}W\right)H+\frac{2}{27}H^3\\
&\hspace{8mm}+\mbox{Im}\left(\frac{1}{6}W^{,1}+\frac{2i}{3}(A^{11})_{,1}\right)-\alpha\left(\mbox{Re} A^1_{\overline{1}}\right)\Bigg]\theta\wedge e^1.
\end{split}
\end{equation}
These were first found by Cheng in \cite{C}. 
Several remarks are in order. 
First note that
the integrals in (\ref{0.5})
 can only be taken over nonsingular 
region of $\Sigma$. 
So we assume
that the singular set of $\Sigma$ 
has measure zero. 
For smooth surfaces, 
the singular set is always of measure zero.
This fact was proved by 
Derridj and later by Balogh \cite{Balogh}
(see also \cite[Theorem D]{CHY}). 
Note also that 
the area functionals in (\ref{0.5})
are CR invariant, i.e. invariant under the contact form change. 
To check this, it suffices to check that the $2$-forms
in (\ref{0.4}) and (\ref{0.3}) are CR invariant. 
We have the following transformation laws under the change of contact form 
$\tilde{\theta}=\lambda^2\theta$: (see P.414-415 of \cite{CYZ})
\begin{align*}
\tilde{e}_1&=\lambda^{-1}e_1,~~\tilde{e}_2=\lambda^{-1}e_2,\\
\tilde{T}&=\lambda^{-2}T+\lambda^3\big(e_2(\lambda)e_1-e_1(\lambda) e_2\big),\\
\tilde{\alpha}&=\lambda^{-1}\alpha+\lambda^{-2}e_1(\lambda),\\
\tilde{H}&=\lambda^{-1}H-3\lambda^{-2}e_2(\lambda),\\
\mbox{Im}\widetilde{A}_{11}&=\lambda^{-2}\mbox{Im}A_{11}+\frac{1}{2}\big[\lambda^{-4}e_2(\lambda)^2-\lambda^{-4}e_1(\lambda)^2
+\lambda^{-1}(\lambda^{-1})_{22}-\lambda^{-1}(\lambda^{-1})_{11}\big],\\
\widetilde{W}&=2\lambda^{-1}\big[ (\lambda^{-1})_{11}+(\lambda^{-1})_{22}\big]
-4\big[\lambda^{-4}e_1(\lambda)^2+\lambda^{-4}e_2(\lambda)^2\big]+\lambda^{-2}W.
\end{align*}
Then the invariance of $dA_1$ and $dA_2$ can be verified directly by using 
these transformation laws.

When $(M,J,\theta)$ has constant Webster curvature and vanishing torsion, 
the Euler-Lagrange equation for $E_1$ reads (c.f. \cite[Theorem 3]{CYZ})
$$\mathcal{E}_1=0,$$
where 
\begin{equation}\label{EE_for_E1}
\begin{split}
\mathcal{E}_1&=\frac{1}{2}e_1(|H_{cr}|^{1/2}\mathfrak{f})+\frac{3}{2}|H_{cr}|^{1/2}\alpha\mathfrak{f}\\
&\hspace{4mm}+\frac{1}{2}sign(H_{cr})|H_{cr}|^{1/2}\left(9h_{00}+6h_{11}h_{10}+\frac{2}{3}h_{11}^3\right),
\end{split}
\end{equation}
where
\begin{equation}\label{0.1}
\begin{split}
H_{cr}&=e_1(\alpha)+\frac{1}{2}\alpha^2+\frac{1}{4}W+\frac{1}{6}H^2,\\
|H_{cr}|\mathfrak{f}&:=e_1(H)\left(e_1(\alpha)+\frac{1}{2}\alpha^2+\frac{1}{3}H^2+\frac{1}{4}W\right)\\
&\hspace{4mm}+H(T+\alpha e_2)(H)+\frac{3}{2}(T+\alpha e_2)\left(e_1(\alpha)+\frac{1}{2}\alpha^2\right)\\
&\hspace{4mm}-\frac{7}{2}\alpha H e_1(\alpha)-\frac{5}{2}\alpha^3 H-\frac{2}{3}\alpha H^3-\frac{5}{4}\alpha H W
\end{split}
\end{equation}
and 
\begin{equation}\label{0.2}
\begin{split}
&9h_{00}+6h_{11}h_{10}+\frac{2}{3}h_{11}^3\\
&=9(T+\alpha e_2)(\alpha)+6H\left(e_1(\alpha)+\frac{1}{2}\alpha^2+\frac{1}{4}W\right)
+\frac{2}{3}H^3. 
\end{split}
\end{equation}
Furthermore, when $(M,J,\theta)$ has constant Webster curvature and vanishing torsion, 
the Euler-Lagrange equation for $E_2$ reads (c.f. \cite[Theorem 8 and Theorem 9]{CYZ})
$$\mathcal{E}_2=0,$$
where
\begin{equation}\label{EE_for_E2}
\begin{split}
\frac{9}{4}\mathcal{E}_2&=H e_1e_1(H)+3 e_1 V(H)+e_1(H)^2+\frac{1}{3}H^4\\
&\hspace{4mm}+3 e_1(\alpha)^2+12\alpha^2 e_1(\alpha)+12\alpha^4\\
&\hspace{4mm}-\alpha H e_1(H)+2 H^2 e_1(\alpha)+5\alpha^2 H^2\\
&\hspace{4mm}+\frac{3}{2} W\left(e_1(\alpha)+\frac{2}{3}H^2+5\alpha^2+\frac{1}{2}W\right),
\end{split}
\end{equation}
where $V:=T+\alpha e_2$. 
Among other things, 
Cheng, Yang, and Zhang 
\cite{CYZ} found
many examples of critical points of 
the invariant area functionals $E_1$ and $E_2$  
when $(M, J,\theta)$ is the Heisenberg group or 
the CR sphere. 
They also showed that they  are related to the CR version of the singular Yamabe problem. 
In particular, it is shown that the invariant area surface functional $E_2$ determines the log term in the volume renormalization
for a formal solution to the singular Yamabe problem (c.f. \cite[Theorem 1]{CYZ}), thus providing an obstruction to the existence of smooth solutions of the singular
CR Yamabe problem. See also \cite{K} for results about 
solving the singular CR Yamabe problem on domains with nonsingular boundary.

In this paper, we continue 
the study of invariant surface area functionals $E_1$ and $E_2$ defined in (\ref{0.5}). 
First we study these area functionals 
of nonsingular surface $\Sigma$ in the disk bundle $B^1\times \mathbb{R}$, 
where $B^1$ is the unit disk in $\mathbb{C}$, i.e. 
$B^1=\{z\in \mathbb{C}: |z|^2<1\}$. 
In Section \ref{section2}, we recall some computations in the disk bundle 
$B^1\times\mathbb{R}$. 
In particular, we will see in Section \ref{section2} that 
the Webster scalar curvature is constant and 
the torsion vanishes for the disk bundle $B^1\times\mathbb{R}$. Hence, the  Euler-Lagrange equations
for $E_1$ and $E_2$ 
deduced by Cheng et al., 
which are stated in (\ref{EE_for_E1})-(\ref{EE_for_E2}), are still applicable to 
the surfaces in $B^1\times\mathbb{R}$. 
In Section \ref{section3} and Section \ref{section4}, we are able to find several examples 
of critical points of the area functionals $E_1$ and $E_2$ in $B^1\times\mathbb{R}$.
See Lemmas \ref{lem3.1}-\ref{lem3.4} and Lemmas \ref{lem4.1}-\ref{lem4.2}. 

Next we study the area functionals $E_1$ and $E_2$
when $(M,J,\theta)$ may not have constant Webster curvature or vanishing torsion. 
To this end, 
we deduce in Section \ref{section5}
the Euler-Lagrange equation for the
invariant surface area functional $E_1$ in the general case, 
for which $(M,J,\theta)$ does not necessarily have
constant Webster curvature or vanishing torsion. 
See Theorem \ref{5.1} for the precise statement. 
In Section \ref{section6}, 
we deduce 
the Euler-Lagrange equation for the
invariant surface area functional $E_2$ 
when $(M,J,\theta)$ has constant Webster curvature 
and  constant, purely imaginary torsion. 
See Theorem \ref{thm7.1} for the precise statement. 

In Section \ref{section8}, 
we study the invariant area functionals 
on the Rossi sphere. 
We will see in Section \ref{section8} 
that, on the Rossi sphere, the Webster curvature is constant and the torsion is constant and purely imaginary. 
In particular, Theorem \ref{thm7.1} is applicable to the Rossi sphere. 
Among other things, we prove that 
the Clifford torus is a minimizer for   $E_1$
with zero energy in the Rossi sphere $S^3_{t_0}$, where 
$t_0=-4+\sqrt{15}$. See Lemma \ref{lem5.1}.
We remark that this is a sharp contrast 
to the situation in the CR sphere. 
Indeed, it was proved in \cite{CYZ}
that if $T^2$ is a torus in the CR sphere 
without any singular points, then it cannot have $E_1=0$. 
We also prove that the 
Clifford torus 
is a critical point for $E_1$
in the Rossi sphere $S^3_{t}$ for any $t$. 
See Lemma \ref{lem5.2}.
On the other hand, we prove that 
the Clifford torus 
is a critical point for $E_2$
in the Rossi sphere $S^3_{t}$ 
if and only if $t=-4+\sqrt{15}$.

In Section \ref{section10}, 
we compute the second variation of the Clifford torus 
on the Rossi sphere. See Theorem \ref{thm8.1}.
In particular, we show that the Clifford torus is not a
minimizer 
of $E_1$ on the Rossi sphere $S^3_t$, when   $t>-4+\sqrt{15}$. 
This is inspired by the very recent result 
of Cheng, Chiu, Yang, and Zhang in \cite{CCYZ}, 
in which they computed the second variation of the Clifford torus 
on the standard CR sphere. 
In particular, they showed the Clifford torus is not a minimizer 
of $E_1$ on the standard CR sphere.

In Section \ref{section9}, we study the invariant area functionals
when $(M,J,\theta)$ is a $3$-dimensional torus. 
We note that 
for different generating curve, it defines different $3$-dimensional tori. 
We also note in Section \ref{section9} that  the torsion of such torus never vanishes. 
In particular, the Euler-Lagrangian equations 
(\ref{EE_for_E1})-(\ref{EE_for_E2}) are not applicable. 
This shows the necessity of proving 
Theorem \ref{thm5.1} in Section \ref{section5} and Theorem \ref{thm7.1}
in Section \ref{section6}. 
We are able to find critical points 
for $E_1$ and $E_2$ on some $3$-dimensional tori. 
See Lemmas \ref{lem6.1} and \ref{lem6.2}.

\section{Computations in $B^1\times\mathbb{R}$}\label{section2}

We include some computations in $B^1\times\mathbb{R}$ for the readers' convenience. 
All these could actually be found in \cite{HH}, in which 
the differential geometry of curves in the disk bundle $B^1\times\mathbb{R}$ has been studied. 

On the disk bundle $B^1\times \mathbb{R}$, 
where $B^1$ is the unit disk in $\mathbb{C}$, i.e. 
$B^1=\{z\in \mathbb{C}: |z|^2<1\}$, 
the contact form is given by 
\begin{equation}\label{1.1}
\theta=dt+\frac{2i}{1-|z|^2}(zd\overline{z}-\overline{z}dz)
\end{equation}
for $(z,t)\in B^1\times \mathbb{R}$. 
The CR structure of $B^1\times\mathbb{R}$
is given by the decomposition of 
the complexified tangent bundle $\mathbb{C}\otimes T(B^1\times\mathbb{R})
=\mbox{span}_{\mathbb{C}}\{Z_1,Z_{\overline{1}}, T\}$
with $\theta(Z_1)=\theta(Z_{\overline{1}})=0$, 
$\theta(T)=1$, where 
\begin{equation}\label{1.2}
T=\frac{\partial}{\partial t},~~
Z_1=\frac{1-|z|^2}{2}\frac{\partial}{\partial z}+i\overline{z}\frac{\partial}{\partial t},~~
Z_{\overline{1}}=\frac{1-|z|^2}{2}\frac{\partial}{\partial \overline{z}}-iz\frac{\partial}{\partial t}.
\end{equation}
The contact bundle $\xi$ is given by $\xi=\mbox{span}_{\mathbb{C}}\{Z_1,Z_{\overline{1}}\}$. 
The almost complex structure $J:\xi\to\xi$
is defined such that $J^2=-I$
with $J(Z_1)=iZ_1$ and $J(Z_{\overline{1}})=-iZ_{\overline{1}}$.

If we define 
\begin{equation}\label{2.0}
\theta^1=\frac{2}{1-|z|^2}dz~~\mbox{
and }~~\theta^{\overline{1}}=\frac{2}{1-|z|^2}d\overline{z},
\end{equation}
then it follows from (\ref{1.1}) that 
\begin{equation}\label{2.1}
d\theta=\frac{4i}{(1-|z|^2)^2}dz\wedge d\overline{z}=i\theta^1\wedge \theta^{\overline{1}}.
\end{equation}
Since $d\theta=ih_{1\overline{1}}\theta^1\wedge \theta^{\overline{1}}$ (see (2.2) in \cite{Lee}), it follows from (\ref{2.1}) that 
\begin{equation}\label{2.11}
h_{1\overline{1}}=1.
\end{equation}
Note that the frame $\{Z_1,Z_{\overline{1}},T\}$ defined in (\ref{1.2})
is dual to the coframe $\{\theta^1,\theta^{\overline{1}}, \theta\}$. 

The Levi metric is defined as 
\begin{equation}\label{Levi}
g_\theta=\theta\otimes\theta+\frac{1}{2}d\theta(\cdot,J\cdot).
\end{equation}
Then we have 
\begin{equation}\label{2.2}
g_\theta(Z_1,Z_1)=g_\theta(Z_{\overline{1}},Z_{\overline{1}})=0,~~
g_\theta(Z_1,Z_{\overline{1}})=g_\theta(Z_{\overline{1}},Z_1)=\frac{1}{2}.
\end{equation}
To see this, we compute 
\begin{equation*}
\begin{split}
g_\theta(Z_1,Z_{\overline{1}})&=\frac{1}{2}d\theta(Z_1,JZ_{\overline{1}})
=\frac{1}{2}d\theta(Z_1,-iZ_{\overline{1}})=-\frac{i}{2}d\theta(Z_1,Z_{\overline{1}})\\
&=\frac{1}{2}(\theta^1\wedge \theta^{\overline{1}})(Z_1,Z_{\overline{1}})=\frac{1}{2},
\end{split}
\end{equation*}
we have used (\ref{2.1}). The others can be computed similarly. 

Note that $Z_1$ and $Z_{\overline{1}}$ can be written as 
\begin{equation}\label{2.3}
\begin{split}
Z_1&=\frac{1-|z|^2}{2}\cdot\frac{1}{2}\left(\frac{\partial}{\partial x}-i\frac{\partial}{\partial y}\right)
+i(x-iy)\frac{\partial}{\partial t}\\
&=\left(\frac{1-|z|^2}{4}\frac{\partial}{\partial x}+y\frac{\partial}{\partial t}\right)
-i\left(\frac{1-|z|^2}{4}\frac{\partial}{\partial y}-x\frac{\partial}{\partial t}\right)
\end{split}
\end{equation}
and 
\begin{equation}\label{2.4}
Z_{\overline{1}}=\left(\frac{1-|z|^2}{4}\frac{\partial}{\partial x}+y\frac{\partial}{\partial t}\right)
+i\left(\frac{1-|z|^2}{4}\frac{\partial}{\partial y}-x\frac{\partial}{\partial t}\right)
\end{equation}
where $z=x+iy$. 
Therefore, if we write 
\begin{equation}\label{2.5}
\overset{\circ}{e}_1=\frac{1-|z|^2}{2}\frac{\partial}{\partial x}+2y\frac{\partial}{\partial t} 
~~\mbox{ and }~~\overset{\circ}{e}_2= \frac{1-|z|^2}{2}\frac{\partial}{\partial y}-2x\frac{\partial}{\partial t}
\end{equation}
where $z=x+iy$, 
then it follows from (\ref{2.3})-(\ref{2.5}) that 
\begin{equation}\label{2.6}
Z_1=\frac{1}{2}(\overset{\circ}{e}_1-i\overset{\circ}{e}_2)~~\mbox{ and }~~Z_{\overline{1}}=\frac{1}{2}(\overset{\circ}{e}_1+i\overset{\circ}{e}_2),
\end{equation}
and hence, 
\begin{equation}\label{2.8}
J(\overset{\circ}{e}_1)=\overset{\circ}{e}_2~~\mbox{ and }~~J(\overset{\circ}{e}_2)=-\overset{\circ}{e}_1.
\end{equation}
It follows from (\ref{2.6}) that 
\begin{equation}\label{2.7}
\overset{\circ}{e}_1=Z_1+Z_{\overline{1}}~~\mbox{ and }~~
\overset{\circ}{e}_2=i(Z_1-Z_{\overline{1}}).
\end{equation}
From this, one can easily verify that $\{\overset{\circ}{e}_1,\overset{\circ}{e}_2,T\}$ is an orthonormal basis 
with respect to the Levi metric $g_\theta$. 
For example, we find 
\begin{equation*}
\begin{split}
g_\theta(\overset{\circ}{e}_2,\overset{\circ}{e}_2)
&=-g_\theta(Z_1-Z_{\overline{1}},Z_1-Z_{\overline{1}})\\
&=-\big(g_\theta(Z_1,Z_1)-g_\theta(Z_1,Z_{\overline{1}})-g_\theta(Z_{\overline{1}},Z_1)+g_\theta(Z_{\overline{1}},Z_{\overline{1}})\big)
=\frac{1}{2}+\frac{1}{2}=1
\end{split}
\end{equation*}
by using (\ref{2.2}), and the others can be computed similarly.

It follows from (\ref{2.0}) that 
\begin{equation}\label{2.9}
\begin{split}
d\theta^1&=\frac{2z}{(1-|z|^2)^2}d\overline{z}\wedge dz=\theta^1\wedge\left(-\frac{z}{2}\theta^{\overline{1}}\right),\\
d\theta^{\overline{1}}&=\frac{2\overline{z}}{(1-|z|^2)^2}dz\wedge d\overline{z}
=\theta^{\overline{1}}\wedge\left(-\frac{\overline{z}}{2}\theta^1\right).
\end{split}
\end{equation}
The structure equations imply that (see (4.1) and (4.2) in \cite{Lee})
\begin{equation}\label{2.10}
d\theta^1=\theta^1\wedge\omega_1^1+\theta\wedge\tau^1~~
\mbox{ and }~~\omega_1^1+\omega_{\overline{1}}^{\overline{1}}=dh_{1\overline{1}},
\end{equation}
where $\omega_1^1$ is the connection $1$-form and $\tau$ is the torsion form. 
Comparing this with (\ref{2.9}) and using (\ref{2.11}), we obtain 
\begin{equation}\label{2.12}
\omega_1^1=\frac{\overline{z}}{2}\theta^1-\frac{z}{2}\theta^{\overline{1}}~~
\mbox{ and }~~\tau\equiv 0.
\end{equation}
By (\ref{2.12}), we compute
\begin{equation}\label{2.13}
\begin{split}
\nabla_{Z_1}Z_1&=\omega_1^1(Z_1)Z_1=\frac{\overline{z}}{2}Z_1,~~
\nabla_{Z_{\overline{1}}}Z_{\overline{1}}=\overline{(\nabla_{Z_1}Z_1)}=\frac{z}{2}Z_{\overline{1}},\\
\nabla_{Z_{\overline{1}}}Z_1&=\omega_1^1(Z_{\overline{1}})Z_1=-\frac{z}{2}Z_1,~~
\nabla_{Z_1}Z_{\overline{1}}=\overline{(\nabla_{Z_{\overline{1}}}Z_1)}=-\frac{\overline{z}}{2}Z_{\overline{1}}.
\end{split}
\end{equation}
From (\ref{2.7}) and (\ref{2.13}), we find
\begin{equation}\label{2.14}
\begin{split}
\nabla_{\overset{\circ}{e}_1}\overset{\circ}{e}_1
&=\nabla_{Z_1+Z_{\overline{1}}}(Z_1+Z_{\overline{1}})
=\nabla_{Z_1}Z_1+\nabla_{Z_1}Z_{\overline{1}}+\nabla_{Z_{\overline{1}}}Z_1+\nabla_{Z_{\overline{1}}}Z_{\overline{1}}\\
&=\frac{\overline{z}-z}{2}(Z_1-Z_{\overline{1}})=-y\overset{\circ}{e}_2,\\
\nabla_{\overset{\circ}{e}_2}\overset{\circ}{e}_2
&=-\nabla_{Z_1-Z_{\overline{1}}}(Z_1-Z_{\overline{1}})
=-\nabla_{Z_1}Z_1+\nabla_{Z_1}Z_{\overline{1}}+\nabla_{Z_{\overline{1}}}Z_1-\nabla_{Z_{\overline{1}}}Z_{\overline{1}}\\
&=-\frac{z+\overline{z}}{2}(Z_1+Z_{\overline{1}})=-x\overset{\circ}{e}_1,\\
\nabla_{\overset{\circ}{e}_1}\overset{\circ}{e}_2
&=i\nabla_{Z_1+Z_{\overline{1}}}(Z_1-Z_{\overline{1}})
=i\Big(\nabla_{Z_1}Z_1-\nabla_{Z_1}Z_{\overline{1}}+\nabla_{Z_{\overline{1}}}Z_1-\nabla_{Z_{\overline{1}}}Z_{\overline{1}}\Big)\\
&=i\frac{\overline{z}-z}{2}(Z_1+Z_{\overline{1}})=y\overset{\circ}{e}_1,\\
\nabla_{\overset{\circ}{e}_2}\overset{\circ}{e}_1
&=i\nabla_{Z_1-Z_{\overline{1}}}(Z_1+Z_{\overline{1}})
=i\Big(\nabla_{Z_1}Z_1+\nabla_{Z_1}Z_{\overline{1}}-\nabla_{Z_{\overline{1}}}Z_1-\nabla_{Z_{\overline{1}}}Z_{\overline{1}}\Big)\\
&=i\frac{z+\overline{z}}{2}(Z_1-Z_{\overline{1}})=x\overset{\circ}{e}_2.
\end{split}
\end{equation}

From (\ref{2.9}) and (\ref{2.12}), we find
$d\omega_1^1=-\frac{1}{2}\theta^1\wedge\theta^{\overline{1}}.$
Hence, the Webster curvature is given by 
\begin{equation}\label{2.15}
W=-\frac{1}{2}. 
\end{equation}
Note that, on the disk bundle $B^1\times\mathbb{R}$,
it follows from 
(\ref{2.12}) and (\ref{2.15}) that  
the Webster scalar curvature is constant and 
the torsion vanishes. 
In particular, 
the  Euler-Lagrange equations
for $E_1$ and $E_2$ stated in (\ref{EE_for_E1})  
and (\ref{EE_for_E2}) are still applicable to 
the surfaces $\Sigma$ in $B^1\times\mathbb{R}$.

\section{Examples of critical points of $E_1$
in the disk bundle}\label{section3}

In this section, we look for  critical points of 
the area functional $E_1$ in the disk bundle. 

\subsection{Example 1.} 

Consider the surface $\Sigma$ in $B^1\times\mathbb{R}$ 
defined by $ax+by=c$ for some real constants $a$, $b$, and $c$
such that $a^2+b^2\neq 0$. 
Since $\frac{\partial}{\partial t}$ lies in $T\Sigma$, 
we conclude that the deviation function 
\begin{equation}\label{3.0}
\alpha\equiv 0.
\end{equation}
Let $u=ax+by-c$ be a defining function. We compute
the sub-gradient of $u$ as follows 
$$\nabla_b u
=\overset{\circ}{e}_1(u) \overset{\circ}{e}_1+\overset{\circ}{e}_2(u) \overset{\circ}{e}_2
=\frac{1-|z|^2}{2}(a\overset{\circ}{e}_1+b\overset{\circ}{e}_2),$$
where $\overset{\circ}{e}_1$ and $\overset{\circ}{e}_2$ are defined as in (\ref{2.5}). 
This implies 
\begin{equation}\label{3.1}
e_2=\frac{\nabla_bu}{|\nabla_b u|}=\frac{a\overset{\circ}{e}_1+b\overset{\circ}{e}_2}{\sqrt{a^2+b^2}},
\end{equation}
where the length $|\cdot|$ is with respect to the Levi metric $g_\theta$ defined in (\ref{Levi}). 
This together with (\ref{2.8}) implies that 
\begin{equation}\label{3.2}
e_1=-Je_2=\frac{b\overset{\circ}{e}_1-a\overset{\circ}{e}_2}{\sqrt{a^2+b^2}}. 
\end{equation}
From (\ref{2.14}) and (\ref{3.2}), 
we compute 
\begin{equation*}
\begin{split}
\nabla_{e_1}\overset{\circ}{e}_1&=\frac{b}{\sqrt{a^2+b^2}}\nabla_{\overset{\circ}{e}_1}\overset{\circ}{e}_1
-\frac{a}{\sqrt{a^2+b^2}}\nabla_{\overset{\circ}{e}_2}\overset{\circ}{e}_1
=-\frac{ax+by}{\sqrt{a^2+b^2}}\overset{\circ}{e}_2=-\frac{c}{\sqrt{a^2+b^2}}\overset{\circ}{e}_2,\\
\nabla_{e_1}\overset{\circ}{e}_2&=\frac{b}{\sqrt{a^2+b^2}}\nabla_{\overset{\circ}{e}_1}\overset{\circ}{e}_2
-\frac{a}{\sqrt{a^2+b^2}}\nabla_{\overset{\circ}{e}_2}\overset{\circ}{e}_2
=\frac{ax+by}{\sqrt{a^2+b^2}}\overset{\circ}{e}_1=\frac{c}{\sqrt{a^2+b^2}}\overset{\circ}{e}_1,
\end{split}
\end{equation*}
which together with (\ref{3.1}) gives 
\begin{equation*}
\begin{split}
\nabla_{e_1}e_1
&=\frac{b}{\sqrt{a^2+b^2}}\nabla_{e_1}\overset{\circ}{e}_1
-\frac{a}{\sqrt{a^2+b^2}}\nabla_{e_1}\overset{\circ}{e}_2\\
&=-\frac{c}{a^2+b^2}(a\overset{\circ}{e}_1+b\overset{\circ}{e}_2)
=-\frac{c}{\sqrt{a^2+b^2}}e_2.
\end{split}
\end{equation*}
Therefore, we have 
\begin{equation}\label{3.3}
H=-\frac{c}{\sqrt{a^2+b^2}}.
\end{equation}
In particular, if 
\begin{equation}\label{3.A}
\frac{c^2}{a^2+b^2}=\frac{3}{4},
\end{equation}
then it follows from (\ref{2.12}), (\ref{2.15}), (\ref{3.0}) and (\ref{3.3}) that 
$$e_1(\alpha)+\frac{1}{2}\alpha^2-\mbox{Im} A_{11}+\frac{1}{4}W+\frac{1}{6}H^2
=0.$$
From this, we have the following: 
\begin{lem}\label{lem3.1}
Suppose $\Sigma$ is a surface in $B^1\times\mathbb{R}$ 
defined by $ax+by=c$ for some real constants $a$, $b$, and $c$
satisfying \eqref{3.A}. 
Then $\Sigma$ is a minimizer for the energy $E_1$ with zero energy.  
\end{lem}

From (\ref{2.15}) and (\ref{3.0}), we have 
$$H_{cr}=e_1(\alpha)+\frac{1}{2}\alpha^2-\mbox{Im} A_{11}+\frac{1}{4}W+\frac{1}{6}H^2
=-\frac{1}{8}+\frac{1}{6}\left(\frac{c^2}{a^2+b^2}\right).$$
Suppose that $H_{cr}\neq 0$, i.e. 
(\ref{3.A}) does not hold. 
Note that $H$ is constant by (\ref{3.3}). 
Combining this with (\ref{3.0}), we obtain 
$$|H_{cr}|\mathfrak{f}=0$$
by (\ref{0.1}).
Hence, it follows from (\ref{EE_for_E1})-(\ref{0.2}) 
that $\mathcal{E}_1=0$ if and only if 
$$\frac{6}{4}HW+\frac{2}{3}H^3=0,$$
which is equivalent to 
$$H\left(\frac{2}{3}H^2-\frac{3}{4}\right)=0$$
by (\ref{2.15}). Solving this and using (\ref{3.3}), we get
\begin{equation}\label{3.B}
c=0~~\mbox{ or }~~\frac{c^2}{a^2+b^2}=\frac{9}{8}.
\end{equation}
From this, we have the following: 

\begin{lem}\label{lem3.2}
Suppose $\Sigma$ is a surface in $B^1\times\mathbb{R}$ 
defined by $ax+by=c$ for some real constants $a$, $b$, and $c$
satisfying \eqref{3.B}. 
Then $\Sigma$ is a
non-minimizing critical point for the energy $E_1$.  
\end{lem}

\subsection{Example 2.}
More generally, we consider the surface $\Sigma$ in $B^1\times\mathbb{R}$ 
defined by $f(x,y)=c$ for some constant $c$. 
Since $\frac{\partial}{\partial t}$ lies in $T\Sigma$, 
the deviation function vanishes, i.e.
\begin{equation}\label{A3.4}
\alpha\equiv 0.
\end{equation} 
The defining function is 
$u=f-c$. 
The sub-gradient of $u$ is 
\begin{equation*}
\nabla_b u=\overset{\circ}{e}_1(u) \overset{\circ}{e}_1+\overset{\circ}{e}_2(u) \overset{\circ}{e}_2
=\frac{1-|z|^2}{2}(f_x\overset{\circ}{e}_1+f_y\overset{\circ}{e}_2),
\end{equation*}
which implies that 
\begin{equation}\label{A3.5}
e_2=\frac{\nabla_b u}{|\nabla_b u|}
=\frac{f_x\overset{\circ}{e}_1+f_y\overset{\circ}{e}_2}{\sqrt{f_x^2+f_y^2}}.
\end{equation}
From (\ref{2.8}) and (\ref{A3.5}), we find 
\begin{equation}\label{A3.6}
e_1=-Je_2=\frac{f_y\overset{\circ}{e}_1-f_x\overset{\circ}{e}_2}{\sqrt{f_x^2+f_y^2}}.
\end{equation}
Since 
\begin{align*}
\overset{\circ}{e}_1\Bigg(\frac{f_x}{\sqrt{f_x^2+f_y^2}}\Bigg)=\frac{1-|z|^2}{2}\frac{f_y^2f_{xx}-f_xf_yf_{xy}}{(f_x^2+f_y^2)^{3/2}},\\
\overset{\circ}{e}_2\Bigg(\frac{f_x}{\sqrt{f_x^2+f_y^2}}\Bigg)=\frac{1-|z|^2}{2}\frac{f_y^2f_{xy}-f_xf_yf_{yy}}{(f_x^2+f_y^2)^{3/2}},\\
\overset{\circ}{e}_1\Bigg(\frac{f_y}{\sqrt{f_x^2+f_y^2}}\Bigg)=\frac{1-|z|^2}{2}\frac{f_x^2f_{xy}-f_xf_yf_{xx}}{(f_x^2+f_y^2)^{3/2}},\\
\overset{\circ}{e}_2\Bigg(\frac{f_y}{\sqrt{f_x^2+f_y^2}}\Bigg)=\frac{1-|z|^2}{2}\frac{f_x^2f_{yy}-f_xf_yf_{xy}}{(f_x^2+f_y^2)^{3/2}},
\end{align*}
we have
\begin{equation}\label{A3.7}
\begin{split}
e_1\Bigg(\frac{f_x}{\sqrt{f_x^2+f_y^2}}\Bigg)
&= \frac{f_y}{\sqrt{f_x^2+f_y^2}} \overset{\circ}{e}_1\Bigg(\frac{f_x}{\sqrt{f_x^2+f_y^2}}\Bigg)
- \frac{f_x}{\sqrt{f_x^2+f_y^2}} \overset{\circ}{e}_2\Bigg(\frac{f_x}{\sqrt{f_x^2+f_y^2}}\Bigg)\\
&=\frac{1-|z|^2}{2}\frac{f_y}{(f_x^2+f_y^2)^{3/2}}(f_y^2f_{xx}-2f_xf_y f_{xy}+f_x^2 f_{yy}),\\
e_1\Bigg(\frac{f_y}{\sqrt{f_x^2+f_y^2}}\Bigg)
&= \frac{f_y}{\sqrt{f_x^2+f_y^2}} \overset{\circ}{e}_1\Bigg(\frac{f_y}{\sqrt{f_x^2+f_y^2}}\Bigg)
- \frac{f_x}{\sqrt{f_x^2+f_y^2}} \overset{\circ}{e}_2\Bigg(\frac{f_y}{\sqrt{f_x^2+f_y^2}}\Bigg)\\
&=-\frac{1-|z|^2}{2}\frac{f_x}{(f_x^2+f_y^2)^{3/2}}(f_y^2f_{xx}-2f_xf_y f_{xy}+f_x^2 f_{yy}).
\end{split}
\end{equation} 
Combining (\ref{A3.6}) and (\ref{A3.7}), we compute 
\begin{align*}
&\nabla_{e_1}e_1
=\nabla_{e_1}\left(\frac{f_y\overset{\circ}{e}_1-f_x\overset{\circ}{e}_2}{\sqrt{f_x^2+f_y^2}}\right)\\
&=e_1\Bigg(\frac{f_y}{\sqrt{f_x^2+f_y^2}}\Bigg)\overset{\circ}{e}_1-e_1\Bigg(\frac{f_x}{\sqrt{f_x^2+f_y^2}}\Bigg)\overset{\circ}{e}_2
+\frac{f_y}{\sqrt{f_x^2+f_y^2}}\nabla_{e_1}\overset{\circ}{e}_1
-\frac{f_x}{\sqrt{f_x^2+f_y^2}}\nabla_{e_1}\overset{\circ}{e}_2\\
&=e_1\Bigg(\frac{f_y}{\sqrt{f_x^2+f_y^2}}\Bigg)\overset{\circ}{e}_1-e_1\Bigg(\frac{f_x}{\sqrt{f_x^2+f_y^2}}\Bigg)\overset{\circ}{e}_2\\
&\hspace{4mm}+\frac{f_y}{\sqrt{f_x^2+f_y^2}}\left(\frac{f_y}{\sqrt{f_x^2+f_y^2}}\nabla_{\overset{\circ}{e}_1}\overset{\circ}{e}_1
-\frac{f_x}{\sqrt{f_x^2+f_y^2}}\nabla_{\overset{\circ}{e}_2}\overset{\circ}{e}_1\right)\\
&\hspace{4mm}
-\frac{f_x}{\sqrt{f_x^2+f_y^2}}\left(\frac{f_y}{\sqrt{f_x^2+f_y^2}}\nabla_{\overset{\circ}{e}_1}\overset{\circ}{e}_2
-\frac{f_x}{\sqrt{f_x^2+f_y^2}}\nabla_{\overset{\circ}{e}_2}\overset{\circ}{e}_2\right)\\
&=-\frac{1-|z|^2}{2}\frac{f_x}{(f_x^2+f_y^2)^{3/2}}(f_y^2f_{xx}-2f_xf_y f_{xy}+f_x^2 f_{yy})\overset{\circ}{e}_1\\
&\hspace{4mm}
-\frac{1-|z|^2}{2}\frac{f_y}{(f_x^2+f_y^2)^{3/2}}(f_y^2f_{xx}-2f_xf_y f_{xy}+f_x^2 f_{yy})\overset{\circ}{e}_2\\
&\hspace{4mm}
+\frac{f_y}{ f_x^2+f_y^2}\left(- f_y y\overset{\circ}{e}_2
- f_x x\overset{\circ}{e}_2\right) 
-\frac{f_x}{ f_x^2+f_y^2 }\left( f_y y\overset{\circ}{e}_1
+ f_x x\overset{\circ}{e}_1\right)\\
&=-\left(\frac{1-|z|^2}{2}\frac{f_y^2f_{xx}-2f_xf_y f_{xy}+f_x^2f_{yy}}{f_x^2+f_y^2} +\frac{xf_x+yf_y}{\sqrt{f_x^2+f_y^2}}\right)\frac{f_x\overset{\circ}{e}_1+f_y\overset{\circ}{e}_2}{\sqrt{f_x^2+f_y^2}}.
\end{align*}
Combining this with (\ref{A3.5}), we deduce that 
\begin{equation}\label{A3.8}
H=-\frac{1-|z|^2}{2}\frac{f_y^2f_{xx}-2f_xf_y f_{xy}+f_x^2f_{yy}}{f_x^2+f_y^2}-\frac{xf_x+yf_y}{\sqrt{f_x^2+f_y^2}}
\end{equation}

\subsubsection{Example 2.1} If $f=f(r)$ depends only on $r=\sqrt{x^2+y^2}$ such that $f'(r)>0$, then
\begin{equation}\label{A3.9}
\begin{split}
&f_x=\frac{xf'}{r}, ~~f_y=\frac{yf'}{r},~~f_{xx}=\frac{f'}{r}-\frac{x^2f'}{r^3}+\frac{x^2f''}{r^2},\\
&
f_{xy}=-\frac{xyf'}{r^3}+\frac{xyf''}{r^2} ,~~
f_{yy}=\frac{f'}{r}-\frac{y^2f'}{r^3}+\frac{y^2f''}{r^2}.
\end{split}
\end{equation}
Here, $'$ denotes the derivative of $f$ with respect to $r$. 
Substituting (\ref{A3.9}) into (\ref{A3.8}) yields 
\begin{equation}\label{A3.10}
H=-\frac{1-r^2}{2}\frac{f'}{r}-r.
\end{equation}
Note that $H<0$, since $f'>0$ by assumption. 
It follows from (\ref{2.12}), (\ref{2.15}), (\ref{A3.4}) and (\ref{A3.10}) that 
$$e_1(\alpha)+\frac{1}{2}\alpha^2-\mbox{Im} A_{11}+\frac{1}{4}W+\frac{1}{6}H^2
=0$$
if and only if 
\begin{equation}\label{A3.11}
\left(\frac{1-r^2}{2}\frac{f'}{r}+r\right)^2=\frac{3}{4}.
\end{equation}

In particular, take $f(r)=ar^2$ for some positive constant $a$. Then $f(r)=ar^2=c$ and 
$f'(r)=2ar$. Putting these into (\ref{A3.11}),  we find
\begin{equation*}
a\Big(1-\frac{c}{a}\Big)+\sqrt{\frac{c}{a}} =\frac{\sqrt{3}}{2}.
\end{equation*}
Viewing it as a quadratic equation in $\sqrt{\frac{c}{a}}
$, we can solve it to get 
\begin{equation}\label{A3.12}
\sqrt{\frac{c}{a}}=\frac{1-\sqrt{4a^2+1-2a\sqrt{3}}}{2a}.
\end{equation}
In $B^1\times\mathbb{R}$, we must have $0\leq r^2<1$. 
That is to say, $0\leq r^2=\frac{c}{a}<1$. 
If we choosing $a$ such that the right hand side of (\ref{A3.12}) 
is between $0$ and $1$, 
then $\Sigma$ defined by 
$ar^2=c$ is a minimizer for the energy $E_1$ with zero energy. 
For example, we can choose $a=1/2$, then 
$$0<\frac{1-\sqrt{4a^2+1-2a\sqrt{3}}}{2a}=1-\sqrt{2-\sqrt{3}}<1,$$
and it follows from (\ref{A3.12}) that 
$c=(1-\sqrt{2-\sqrt{3}})^2/2$. 
Thus, we have the following: 

\begin{lem}\label{lem3.3}
The surface 
$\Sigma$ in $B^1\times\mathbb{R}$ defined by 
$x^2+y^2=(1-\sqrt{2-\sqrt{3}})^2$ is a minimizer for the energy $E_1$ with zero energy.
\end{lem}

\subsection{Example 3.} 

We consider the surface $\Sigma$ in $B^1\times\mathbb{R}$ 
defined by $f\big(\ln(1-r^2)\big)=t^2$ where $r^2=|z|^2=x^2+y^2$. 
Note that this is well-defined, since $1-r^2=1-|z|^2>0$ in  $B^1$.
The defining function 
$u=f-t^2$. 
The sub-gradient of $u$ is given by 
\begin{equation}\label{3.4}
\nabla_b u=(-xf'-4yt)\overset{\circ}{e}_1
+(-yf'+4xt)\overset{\circ}{e}_2.
\end{equation}
From (\ref{3.4}), we have 
\begin{equation}\label{3.5}
|\nabla_b u|^2=(-xf'-4yt)^2+(-yf'+4xt)^2=r^2\big((f')^2+16t^2\big).
\end{equation}
By (\ref{3.4}) and (\ref{3.5}), we find 
\begin{equation}\label{3.6}
e_2=\frac{\nabla_b u}{|\nabla_b u|}
=\frac{(-xf'-4yt)\overset{\circ}{e}_1
+(-yf'+4xt)\overset{\circ}{e}_2}
{r\sqrt{(f')^2+16t^2}}.
\end{equation}
From (\ref{2.8}) and (\ref{3.6}), we have 
\begin{equation}\label{3.7}
e_1=-Je_2=\frac{(-yf'+4xt)\overset{\circ}{e}_1
+(xf'+4yt)\overset{\circ}{e}_2}
{r\sqrt{(f')^2+16t^2}}.
\end{equation}

The tangent plane $T\Sigma$ is spanned by 
$$\frac{\partial}{\partial x}-\frac{xf'}{t(1-r^2)}\frac{\partial}{\partial t},~~
\frac{\partial}{\partial y}-\frac{yf'}{t(1-r^2)}\frac{\partial}{\partial t}.$$
In order to find the deviation function $\alpha$, 
we find $a$ and $b$ such that 
\begin{equation}\label{3.8A}
T+\alpha e_2=a\left(\frac{\partial}{\partial x}-\frac{xf'}{t(1-r^2)}\frac{\partial}{\partial t}\right)
+b\left(\frac{\partial}{\partial y}-\frac{yf'}{t(1-r^2)}\frac{\partial}{\partial t}\right),
\end{equation}
which is equivalent to 
\begin{equation}\label{3.8}
\begin{split}
&\frac{\partial}{\partial t}+\alpha\frac{(-xf'-4yt)\overset{\circ}{e}_1
+(-yf'+4xt)\overset{\circ}{e}_2}
{r\sqrt{(f')^2+16t^2}}\\
&=a\left(\frac{\partial}{\partial x}-\frac{xf'}{t(1-r^2)}\frac{\partial}{\partial t}\right)
+b\left(\frac{\partial}{\partial y}-\frac{yf'}{t(1-r^2)}\frac{\partial}{\partial t}\right),
\end{split}
\end{equation}
by (\ref{1.2}) and (\ref{3.6}). 
From (\ref{3.8}), we find 
\begin{equation}\label{3.9A}
\begin{split}
\alpha\frac{(-xf'-4yt)}{r\sqrt{(f')^2+16t^2}}\frac{1-r^2}{2}&=a,\\
\alpha\frac{(-yf'+4xt)}{r\sqrt{(f')^2+16t^2}}\frac{1-r^2}{2}&=b,\\
1+2y\alpha\frac{(-xf'-4yt)}{r\sqrt{(f')^2+16t^2}}
-2x\alpha\frac{(-yf'+4xt)}{r\sqrt{(f')^2+16t^2}}&=-\frac{axf'}{t(1-r^2)}-\frac{byf'}{t(1-r^2)}.
\end{split}
\end{equation}
Substituting the first two equations into the third yields 
\begin{equation}\label{3.9}
\alpha=\frac{2t}{r\sqrt{(f')^2+16t^2}}.
\end{equation}

To compute the $p$-mean curvature, we compute 
\begin{equation}\label{3.10}
\begin{split}
\nabla_{e_1}e_1
&=
\frac{-yf'+4xt}{r\sqrt{(f')^2+16t^2}}
\nabla_{e_1}\overset{\circ}{e}_1
+\frac{xf'+4yt}
{r\sqrt{(f')^2+16t^2}}\nabla_{e_1}\overset{\circ}{e}_2\\
&\hspace{4mm}+e_1\left(\frac{-yf'+4xt}{r\sqrt{(f')^2+16t^2}}\right)\overset{\circ}{e}_1
+e_1\left(\frac{xf'+4yt}
{r\sqrt{(f')^2+16t^2}}\right)\overset{\circ}{e}_2
\end{split}
\end{equation}
It follows from (\ref{2.14}), (\ref{3.6}) and (\ref{3.7}) that 
\begin{equation}\label{3.11}
\begin{split}
&\frac{-yf'+4xt}{r\sqrt{(f')^2+16t^2}}
\nabla_{e_1}\overset{\circ}{e}_1
+\frac{xf'+4yt}
{r\sqrt{(f')^2+16t^2}}\nabla_{e_1}\overset{\circ}{e}_2\\
&=\frac{-yf'+4xt}{r\sqrt{(f')^2+16t^2}}
\left(\frac{-yf'+4xt}{r\sqrt{(f')^2+16t^2}}
\nabla_{\overset{\circ}{e}_1}\overset{\circ}{e}_1
+\frac{xf'+4yt}
{r\sqrt{(f')^2+16t^2}}\nabla_{\overset{\circ}{e}_2}\overset{\circ}{e}_1\right)\\
&\hspace{4mm}
+\frac{xf'+4yt}
{r\sqrt{(f')^2+16t^2}}\left(\frac{-yf'+4xt}{r\sqrt{(f')^2+16t^2}}
\nabla_{\overset{\circ}{e}_1}\overset{\circ}{e}_2
+\frac{xf'+4yt}
{r\sqrt{(f')^2+16t^2}}\nabla_{\overset{\circ}{e}_2}\overset{\circ}{e}_2\right)\\
&=\frac{r^2f'}
{r\sqrt{(f')^2+16t^2}}e_2.
\end{split}
\end{equation}
On the other hand, we have 
\begin{equation*}
\begin{split}
\overset{\circ}{e}_1\left(\frac{-yf'+4xt}{r\sqrt{(f')^2+16t^2}}\right)
&=\frac{xyf''+2(1-r^2)t+8xy}{r\sqrt{(f')^2+16t^2}}-\frac{(-yf'+4xt)x(1-r^2)}{2r^3\sqrt{(f')^2+16t^2}}\\
&\hspace{4mm}-\frac{(-yf'+4xt)}{r((f')^2+16t^2)^{\frac{3}{2}}}(-xf'f''+32yt)
\end{split}
\end{equation*}
and 
\begin{equation*}
\begin{split}
\overset{\circ}{e}_2\left(\frac{-yf'+4xt}{r\sqrt{(f')^2+16t^2}}\right)
&=\frac{y^2f''-\frac{1-r^2}{2}f'-8x^2}{r\sqrt{(f')^2+16t^2}}-\frac{(-yf'+4xt)y(1-r^2)}{2r^3\sqrt{(f')^2+16t^2}}\\
&\hspace{4mm}-\frac{(-yf'+4xt)}{r((f')^2+16t^2)^{\frac{3}{2}}}(-yf'f''-32xt).
\end{split}
\end{equation*}
Hence, we have 
\begin{equation}\label{3.12}
\begin{split}
&e_1\left(\frac{-yf'+4xt}{r\sqrt{(f')^2+16t^2}}
\right)\\
&=\frac{-yf'+4xt}{r\sqrt{(f')^2+16t^2}}
\overset{\circ}{e}_1\left(\frac{-yf'+4xt}{r\sqrt{(f')^2+16t^2}}\right)
+\frac{xf'+4yt}
{r\sqrt{(f')^2+16t^2}}\overset{\circ}{e}_2\left(\frac{-yf'+4xt}{r\sqrt{(f')^2+16t^2}}\right)\\
&=\left[-\frac{\frac{1-r^2}{2}f'+8r^2}{r^2((f')^2+16t^2)}+\frac{16f''r^2t^2+128r^2t^2}{r^2((f')^2+16t^2)^{2}}\right](xf'+4yt).
\end{split}
\end{equation}
Similarly, we have
\begin{equation*}
\begin{split}
\overset{\circ}{e}_1\left(\frac{xf'+4yt}
{r\sqrt{(f')^2+16t^2}}\right)
&=\frac{-x^2f''+\frac{1-r^2}{2}f'+8y^2}{r\sqrt{(f')^2+16t^2}}-\frac{(xf'+4yt)x(1-r^2)}{2r^3\sqrt{(f')^2+16t^2}}\\
&\hspace{4mm}-\frac{(xf'+4yt)}{r((f')^2+16t^2)^{\frac{3}{2}}}(-xf'f''+32yt)
\end{split}
\end{equation*}
and 
\begin{equation*}
\begin{split}
\overset{\circ}{e}_2\left(\frac{xf'+4yt}
{r\sqrt{(f')^2+16t^2}}\right)
&=\frac{-xyf''+2(1-r^2)t-8xy}{r\sqrt{(f')^2+16t^2}}-\frac{(xf'+4yt)y(1-r^2)}{2r^3\sqrt{(f')^2+16t^2}}\\
&\hspace{4mm}-\frac{(xf'+4yt)}{r((f')^2+16t^2)^{\frac{3}{2}}}(-yf'f''-32xt).
\end{split}
\end{equation*}
From these, we have 
\begin{equation}\label{3.13}
\begin{split}
&e_1\left(\frac{xf'+4yt}{r\sqrt{(f')^2+16t^2}}
\right)\\
&=\frac{-yf'+4xt}{r\sqrt{(f')^2+16t^2}}
\overset{\circ}{e}_1\left(\frac{xf'+4yt}{r\sqrt{(f')^2+16t^2}}
\right)
+\frac{xf'+4yt}
{r\sqrt{(f')^2+16t^2}}\overset{\circ}{e}_2\left(\frac{xf'+4yt}{r\sqrt{(f')^2+16t^2}}
\right)\\
&=\left[\frac{\frac{1-r^2}{2}f'+8r^2}{r^2((f')^2+16t^2)}-\frac{16f''r^2t^2+128r^2t^2}{r^2((f')^2+16t^2)^{2}}\right]
(-yf'+4xt).
\end{split}
\end{equation}
Putting (\ref{3.12}) and (\ref{3.13}) together and using (\ref{3.6}), we obtain 
\begin{equation}\label{3.14}
\begin{split}
&e_1\left(\frac{-yf'+4xt}{r\sqrt{(f')^2+16t^2}}\right)\overset{\circ}{e}_1
+e_1\left(\frac{xf'+4yt}
{r\sqrt{(f')^2+16t^2}}\right)\overset{\circ}{e}_2\\
&=\left[\frac{\frac{1-r^2}{2}f'+8r^2}{r\sqrt{(f')^2+16t^2}}-\frac{16f''rt^2+128rt^2}{ ((f')^2+16t^2)^{\frac{3}{2}}}\right]e_2. 
\end{split}
\end{equation}
Substituting (\ref{3.11}) and (\ref{3.14}) into (\ref{3.10}), we find 
\begin{equation*}
\begin{split}
\nabla_{e_1}e_1
=\left[\frac{r^2f'+\frac{1-r^2}{2}f'+8r^2}{r\sqrt{(f')^2+16t^2}}-\frac{16f''rt^2+128rt^2}{ ((f')^2+16t^2)^{\frac{3}{2}}}\right]e_2. 
\end{split}
\end{equation*}
That is to say, the $p$-mean curvature is equal to 
\begin{equation}\label{3.16}
H=\frac{\frac{1+r^2}{2}f'+8r^2}{r\sqrt{(f')^2+16t^2}}-\frac{16rt^2(f'' +8)}{ ((f')^2+16t^2)^{\frac{3}{2}}}.
\end{equation}
It follows from (\ref{3.7}) and (\ref{3.9})
that 
\begin{equation}\label{3.17}
\begin{split}
e_1(\alpha)
=-\frac{4t^2(1-r^2)+4r^2f'}{r^2((f')^2+16t^2)}
+\frac{8r^2t^2f'(f''+8)}{((f')^2+16t^2)^2}.
\end{split}
\end{equation}

Now, substituting 
(\ref{2.15}, (\ref{3.9}), 
(\ref{3.16}) and (\ref{3.17})
into (\ref{0.1})
and using the fact that $f(\ln(1-r^2))=t^2$, we obtain 
\begin{equation}\label{3.22}
\begin{split}
H_{cr}&=e_1(\alpha)+\frac{1}{2}\alpha^2+\frac{1}{4}W+\frac{1}{6}H^2\\
&=-\frac{4f(1-r^2)+4r^2f'}{r^2((f')^2+16f)}
+\frac{8r^2ff'(f''+8)}{((f')^2+16f)^2}+ \frac{2f}{r^2((f')^2+16f)}-\frac{1}{8} \\
&\hspace{4mm}+\frac{1}{6}\left[\frac{(\frac{1+r^2}{2}f'+8r^2)^2}{r^2((f')^2+16f)}
-\frac{32rf(f'' +8)(\frac{1+r^2}{2}f'+8r^2)}{r((f')^2+16f)^{2}}+\frac{256r^2f^2(f'' +8)^2}{ ((f')^2+16f)^{3}}\right]. 
\end{split}
\end{equation}
From this, we see that $H_{cr}=0$ if and only if 
the right hand side of (\ref{3.22}) vanishes, 
which is an ODE in $f$.

\subsubsection{Example 3.1} Take $f\equiv c$ for some positive constant $c$. 
This implies that 
\begin{equation}\label{3.18}
f'=f''=0~~\mbox{ and }~~t^2=f=c. 
\end{equation}
Substituting (\ref{3.18}) into 
(\ref{3.9}), (\ref{3.16}), and (\ref{3.17}), we obtain 
\begin{equation}\label{3.19}
\alpha=\frac{1}{2r},~~H=0,~~e_1(\alpha)=-\frac{1-r^2}{4r^2}.
\end{equation}
Substituting (\ref{2.15}) and (\ref{3.19}) into (\ref{0.1}), 
we have
$$H_{cr}=e_1(\alpha)+\frac{1}{2}\alpha^2+\frac{1}{4}W+\frac{1}{6}H^2
=-\frac{1-r^2}{8r^2}\neq 0.$$
On the other hand, since $H=0$ by (\ref{3.19}), 
it follows from (\ref{0.1}) and (\ref{0.2}) that 
\begin{equation}\label{3.20}
\begin{split}
|H_{cr}|\mathfrak{f}&=\frac{3}{2}(T+\alpha e_2)\left(e_1(\alpha)+\frac{1}{2}\alpha^2\right),\\
9h_{00}+6h_{11}h_{10}+\frac{2}{3}h_{11}^3
&=9(T+\alpha e_2)(\alpha).
\end{split}
\end{equation}
By (\ref{3.8A}),  (\ref{3.9A}) and (\ref{3.9}), we find 
\begin{equation}\label{3.21}
T+\alpha e_2=\frac{1-r^2}{4r^2}\left(-y\frac{\partial}{\partial x}+x\frac{\partial}{\partial y}\right). 
\end{equation}
It follows from (\ref{3.21}) that 
$(T+\alpha e_2)(r)=0$, which implies that 
$(T+\alpha e_2)g(r)=0$ for any function $g$ depending only on $r$. 
In view of this fact, we can conclude 
from (\ref{EE_for_E1}),  (\ref{3.19}) and (\ref{3.20})
that $\mathcal{E}_1=0$. 
From this, we have the following: 

\begin{lem}\label{lem3.4}
The surface 
$\Sigma$ in $B^1\times\mathbb{R}$ defined by 
$t^2=c$ is a non-minimizing critical point for the energy $E_1$. 
\end{lem}

\section{Examples of critical points of $E_2$
in the disk bundle}\label{section4}

In this section, we look for critical points of 
the area functional $E_2$ in the disk bundle. 

\subsection{Example 1.}
Revisiting the Example 1 in Section \ref{section3}, we 
consider the surface $\Sigma$ in $B^1\times\mathbb{R}$ 
defined by $ax+by=c$ for some real constants $a$, $b$, and $c$
such that $a^2+b^2\neq 0$. 
From  (\ref{3.0}) and (\ref{3.3}), we have
\begin{equation}\label{4.1}
\alpha=0~~\mbox{ and }~~H=-\frac{c}{\sqrt{a^2+b^2}}.
\end{equation}
Combining this with (\ref{2.12}) and (\ref{2.15}), 
we find 
\begin{equation*}
\begin{split}
\frac{9}{4}\mathcal{E}_2
&=\frac{1}{3}H^4+\frac{3}{2}W\left(\frac{2}{3}H^2+\frac{1}{2}W\right)\\
&=\frac{1}{3}H^4-\frac{1}{2}H^2+\frac{3}{16}
=\frac{1}{3}\left(H^2-\frac{3}{4}\right)^2
=\frac{1}{3}\left(\frac{c^2}{a^2+b^2}-\frac{3}{4}\right)^2.
\end{split}
\end{equation*}
Therefore, if 
$\displaystyle\frac{c^2}{a^2+b^2}=\frac{3}{4},$
then $\mathcal{E}_2=0$. 
Not that this is the same condition given in (\ref{3.A}). 
Hence, we have the following: 

\begin{lem}\label{lem4.1}
Suppose $\Sigma$ is a surface in $B^1\times\mathbb{R}$ 
defined by $ax+by=c$ for some real constants $a$, $b$, and $c$
satisfying \eqref{3.A}. 
Then $\Sigma$ is a critical point of $E_2$. 
In fact, $\Sigma$ is a critical point of both $E_1$ and $E_2$. 
\end{lem}

\subsection{Example 2.}
Revisiting the Example 2.1 in Section \ref{section3}, we 
consider the surface $\Sigma$ in $B^1\times\mathbb{R}$ 
defined by $f(r)=c$ for some constant $c$, 
where $f=f(r)$ depends only on $r=\sqrt{x^2+y^2}$ such that $f'(r)>0$. 
By (\ref{A3.6}) and (\ref{A3.9}), 
we find 
\begin{equation}\label{4.2} 
e_1=\frac{y\overset{\circ}{e}_1-x\overset{\circ}{e}_2}{r}. 
\end{equation}
Recall that it follows from (\ref{A3.4}) and (\ref{A3.10}) 
that 
\begin{equation}\label{4.3}
\alpha=0~~\mbox{ and }~~
H=-\frac{1-r^2}{2}\frac{f'}{r}-r. 
\end{equation}
Hence, by (\ref{1.2}), (\ref{2.5}), (\ref{4.2}) and (\ref{4.3}) that 
\begin{equation}\label{4.4}
e_1(H)=0~~\mbox{ and }~~V(H)=T(H)=0.
\end{equation}
Combining (\ref{2.15}), (\ref{4.3}) and (\ref{4.4}), 
we see from (\ref{EE_for_E2}) that 
$\mathcal{E}_2=0$ if and only if 
\begin{equation}\label{4.5}
\begin{split}
0&=\frac{1}{3}H^4+\frac{3}{2}W\left(\frac{2}{3}H^2+\frac{1}{2}W\right)\\
&=\frac{1}{3}H^4-\frac{1}{2}H^2+\frac{3}{16}
=\frac{1}{3}\left(H^2-\frac{3}{4}\right)^2,
\end{split}
\end{equation}
which is equivalent to (\ref{A3.11}). 
By the same argument 
in solving (\ref{A3.11}), 
we see that 
$\Sigma$ defined by 
$r^2=(1-\sqrt{2-\sqrt{3}})^2$ is a critical point for the energy $E_2$.
Hence, we have the following: 

\begin{lem}\label{lem4.2}
The surface 
$\Sigma$ in $B^1\times\mathbb{R}$ defined by 
$x^2+y^2=(1-\sqrt{2-\sqrt{3}})^2$ is a critical point of $E_2$. 
In fact, $\Sigma$ is a critical point of both $E_1$ and $E_2$. 
\end{lem}

\section{Euler-Lagrange equation for $E_1$ in the general case}\label{section5}

In this section, we derive the Euler-Lagrange equation for the
invariant surface area functional $E_1$ in the general case, 
for which $(M,J,\theta)$ does not necessarily have
constant Webster curvature or vanishing torsion. 

To this end, recall that the Euler-Lagrange equation for $E_1$ reads
$$\mathcal{E}_1=0,$$
where (see (3.5) in \cite{CYZ})
\begin{equation}\label{5.1}
\begin{split}
\mathcal{E}_1&=\frac{1}{2}e_1(|H_{cr}|^{1/2}\mathfrak{f})+\frac{3}{2}|H_{cr}|^{1/2}\alpha\mathfrak{f}\\
&\hspace{4mm}+\frac{1}{2}sign(H_{cr})|H_{cr}|^{1/2}\left(9h_{00}+6h_{11}h_{10}+\frac{2}{3}h_{11}^3\right).
\end{split}
\end{equation}
Once again, we emphasize that we do not assume that  $(M,J,\theta)$ has
constant Webster curvature or vanishing torsion. 
We are going to find each terms on the right hand side of (\ref{5.1}). 
From (3.2) in \cite{CYZ}, we have
\begin{equation}\label{5.2}
H_{cr}=e_1(\alpha)+\frac{1}{2}\alpha^2-\mbox{Im} A_{11}+\frac{1}{4}W+\frac{1}{6}H^2. 
\end{equation}
From (3.4) in \cite{CYZ}, we have
\begin{equation}\label{5.3}
\begin{split}
|H_{cr}|\mathfrak{f}
=h_{10}h_{111}+\frac{1}{3}h_{11}^2h_{111}+h_{11}h_{110}+\frac{3}{2}h_{100}. 
\end{split}
\end{equation}
To find the terms on the right hand side of (\ref{5.3}), we 
(c.f. (3.3) in \cite{CYZ})
\begin{equation}\label{5.4}
\begin{split}
dh_{11}+3\widetilde{\phi}^1_r-h_{11}\widetilde{\phi}^1_{1r}&=h_{111}\omega^1+h_{110}\theta,\\
dh_{10}-h_{11}\widetilde{\phi}^1_r-h_{10}\widetilde{\phi}^1_r&=h_{101}\omega^1+h_{100}\theta
\end{split}
\end{equation}
with $h_{110}=h_{101}$. Here, 
 \begin{equation}\label{5.5}
\begin{split}
   h_{11}&=H,\\
   h_{10}&=e_1(\alpha)+\frac{1}{2}\alpha^2-\mbox{Im} A_{11}+\frac{1}{4}W,
\end{split}
\end{equation}
by (2.15) in \cite{CYZ}, and
\begin{equation}\label{5.6}
\widetilde{\phi}^1_{1r}=-\alpha e^1
\end{equation}
by (2.18) in \cite{CYZ}, and 
\begin{equation}\label{5.7}
\begin{split}
\widetilde{\phi}^1_r
&=(\mbox{Re} A^1_{\overline{1}}-\alpha H) e^1+\left(\mbox{Im} A^1_{\overline{1}}-\frac{1}{4}W+\frac{3}{2}\alpha^2\right)e^2\\
&\hspace{4mm}+\left[\mbox{Re}\Big(\frac{1}{6}W^{,1}+\frac{2i}{3}(A^{11})_{,1}\Big)-\alpha\Big(\omega(T+\alpha e_2)+\frac{1}{4}W\Big)-\frac{1}{2}\alpha^3\right]\theta\\
&=(\mbox{Re} A^1_{\overline{1}}-\alpha H) e^1\\
&\hspace{4mm}+\left[
\alpha \mbox{Im} A^1_{\overline{1}}+\alpha^3+
\mbox{Re}\Big(\frac{1}{6}W^{,1}+\frac{2i}{3}(A^{11})_{,1}\Big)
-\frac{1}{2}\alpha W
-\alpha \omega(T+\alpha e_2)  \right]\theta\\
&=(\mbox{Re} A^1_{\overline{1}}-\alpha H) e^1\\
&\hspace{4mm}+
\Bigg[\alpha \mbox{Im} A^1_{\overline{1}}+\alpha^3+
\mbox{Re}\Big(\frac{1}{6}W^{,1}+\frac{2i}{3}(A^{11})_{,1}\Big)
-\frac{1}{2}\alpha W-\alpha e_1(\alpha)+\alpha \mbox{Im} A_{11}\Bigg]\theta,
\end{split}
\end{equation}
where the first equality follows from (3.8) in \cite{CYZ}, 
the second equality follows from the fact that $e^2=\alpha\theta$, i.e. (2.14) in \cite{CYZ},
and the last equality follows from the fact that 
$$\omega(T+\alpha e_2)=e_1(\alpha)+2\alpha^2-\mbox{Im} A_{11},$$
i.e. (2.16) in \cite{CYZ}.  
Putting (\ref{5.5})-(\ref{5.7}) into (\ref{5.4}), we find 
\begin{equation}\label{5.8}
\begin{split}
h_{111}&=e_1(H)-2\alpha H+3 \mbox{Re} A^1_{\overline{1}},\\
h_{110}&=(T+\alpha e_2)(H)-3\alpha e_1(\alpha)-3\alpha^3-\frac{3}{2}\alpha W\\
&\hspace{4mm}+3\alpha \mbox{Im} A^1_{\overline{1}}+3\alpha \mbox{Im} A_{11}
+3
\mbox{Re}\Big(\frac{1}{6}W^{,1}+\frac{2i}{3}(A^{11})_{,1}\Big),\\
h_{100}&=(T+\alpha e_2)\left(e_1(\alpha)+\frac{1}{2}\alpha^2-\mbox{Im} A_{11}+\frac{1}{4}W\right)+\alpha H e_1(\alpha)+\alpha^3 H\\
&\hspace{4mm}+\frac{1}{2}\alpha HW-\alpha H \mbox{Im} A^1_{\overline{1}}-\alpha H \mbox{Im}A_{11}
-H\mbox{Re}\Big(\frac{1}{6}W^{,1}+\frac{2i}{3}(A^{11})_{,1}\Big).
\end{split}
\end{equation}
Substituting  (\ref{5.5}) and (\ref{5.8}) into (\ref{5.3}) yields
\begin{equation}\label{5.9}
\begin{split}
|H_{cr}|\mathfrak{f}
&=\left(e_1(\alpha)+\frac{1}{2}\alpha^2+\frac{1}{3}H^2-\mbox{Im} A_{11}+\frac{1}{4}W\right)\Big(e_1(H)-2\alpha H+3 \mbox{Re} A^1_{\overline{1}}\Big)\\
&\hspace{4mm}
+H(T+\alpha e_2)(H)-3\alpha He_1(\alpha)-3\alpha^3H-\frac{3}{2}\alpha HW\\
&\hspace{4mm}+3\alpha H \mbox{Im} A^1_{\overline{1}}+3\alpha H \mbox{Im} A_{11}
+3H
\mbox{Re}\Big(\frac{1}{6}W^{,1}+\frac{2i}{3}(A^{11})_{,1}\Big)\\
&\hspace{4mm}
+\frac{3}{2}(T+\alpha e_2)\left(e_1(\alpha)+\frac{1}{2}\alpha^2-\mbox{Im} A_{11}+\frac{1}{4}W\right)+\frac{3}{2}\alpha H e_1(\alpha)+\frac{3}{2}\alpha^3 H\\
&\hspace{4mm}+\frac{3}{4}\alpha HW-\frac{3}{2}\alpha H \mbox{Im} A^1_{\overline{1}}-\frac{3}{2}\alpha H \mbox{Im}A_{11}
-\frac{3}{2}H\mbox{Re}\Big(\frac{1}{6}W^{,1}+\frac{2i}{3}(A^{11})_{,1}\Big).
\end{split}
\end{equation}
On the other hand, it follows from (2.22) in \cite{CYZ} that 
$$h_{00}=(T+\alpha e_2)(\alpha)
+\mbox{Im}\Big(\frac{1}{6}W^{,1}+\frac{2i}{3}(A^{11})_{,1}\Big)-\alpha \mbox{Re} A^1_{\overline{1}}.$$
Combining this with (\ref{5.5}) yields 
\begin{equation}\label{5.10}
\begin{split}
&9h_{00}+6h_{11}h_{10}+\frac{2}{3}h_{11}^3\\
&=9(T+\alpha e_2)(\alpha)
+9\mbox{Im}\Big(\frac{1}{6}W^{,1}+\frac{2i}{3}(A^{11})_{,1}\Big)-9\alpha \mbox{Re} A^1_{\overline{1}}\\
&\hspace{4mm}+6He_1(\alpha)+3H\alpha^2-6H\mbox{Im} A_{11}+\frac{3}{2}HW+\frac{2}{3}H^3.
\end{split}
\end{equation}

To conclude, we have proved the following: 
\begin{theorem}\label{thm5.1}
Let $\Sigma$ be a smooth surface in a $3$-dimensional 
pseudohermitian manifold $(M,J,\theta)$. 
Suppose $\Sigma$ is nonsingular and $H_{cr}\neq 0$ on $\Sigma$. 
Then $\Sigma$ satisfies the Euler-Lagrange equation 
for the energy functional $E_1$
if and only if $\mathcal{E}_1=0$, 
where $\mathcal{E}_1$ is given as in 
\eqref{5.1}, 
and 
$H_{cr}$, $\mathfrak{f}$
and $9h_{00}+6h_{11}h_{10}+\frac{2}{3}h_{11}^3$ 
are given as in \eqref{5.2}, \eqref{5.9}
\eqref{5.10} respectively. 
\end{theorem}

\section{Euler-Lagrange equation for $E_2$ in the general case}\label{section6}

In this section, we are going to derive the Euler-Lagrange equation 
for the functional $E_2$ when the pseudohermitian $3$-manifold $(M,J,\theta)$
has constant Webster scalar curvature and 
its torsion is constant and purely imaginary. That is to say, we have 
\begin{equation}\label{7.1}
W\equiv c_1~~\mbox{ and }A_{11}=ic_2
\end{equation}
for some real constants $c_1$ and $c_2$. 
We will assume that (\ref{7.1}) holds
throughout this section.

Note that it follows from (2.13) and (2.14) in \cite{CYZ}  that 
\begin{equation}\label{7.2}
\omega(e_1)=H
\end{equation}
and 
\begin{equation}\label{7.3}
e^2=\alpha\theta
\end{equation}
on $\Sigma$. Here, $\omega$ is real $1$-form such that 
$\omega_1^1=i\omega$, where $\omega_1^1$ is the connection form
of $(M,J,\theta)$.

Writing 
\begin{equation}\label{7.4}
A_{11}=A_1^{\overline{1}}=a_1+ia_2~~
\mbox{ and }~~
A_{\overline{1}\,\overline{1}}=A_{\overline{1}}^1=\overline{A_{11}}=a_1-ia_2,
\end{equation}
for some real-valued functions $a_1$ and $a_2$. 
Then we have (c.f. (7.5) in \cite{CYZ})
\begin{equation}\label{7.6}
\begin{split}
de^1 &=-e^2\wedge\omega+\theta\wedge(a_1e^1-a_2e^2), \\
de^2 &=e^1\wedge\omega-\theta\wedge(a_1e^1+a_2e^2). 
\end{split}
\end{equation}
This together with the assumption (\ref{7.1}) implies that 
\begin{equation}\label{7.5}
\begin{split}
de^1 &=-e^2\wedge\omega-c_2\theta\wedge e^2, \\
de^2 &=e^1\wedge\omega-c_2\theta\wedge e^2. 
\end{split}
\end{equation}

It follows from (\ref{7.2}), (\ref{7.3}) and (\ref{7.5}) that
\begin{equation}\label{7.7}
\begin{split}
d(\alpha e^1)
&=d\alpha\wedge e^1+\alpha de^1\\
&=[(T\alpha)\theta+e_2(\alpha)e^2]\wedge e^1-\alpha e^2\wedge \omega -\alpha c_2\theta\wedge e^2\\
&=[(T+\alpha e_2)(\alpha)-\alpha^2H]\theta\wedge e^1
\end{split}
\end{equation}
on $\Sigma$. 
From (\ref{0.3})
and (\ref{7.1}), 
we compute 
\begin{equation}\label{7.0}
\begin{split}
dA_2 &=\Bigg[(T+\alpha e_2)(\alpha)+\frac{2}{3}\left(e_1(\alpha)+\frac{1}{2}\alpha^2-\mbox{Im} A_{11}+\frac{1}{4}W\right)H+\frac{2}{27}H^3\\
&\hspace{8mm}+\mbox{Im}\left(\frac{1}{6}W^{,1}+\frac{2i}{3}(A^{11})_{,1}\right)-\alpha\left(\mbox{Re} A^1_{\overline{1}}\right)\Bigg]\theta\wedge e^1\\
&=\Bigg[(T+\alpha e_2)(\alpha)+\frac{2}{3}\left(e_1(\alpha)+\frac{1}{2}\alpha^2-\mbox{Im} A_{11}+\frac{1}{4}W\right)H+\frac{2}{27}H^3\Bigg]\theta\wedge e^1\\
&=\left[\frac{2}{3}e_1(\alpha)+\frac{4}{3}\alpha^2-\frac{2}{3}\mbox{Im} A_{11}
+\frac{1}{6}W+\frac{2}{27}H^2\right]H\theta\wedge e^1+d(\alpha e^1).
\end{split}
\end{equation}
Hence, we are reduce to computing the Euler-Lagrange equation of 
$$\int_\Sigma \left[\frac{2}{3}e_1(\alpha)+\frac{4}{3}\alpha^2-\frac{2}{3}\mbox{Im} A_{11}
+\frac{1}{6}W+\frac{2}{27}H^2\right]H\theta\wedge e^1.$$

Let 
$$\Sigma_t=F_t(\Sigma)$$
be a family of immersions such that 
\begin{equation}\label{7.8}
\frac{d}{dt}F_t=X=fe_2+gT.
\end{equation}
Here, we let $e_1$ be the unit vector in $T\Sigma_t\cap\xi$
and $e_1=Je_1$. 
We assume that 
$f$ and $g$ are supported in a domain of $\Sigma$ away from the singular set of $\Sigma$. 
We have the following:

\begin{lem}\label{lem7.1}
Let $h:=f-\alpha g$ and $V:=T+\alpha e_2$. 
Then we have 
\begin{eqnarray}
\label{7.10}\omega(e_2)&=&h^{-1}e_1(h)+2\alpha,\\
\label{7.11}\omega(T)&=&e_1(\alpha)-\alpha h^{-1}e_1(h)-\mbox{\emph{Im}}A_{11},\\
\label{7.12}e_2(\alpha)&=&h^{-1}V(h),
\end{eqnarray}  
and 
\begin{equation}\label{7.13}
e_2(H)=2W+4e_1(\alpha)+H^2+4\alpha^2+h^{-1}e_1e_1(h)
+2\alpha h^{-1}e_1(h)-2\mbox{\emph{Im}}A_{11}.
\end{equation}
\end{lem}

To prove Lemma \ref{lem7.1}, we need the following formulas: 
\begin{equation}\label{7.9}
\begin{split}
[T,e_1] & =(\mbox{Im}A_{11}+\omega(T))e_2, \\
[T,e_2] & =-(\mbox{Im}A_{\overline{1}\,\overline{1}}+\omega(T))e_1,\\
d\omega&=-2We^1\wedge e^2,\\
[e_1,e_2]&=-2T-\omega(e_1)e_1-\omega(e_2)e_2.\\
-2W&=e_1\omega(e_2)-e_2\omega(e_1)+\omega(e_2)^2+\omega(e_1)^2+2\omega(T).
\end{split}
\end{equation}
Here, the first two formulas in (\ref{7.9}) follow from 
(A.7r) in \cite{CHMY} and the assumption (\ref{7.1}),
the third formula in (\ref{7.9}) follows from 
(A.5) in \cite{CHMY}
and the assumption (\ref{7.1}), 
the second last formula in (\ref{7.9}) follows from 
(7.3) in \cite{CYZ} (see also (A.6r) in \cite{CHMY}), 
and the last formula in (\ref{7.9}) follows from 
(7.8) in \cite{CYZ}. 
We would like to take this opportunity to point out that the second equation 
in (A.7r) of \cite{CHMY} is stated incorrectly; and the correct 
one should be
$$[e_2,T]=\big((\mbox{Im}A_{\overline{1}\,\overline{1}}+\omega(T)\big)e_1-(\mbox{Re}A_{\overline{1}\,\overline{1}})e_2.$$
We thank Prof. Jih-Hsin Cheng and Prof. Yongbing Zhang 
for communicating to us about this.

\begin{proof}[Proof of Lemma \ref{lem7.1}]
The proof of (\ref{7.10}) is the same as that of \cite{CYZ}. 
For the readers' convenience, we include it here. 
Suppose the surfaces $F_t(\Sigma)$ are the level sets of a defining 
function $t$ such that 
$$\frac{d}{dt}F_t=fe_2+gT,~~t(F_t(\Sigma))=t.$$
We observe that 
\begin{equation*}
(fe_2+gT)(t)=1~~\mbox{ and }~~(T+\alpha e_2)(t)=0,
\end{equation*}
and hence 
\begin{equation}\label{7.14}
T(t)=-\alpha e_2(t)~~\mbox{ and }~~
(f-\alpha g)e_2(t)=1.
\end{equation}
Then for $f-\alpha g\neq 0$, we have 
\begin{equation}\label{7.15}
e_2(t)=(f-\alpha g)^{-1}=h^{-1}.
\end{equation}
Note that 
\begin{equation}\label{7.16}
e_1(t)\equiv 0~~\mbox{ and }~~T=V-\alpha e_2. 
\end{equation}
This together with 
(\ref{7.14}), (\ref{7.15}) and the second last formula in (\ref{7.9})
implies that 
\begin{equation*}
\begin{split}
[e_1,e_2](t)
&=e_1e_2(t)-e_2e_1(t)=e_1e_2(t)=e_1(h^{-1})=-h^{-2}e_1(h)\\
&=-\omega(e_1)e_1(t)-\omega(e_2)e_2(t)-2T(t)\\
&=-\omega(e_2)h^{-1}+2\alpha h^{-1},
\end{split}
\end{equation*}
which gives (\ref{7.10}). 

Now applying the first formula of (\ref{7.9}) to $t$ yields 
\begin{equation*}
\begin{split}
&h^{-1}\big(\mbox{Im}A_{11}+\omega(T)\big)=\big(\mbox{Im}A_{11}+\omega(T)\big)e_2(t)\\
&=[T,e_1](t)=-e_1(T(t))=e_1\big(\alpha e_2(t)\big)
=e_2(t)e_1(\alpha)+\alpha e_1(e_2(t))\\
&=h^{-1}e_1(\alpha)+\alpha e_1(h^{-1})
=h^{-1}e_1(\alpha)-\alpha h^{-2} e_1(h),
\end{split}
\end{equation*}
where we have used (\ref{7.14})-(\ref{7.16}). 
Multiplying $h$ on both sides gives  (\ref{7.11}). 

Similarly, applying the second formula of (\ref{7.9}) to $t$ yields
\begin{equation*}
\begin{split}
0&=-\big(\mbox{Im}A_{\overline{1}\,\overline{1}}+\omega(T)\big)e_1(t)\\
&=[T,e_2](t)=T(e_2(t))-e_2(T(t))=T(h^{-1})+e_2\big(\alpha h^{-1}\big)\\
&=h^{-2}[-\alpha e_2(h)+he_2(\alpha)-T(h)],
\end{split}
\end{equation*}
where we have used (\ref{7.14})-(\ref{7.16}). This gives 
$$e_2(\alpha)=h^{-1}\big(\alpha e_2(h)+T(h)\big)
=h^{-1}(T+\alpha e_2)(h)=h^{-1}V(h).$$
This proves (\ref{7.12}). 

Finally, we compute 
\begin{equation*}
\begin{split}
-2W&=d\omega(e_1,e_2)=e_1\omega(e_2)-e_2\omega(e_1)+\omega(e_2)^2+\omega(e_1)^2+2\omega(T)\\
&=e_1(h^{-1}e_1(h)+2\alpha)-e_2(H)+(h^{-1}e_1(h)+2\alpha)^2+H^2\\
&\hspace{4mm}+2
(e_1(\alpha)-\alpha h^{-1}e_1(h)-\mbox{Im}A_{11}).
\end{split}
\end{equation*}
where the first two equalities follow from the third and the last formulas in (\ref{7.9}), 
the last equality follows from (\ref{7.2}), (\ref{7.10}) and (\ref{7.11}). 
Simplifying it gives 
(\ref{7.13}). This completes the proof of Lemma \ref{lem7.1}. 
\end{proof}

\begin{lem}\label{lem7.2}
On $\Sigma$, we have 
\begin{equation}\label{7.17}
e_1e_1(\alpha)
=V(H)-6\alpha e_1(\alpha)-\alpha H^2-4\alpha^3-2W\alpha+2\alpha \mbox{\emph{Im}}A_{11}.
\end{equation}
\end{lem}

The proof is similar to that in \cite[Section 3.2, (3.24) in Lemma 5]{CYZ}
 with the torsion term added. For the reader's convenience, we
include the proof in the Appendix.

\begin{lem}\label{lem7.3}
There holds
\begin{equation}
\label{7.21}\frac{d}{dt}[F_t^*(\theta\wedge e^1)]
=\big(-fH+V(g)\big)\theta\wedge e^1,
\end{equation}
\begin{equation}\label{7.22}
\begin{split}
\frac{dH}{dt}
&=e_1e_1(h)+2\alpha e_1(h)+4h\left(e_1(\alpha)+\alpha^2+\frac{1}{4}H^2+\frac{1}{2}W-\frac{1}{2}\mbox{\emph{Im}}A_{11}\right)\\
&\hspace{4mm}
+g\big[V(H)-2\alpha\mbox{\emph{Im}}A_{11}-h^{-1}e_1(h)\mbox{\emph{Im}}A_{11}+\mbox{\emph{Im}}A_{11}\big],
\end{split}
\end{equation}
\begin{equation}\label{7.23}
\frac{d\alpha}{dt}
=V(h)+gV(\alpha),
\end{equation}
and 
\begin{equation}\label{7.24}
\frac{d}{dt}e_1(\alpha)=e_1V(h)+ge_1V(\alpha)+2fV(\alpha)+fHe_1(\alpha).
\end{equation}
\end{lem}

The proof is similar to that in \cite[Section 3.2, Lemma 6]{CYZ}
 with the torsion term added. We
include the proof in the Appendix for the reader's convenience.

The following lemma could be found in \cite[Section 3.2, Lemma 7]{CYZ}. 

\begin{lem}\label{lem7.4}
Suppose either $f_1$ and $f_2$ has compact support in the nonsingular domain of $\Sigma$. 
Then we have 
\begin{equation*}
\int_\Sigma f_1 e_1(f_2)\theta\wedge e^1=-\int_\Sigma[e_1(f_1)+2\alpha f_1]f_2\theta\wedge e^1,
\end{equation*}
and 
\begin{equation*}
\int_\Sigma f_1 V(f_2)\theta\wedge e^1=-\int_\Sigma[V(f_1)-\alpha Hf_1]f_2\theta\wedge e^1.
\end{equation*}
\end{lem}

\begin{theorem}\label{thm7.1}
Assume that $(M,J,\theta)$ has constant Webster scalar curvature and 
its torsion is constant and purely imaginary in the sense of \eqref{7.1}. 
Let $F_t(\Sigma)$ be given by \eqref{7.8} such that 
\begin{equation}\label{condition}
e_1(h)+2\alpha h=h.
\end{equation}
Then we have 
\begin{equation}\label{7.33}
\frac{d}{dt}\int_{F_t(\Sigma)}dA_2
=\int_\Sigma\mathcal{E}_2 h\theta\wedge e^1,
\end{equation}
where 
\begin{align}\label{7.34}
\begin{split}
\mathcal{E}_2&=\frac{4}{9}\Bigg\{He_1e_1(H)+3e_1V(H)+e_1(H)^2+\frac{1}{3}H^4\\
&\hspace{12mm}+3e_1(\alpha)^2+12\alpha^2e_1(\alpha)+12\alpha^4\\
&\hspace{12mm}-\alpha He_1(H)+2H^2 e_1(\alpha)+5\alpha^2H^2\\
&\hspace{12mm}+\frac{3}{2} W\Big(e_1(\alpha)+\frac{2}{3}H^2+5\alpha^2+\frac{1}{2}W\Big)\\
&\hspace{12mm}+6\mbox{\emph{Im}}A_{11}\left[\frac{1}{2}\mbox{\emph{Im}}A_{11}-e_1(\alpha)-2\alpha^2-\frac{5}{8}W-\frac{1}{6}H^2\right]\Bigg\}.
\end{split}
\end{align}
\end{theorem}

With Lemmas \ref{lem7.1}-\ref{lem7.4},
the proof of Theorem \ref{thm7.1} is similar to that in \cite[Section 3.2, Theorem 8]{CYZ}
 with the torsion term added. In the Appendix, we
include the proof  for the reader's convenience.

\section{Rossi sphere}\label{section8}

In this section, we study the invariant area functionals 
in the Rossi sphere. 
We first include some computations in the Rossi sphere. 
They could be in fact found in \cite{CHP,CH,T}. 
The standard CR sphere
$$S^3=\big\{(z_1,z_2)\in\mathbb{C}^2: |z_1|^2+|z_2|^2=1\big\}$$
has the CR structure spanned by
\begin{equation}\label{10.1}
Z_1=\overline{z}_2\frac{\partial}{\partial z_1}-\overline{z}_1\frac{\partial}{\partial z_2}.
\end{equation}
A canonical contact form $\theta$ on $S^3$ is given by
\begin{equation}\label{10.2}
\begin{split}
\theta&=\frac{i}{2}(z_1d\overline{z}_1+z_2d\overline{z}_2-\overline{z}_1dz_1-\overline{z}_2dz_2)\Big|_{S^3}\\
&=\rho_1^2d\varphi_1+\rho_2^2d\varphi_2.
\end{split}
\end{equation}
where $z_1=\rho_1 e^{i\varphi_1}$ and $z_2=\rho_2 e^{i\varphi_2}$ are the polar coordinates. 
The almost complex structure $J$ maps 
$$J(Z_1)=iZ_1~~\mbox{ and }~~J(Z_{\overline{1}})=-iZ_{\overline{1}}$$
where 
$$Z_{\overline{1}}=\overline{Z_1}=z_2\frac{\partial}{\partial \overline{z}_1}-z_1\frac{\partial}{\partial \overline{z}_2}.$$
The Reeb vector field with respect to $\theta$ is
\begin{equation}\label{10.3}
\begin{split}
T&=i\left(z_1\frac{\partial}{\partial z_1}+z_2\frac{\partial}{\partial z_2}-\overline{z}_1\frac{\partial}{\partial\overline{z}_1}-\overline{z}_2\frac{\partial}{\partial\overline{z}_2}\right)\\
&=\frac{\partial}{\partial\varphi_1}+\frac{\partial}{\partial\varphi_2}.
\end{split}
\end{equation}
The admissible coframe corresponding to $(T,Z_1,Z_{\overline{1}})$ is given by
$$(\theta,\theta^1=(z_2dz_1-z_1dz_2)|_{S^3},\theta^{\overline{1}}=\overline{\theta^1}).$$
For this coframe, we have
$$d\theta=i\theta^1\wedge \theta^{\overline{1}}~~\mbox{ and }~~d\theta^1=2i\theta\wedge\theta^1.$$
This implies the torsion, the connection $1$-form, and the Webster curvature 
of the standard CR sphere are respectively given by 
$$A_{11}=0,~~\omega_1^1=-2i\theta,~~W=2.$$

Now, for a real number $|t|<1$, the Rossi sphere $S_t^3$ is
the unit sphere $S^3$ equipped with the CR structure spanned by
$$Z_1(t)=\frac{Z_1+tZ_{\overline{1}}}{\sqrt{1-t^2}}.$$
The almost complex structure $J_t$ maps
\begin{equation}\label{10.4}
J_t\big(Z_1(t)\big)=iZ_1(t)~~\mbox{ and }~~J_t\big(Z_{\overline{1}}(t)\big)=-iZ_{\overline{1}}(t).
\end{equation}
The contact form is
$$\theta(t)=\theta$$
in (\ref{10.2}), 
and the Reeb vector field with respect to $\theta(t)$ is $T$ as in (\ref{10.3}). 
The admissible coframe with respect to $(T,Z_1(t),Z_{\overline{1}}(t))$ is given by
$$\Big(\theta(t)=\theta,\theta^1(t)=\frac{1}{\sqrt{1-t^2}}(\theta^1-t\theta^{\overline{1}}),\theta^{\overline{1}}(t)=\overline{\theta^1(t)}\Big).$$
For this coframe, we have
\begin{equation}\label{10.5}
\begin{split}
d\theta(t)&=i\theta^1(t)\wedge\theta^{\overline{1}}(t),\\
d\theta^1(t)&=\frac{2i(1+t^2)}{1-t^2}\theta(t)\wedge \theta^1(t)+\frac{4it}{1-t^2}\theta(t)\wedge \theta^{\overline{1}}(t).
\end{split}
\end{equation}
These imply that  the torsion, the connection $1$-form, and the Webster curvature 
of the Rossi sphere are respectively given by 
\begin{equation}\label{10.6}
A_{11}(t)=\overline{A_{\overline{1}\,\overline{1}}(t)}=\overline{\left(\frac{4it}{1-t^2}\right)}=-\frac{4it}{1-t^2},~~\omega(t)_1^1=-\frac{2i(1+t^2)}{1-t^2}\theta(t),~~W=\frac{2(1+t^2)}{1-t^2}.
\end{equation}
To reiterate, it follows from (\ref{10.6}) that the torsion on the Rossi sphere does not vanish. 
Moreover, it follows from 
(\ref{10.6}) that, on the Rossi sphere, the Webster curvature 
is constant and the torsion is constant and purely imaginary in the sense of \eqref{7.1}. 
In particular, Theorem \ref{thm7.1} is applicable to the Rossi sphere.

\subsection{The functional $E_1$ on the Rossi sphere.}

In the Rossi sphere $S^3_t$, we consider 
a closed surface $\Sigma_{[c]}$ defined by 
$\rho_1=c$, 
where 
$$z_1=\rho_1 e^{i\varphi_1},~~ z_2=\rho_2 e^{i\varphi_2}~~\mbox{ and }~~
\rho_1^2+\rho_2^2=1.$$
Since the Reeb vector field $T$ is given by (\ref{10.3}), 
we see that $T\in T\Sigma_{[c]}$, which implies that 
\begin{equation}\label{10.6A}
\alpha\equiv 0
\end{equation}
on $\Sigma_{[c]}$. 
Note that  
\begin{equation}\label{10.7}
\rho_j\frac{\partial}{\partial\rho_j}=2\mbox{Re}\left(z_j\frac{\partial}{\partial z_j}\right)~~\mbox{ and }~~
\frac{\partial}{\partial\varphi_j}=2\mbox{Re}\left(iz_j\frac{\partial}{\partial z_j}\right)
\end{equation}
for $j=1,2$. 
By definition, $e_1\in T\Sigma_{[c]}\cap \ker\theta$. 
For $\Sigma_{[c]}$, we have
\begin{equation}\label{10.8}
e_1=c\left(-\frac{\rho_2}{\rho_1}\frac{\partial}{\partial\varphi_1}+\frac{\rho_1}{\rho_2}\frac{\partial}{\partial\varphi_2}\right)
\end{equation}
for some real constant $c$. 
To find $c$, we write 
$$
-\frac{\rho_2}{\rho_1}\frac{\partial}{\partial\varphi_1}+\frac{\rho_1}{\rho_2}\frac{\partial}{\partial\varphi_2}=aZ_1(t)+bZ_{\overline{1}}(t).$$
Since the left hand side is real, we must have $b=\overline{a}$, i.e. 
\begin{equation}\label{10.13}
-\frac{\rho_2}{\rho_1}\frac{\partial}{\partial\varphi_1}+\frac{\rho_1}{\rho_2}\frac{\partial}{\partial\varphi_2}=aZ_1(t)+\overline{a}Z_{\overline{1}}(t),
\end{equation}
which is equivalent to
\begin{equation}\label{10.9}
\begin{split}
&-\frac{\rho_2}{\rho_1}\left(iz_1\frac{\partial}{\partial z_1}-i\overline{z}_1\frac{\partial}{\partial \overline{z}_1}\right)+\frac{\rho_1}{\rho_2}\left(iz_2\frac{\partial}{\partial z_2}-i\overline{z}_2\frac{\partial}{\partial \overline{z}_2}\right)\\
&=-\frac{\rho_2}{\rho_1}\frac{\partial}{\partial\varphi_1}+\frac{\rho_1}{\rho_2}\frac{\partial}{\partial\varphi_2}=aZ_1(t)+\overline{a}Z_{\overline{1}}(t)\\
&=\frac{a+t\overline{a}}{\sqrt{1-t^2}}\overline{z}_2\frac{\partial}{\partial z_1}
+\frac{\overline{a}+ta}{\sqrt{1-t^2}}z_2\frac{\partial}{\partial \overline{z}_1}
-\frac{a+t\overline{a}}{\sqrt{1-t^2}}\overline{z}_1\frac{\partial}{\partial z_2}
-\frac{\overline{a}+ta}{\sqrt{1-t^2}}z_1\frac{\partial}{\partial \overline{z}_2}
\end{split}
\end{equation}
by (\ref{10.7}) and (\ref{10.8}). 
Equating the coefficients of $\frac{\partial}{\partial z_j}$ and $\frac{\partial}{\partial \overline{z}_j}$ 
in (\ref{10.9}), we find 
\begin{equation}\label{10.10}
\begin{split}
-\frac{\rho_2}{\rho_1}iz_1&=\frac{a+t\overline{a}}{\sqrt{1-t^2}}\overline{z}_2,~~
\frac{\rho_2}{\rho_1}i\overline{z}_1=\frac{\overline{a}+ta}{\sqrt{1-t^2}}z_2,\\
\frac{\rho_1}{\rho_2}iz_2&=-\frac{a+t\overline{a}}{\sqrt{1-t^2}}\overline{z}_1,~~
-\frac{\rho_1}{\rho_2}i\overline{z}_2=-\frac{\overline{a}+ta}{\sqrt{1-t^2}}z_1.
\end{split}
\end{equation}
Writing (\ref{10.10}) in terms of polar coordinates $z_j=\rho_j e^{i\varphi_j}$ 
for $j=1,2$, 
we see that all equations in (\ref{10.10})
are equivalent to 
\begin{equation}\label{10.11}
\frac{a+t\overline{a}}{\sqrt{1-t^2}}=-ie^{i(\varphi_1+\varphi_2)}. 
\end{equation}
Solving (\ref{10.11}) yields 
\begin{equation}\label{10.12}
a=-\frac{i\big(e^{i(\varphi_1+\varphi_2)}+te^{-i(\varphi_1+\varphi_2)}\big)}{\sqrt{1-t^2}}.
\end{equation}
By definition, $e_1$ is a unit vector with respect to $\frac{1}{2}d\theta(t)\big(\cdot, J_t\cdot\big)$. 
This implies 
\begin{equation*}
\begin{split}
1&=\frac{1}{2}d\theta(t)\big(e_1,J_te_1\big)=\frac{c^2}{2}d\theta(t)\left(
-\frac{\rho_2}{\rho_1}\frac{\partial}{\partial\varphi_1}+\frac{\rho_1}{\rho_2}\frac{\partial}{\partial\varphi_2},
J_t\Big(-\frac{\rho_2}{\rho_1}\frac{\partial}{\partial\varphi_1}+\frac{\rho_1}{\rho_2}\frac{\partial}{\partial\varphi_2}\Big)\right)\\
&=\frac{c^2}{2}d\theta(t)\big(aZ_1(t)+\overline{a}Z_{\overline{1}}(t), 
J_t(aZ_1(t)+\overline{a}Z_{\overline{1}}(t))\big)\\
&=\frac{c^2}{2}d\theta(t)\big(aZ_1(t)+\overline{a}Z_{\overline{1}}(t), 
iaZ_1(t)-i\overline{a}Z_{\overline{1}}(t)\big)\\
&=-c^2|a|^2id\theta(t)\big(Z_1(t),Z_{\overline{1}}(t)\big)\\
&=c^2|a|^2\big(\theta^1(t)\wedge\theta^{\overline{1}}(t)\big)\big(Z_1(t),Z_{\overline{1}}(t)\big)=c^2|a|^2,
\end{split}
\end{equation*}
where the second equality follows from (\ref{10.8}), 
the third equality follows from (\ref{10.13}), 
the fourth equality follows from (\ref{10.4}), 
and the second last equality follows from (\ref{10.5}). 
Thus, we have 
\begin{equation}\label{10.14A}
c=\frac{1}{|a|}. 
\end{equation}
We compute 
\begin{equation}\label{10.14}
\begin{split}
e_2&=J_t(e_1)=\frac{1}{|a|}\left(aJ_t\big(Z_1(t)\big) +\overline{a}J_t\big(Z_{\overline{1}}(t)\big)\right)=\frac{1}{|a|}\left(aiZ_1(t)-\overline{a}iZ_{\overline{1}}(t)\right)\\
&=\frac{1}{|a|}\left(\frac{i(a-t\overline{a})}{\sqrt{1-t^2}}\overline{z}_2\frac{\partial}{\partial z_1}-\frac{i(\overline{a}-ta)}{\sqrt{1-t^2}}z_2\frac{\partial}{\partial \overline{z}_1}
-\frac{i(a-t\overline{a})}{\sqrt{1-t^2}}\overline{z}_1\frac{\partial}{\partial z_2}
+\frac{i(\overline{a}-ta)}{\sqrt{1-t^2}}z_1\frac{\partial}{\partial \overline{z}_2}\right)\\
&=\frac{\rho_2}{\rho_1}\frac{1}{|a|}\left(\frac{a-t\overline{a}}{a+t\overline{a}}\right)z_1\frac{\partial}{\partial z_1}
+\frac{\rho_2}{\rho_1}\frac{1}{|a|}\left(\frac{\overline{a}-ta}{\overline{a}+ta}\right)\overline{z}_1\frac{\partial}{\partial 
\overline{z}_1}\\
&\hspace{4mm}
-\frac{\rho_1}{\rho_2}\frac{1}{|a|}\left(\frac{a-t\overline{a}}{a+t\overline{a}}\right)z_2\frac{\partial}{\partial z_2}
-\frac{\rho_1}{\rho_2}\frac{1}{|a|}\left(\frac{\overline{a}-ta}{\overline{a}+ta}\right)\overline{z}_2\frac{\partial}{\partial 
\overline{z}_2}
\end{split}
\end{equation}
where the second equality follows from  (\ref{10.8}) and (\ref{10.14A}), 
the third equality follows from (\ref{10.4}), 
and the  last equality follows from (\ref{10.10}). 
It follows from (\ref{10.11}) that 
$|a+t\overline{a}|^2=1-t^2$, 
which gives 
$$\frac{a-t\overline{a}}{a+t\overline{a}}=|a|^2-\frac{t(\overline{a}^2-a^2)}{1-t^2}~~\mbox{ and }~~
\frac{\overline{a}-ta}{\overline{a}+ta}=|a|^2+\frac{t(\overline{a}^2-a^2)}{1-t^2}.$$
This together with (\ref{10.7}), we can rewrite  (\ref{10.14}) as
\begin{equation}\label{10.15}
\begin{split}
e_2=|a|\left(\rho_2\frac{\partial}{\partial\rho_1}-\rho_1\frac{\partial}{\partial\rho_2}\right)
+i\frac{t(\overline{a}^2-a^2)}{(1-t^2)|a|}
\left(\frac{\rho_2}{\rho_1}\frac{\partial}{\partial\varphi_1}-\frac{\rho_1}{\rho_2}\frac{\partial}{\partial\varphi_2}\right).
\end{split}
\end{equation}

It follows that 
\begin{equation}\label{10.16}
\begin{split}
e^1&=-|a|\rho_1\rho_2(d\varphi_1-d\varphi_2)+\frac{it(\overline{a}^2-a^2)}{(1-t^2)|a|}(\rho_2d\rho_1-\rho_1d\rho_2),\\
e^2&=\frac{1}{|a|}(\rho_2d\rho_1-\rho_1d\rho_2).
\end{split}
\end{equation}
To check (\ref{10.16}), we can see from (\ref{10.3}), (\ref{10.8}), (\ref{10.15}) and  (\ref{10.16}) that 
$$e^j(e_k)=\delta_{jk}~~\mbox{ and }~~e^j(T)=0$$
for $j,k=1,2$. 
From (\ref{10.2}) and (\ref{10.16}), we find 
\begin{equation}\label{10.17}
\begin{split}
\theta\wedge e^1
&=|a|\rho_1\rho_2d\varphi_1\wedge d\varphi_2
+\frac{it(\overline{a}^2-a^2)}{(1-t^2)|a|}\Big[\rho_1^2(\rho_2d\varphi_1\wedge d\rho_1-\rho_1 d\varphi_1\wedge d\rho_2)\\
&\hspace{5cm}
+\rho_2^2(\rho_2 d\varphi_2\wedge d\rho_1-\rho_1d\varphi_2\wedge d\rho_2)\Big],\\
\theta\wedge e^1\wedge e^2
&=\rho_1 d\rho_1\wedge d\varphi_1\wedge d\varphi_2, 
\end{split}
\end{equation}
where we have used   
\begin{equation}\label{10.20}
\rho_1^2+\rho_2^2=1~~\mbox{ and }~~\rho_1d\rho_1+\rho_2 d\rho_2=0. 
\end{equation}
It follows from 
(\ref{10.12}) that 
\begin{equation*}
da=-i\overline{a}(d\varphi_1+d\varphi_2)~~\mbox{ and }~~
d\overline{a}=ia(d\varphi_1+d\varphi_2),
\end{equation*}
which gives 
\begin{equation}\label{10.19}
\begin{split}
d(|a|)&=-\frac{i(\overline{a}^2-a^2)}{2|a|}(d\varphi_1+d\varphi_2),\\
d\left(\frac{it(\overline{a}^2-a^2)}{(1-t^2)|a|}\right)
&=-\frac{t}{1-t^2}\left(4|a|+\frac{(\overline{a}^2-a^2)^2}{2|a|^3}\right)(d\varphi_1+d\varphi_2). 
\end{split}
\end{equation}
Taking exterior derivative of the first equation in (\ref{10.17})
and
using  (\ref{10.20}) and (\ref{10.19}), we find 
\begin{equation}\label{10.18}
\begin{split}
d(\theta\wedge e^1)
&=\frac{\rho_2^2-\rho_1^2}{\rho_2}|a|d\rho_1\wedge d\varphi_1\wedge d\varphi_2\\
&\hspace{4mm}
-\frac{t}{1-t^2}\left(4|a|+\frac{(\overline{a}^2-a^2)^2}{2|a|^3}\right)\frac{\rho_2^2-\rho_1^2}{\rho_2}d\rho_1\wedge d\varphi_1\wedge d\varphi_2.
\end{split}
\end{equation}
Since $d(\theta\wedge e^1)=-H\theta\wedge e^1\wedge e_2$, 
we deduce from (\ref{10.17}) and (\ref{10.18}) that 
\begin{equation}\label{10.21}
H=-\frac{\rho_2^2-\rho_1^2}{\rho_2\rho_1}\left[|a|-\frac{t}{1-t^2}\left(4|a|+\frac{(\overline{a}^2-a^2)^2}{2|a|^3}\right)\right].
\end{equation}

In the Rossi sphere $S_t^3$, 
it follows from (\ref{5.2}), (\ref{10.6}),
(\ref{10.6A}) that 
\begin{equation}\label{10.22}
\begin{split} 
H_{cr}&=-\mbox{Im} A_{11}+\frac{1}{4}W+\frac{1}{6}H^2=\frac{1+8t+t^2}{2(1-t^2)}+\frac{1}{6}H^2.
\end{split}
\end{equation}
In particular,  the Clifford torus $\Sigma_{[\frac{\sqrt{2}}{2}]}$
in $S^3_t$ with  $c=\rho_1=\rho_2=\frac{\sqrt{2}}{2}$
satisfying $H=0$ by (\ref{10.21}). 
Therefore, it follows from (\ref{10.22}) that 
$H_{cr}=\frac{1+8t+t^2}{2(1-t^2)}$. 
If $t_0=-4+\sqrt{15}$, we have
$0<t_0<1$ and $1+8t_0+t_0^2=0$, 
and hence, 
$H_{cr}=0$. 
From this, we conclude the following: 

\begin{lem}\label{lem5.1}
The Clifford torus
 $\Sigma_{[\frac{\sqrt{2}}{2}]}$ 
in the Rossi sphere $S^3_{t_0}$, where 
$t_0=-4+\sqrt{15}$, 
is 
 a minimizer for the energy $E_1$ 
 with zero energy.    
\end{lem}

As pointed out in the introduction, 
this is a sharp contrast 
to the situation in the CR sphere; for it was proved in \cite{CHY}
that if $T^2$ is a torus in the CR sphere 
without any singular points, then it cannot have $E_1=0$. 
In fact, we  have the following: 

\begin{lem}\label{lem5.3}
Suppose that 
$S^3_t$ is the Rossi sphere with $-4+\sqrt{15}<t<1$. 
If $T^2\subset S^3_t$ is a torus  without singular points, 
then it cannot have $E_1=0.$
\end{lem}
\begin{proof}
If $E_1=0$, then it follows from (\ref{5.2}) and (\ref{10.6}) that 
\begin{equation*}
\begin{split}
e_1(\alpha)&=-\frac{1}{2}\alpha^2+\mbox{Im} A_{11}-\frac{1}{4}W-\frac{1}{6}H^2\\
&=-\frac{1}{2}\alpha^2-\frac{1+8t+t^2}{2(1-t^2)}-\frac{1}{6}H^2\\
&\leq -\frac{1+8t+t^2}{2(1-t^2)}
\end{split}
\end{equation*}
which is negative by assumption on the values of $t$. Therefore, the characteristic curve of $e_1$ is open 
and along the curve $\alpha$ goes to infinity, which contradicts to $\alpha$ being 
bounded on a nonsingular compact surface.   
\end{proof}

Lemma \ref{lem5.1} asserts that he Clifford torus
 $\Sigma_{[\frac{\sqrt{2}}{2}]}$  is a critical point for $E_1$ 
 with zero energy
 in the Rossi sphere $S^3_{t}$, when $t=-4+\sqrt{15}$. 
 For the other values of $t$, we have the following: 

\begin{lem}\label{lem5.2}
Suppose that $t\neq -4+\sqrt{15}$. 
Then the Clifford torus
 $\Sigma_{[\frac{\sqrt{2}}{2}]}$ 
in the Rossi sphere $S^3_{t}$
is a critical point for the energy $E_1$ with nonzero energy.
\end{lem}
\begin{proof}
It follows from  (\ref{5.2}), (\ref{5.9}),
(\ref{5.10}), (\ref{10.6}) and (\ref{10.6A})
that 
\begin{equation}\label{10.28}
\begin{split}
H_{cr}&=-\mbox{Im} A_{11}+\frac{1}{4}W+\frac{1}{6}H^2=\frac{1+8t+t^2}{2(1-t^2)}+\frac{1}{6}H^2,\\
|H_{cr}|\mathfrak{f}&=\left(\frac{1}{3}H^2-\mbox{Im} A_{11}+\frac{1}{4}W\right)e_1(H)+HT(H)\\
&=\left(\frac{1}{3}H^2+\frac{1+8t+t^2}{2(1-t^2)}\right)e_1(H)+HT(H),\\
9h_{00}+6h_{11}h_{10}+\frac{2}{3}h_{11}^3&=-6H\mbox{Im} A_{11}+\frac{3}{2}HW+\frac{2}{3}H^3
=\frac{6(1+8t+t^2)}{2(1-t^2)}H+\frac{2}{3}H^3.
\end{split}
\end{equation}
Note that it follows from (\ref{10.21}) that 
\begin{equation}\label{10.31}
H=0~~\mbox{ whenever }\rho_1=\rho_2=\frac{\sqrt{2}}{2}.
\end{equation}
On the other hand, from (\ref{10.3}) and (\ref{10.8}), 
we see that $e_1$ and $T$ only involve derivatives
with respect to $\varphi_1$ and $\varphi_2$. 
In particular, it follows from (\ref{10.21}) that 
\begin{equation*}
\begin{split}
e_1(H)&=-\frac{\rho_2^2-\rho_1^2}{\rho_2\rho_1}e_1\left[|a|-\frac{t}{1-t^2}\left(4|a|+\frac{(\overline{a}^2-a^2)^2}{2|a|^3}\right)\right],\\
T(H)&=-\frac{\rho_2^2-\rho_1^2}{\rho_2\rho_1}T\left[|a|-\frac{t}{1-t^2}\left(4|a|+\frac{(\overline{a}^2-a^2)^2}{2|a|^3}\right)\right],
\end{split}
\end{equation*}
which implies that 
\begin{equation}\label{10.30}
\begin{split}
e_1(H)=0~~\mbox{ and }~~
T(H)=0~~\mbox{ whenever }\rho_1=\rho_2=\frac{\sqrt{2}}{2}. 
\end{split}
\end{equation}
Similarly, we can see that 
\begin{equation}\label{10.30A}
X_1X_2\cdots X_n(H)=0~~\mbox{ whenever }\rho_1=\rho_2=\frac{\sqrt{2}}{2},
\end{equation}
as long as the vector fields $X_j=e_1$ or $T$ for all $1\leq j\leq n$. 
Combining 
(\ref{5.1}) and (\ref{10.31})-(\ref{10.30A}), we conclude that 
$\mathcal{E}_1=0$ 
whenever $\rho_1=\rho_2=\frac{\sqrt{2}}{2}$. 
Hence, it follows from Theorem \ref{thm5.1} that the Clifford torus
 $\Sigma_{[\frac{\sqrt{2}}{2}]}$ 
in the Rossi sphere $S^3_{t}$
is a critical point for the functional $E_1$. 
On the other hand, 
it follows from (\ref{10.17}), (\ref{10.28}) and (\ref{10.31})
that 
\begin{equation*}
\begin{split}
E(\Sigma_{[\frac{\sqrt{2}}{2}]})
&=\int_{\Sigma_{[\frac{\sqrt{2}}{2}]}}dA_1=\int_{\Sigma_{[\frac{\sqrt{2}}{2}]}}|H_{cr}|^{\frac{3}{2}}\theta\wedge e^1=\int_{\Sigma_{[\frac{\sqrt{2}}{2}]}}\left|\frac{1+8t+t^2}{2(1-t^2)}\right|^{\frac{3}{2}}\rho_1\rho_2
\end{split}
\end{equation*}
which is positive by the assumption on the values of $t$. 
This proves the assertion. 
\end{proof}

\subsection{The functional $E_2$ on the Rossi sphere.}

We still consider 
a closed surface $\Sigma_{[c]}$ in the Rossi sphere $S^3_t$ defined by 
$\rho_1=c$, 
where $z_j=\rho_j e^{i\varphi_j}$ for $j=1,2$. 
It follows from (\ref{10.6}), (\ref{10.6A}), (\ref{10.17})
and (\ref{10.21})
that 
\begin{equation}\label{10.27}
\begin{split}
E_2(\Sigma_{[c]})&=\int_{\Sigma_{[c]}}dA_2=\int_{\Sigma_{[c]}}\left[\frac{2}{3}\left(-\mbox{Im} A_{11}+\frac{1}{4}W\right)H+\frac{2}{27}H^3\right]\theta\wedge e^1\\
&=\int_{\Sigma_{[c]}}\left(\frac{1+8t+t^2}{3(1-t^2)}H+\frac{2}{27}H^3\right)\rho_1\rho_2d\varphi_1\wedge d\varphi_2\\
&=(\rho_1^2-\rho_2^2)\int_{\Sigma_{[c]}}\left(\frac{1+8t+t^2}{3(1-t^2)}+\frac{2}{27}H^2\right)\\
&\hspace{2cm}\cdot
\left[|a|-\frac{t}{1-t^2}\left(4|a|+\frac{(\overline{a}^2-a^2)^2}{2|a|^3}\right)\right]d\varphi_1\wedge d\varphi_2,
\end{split}
\end{equation}
where we have used the fact that $\rho_1=c$
and $\rho_2=\sqrt{1-c^2}$ on $\Sigma_{[c]}$. 
It follows from (\ref{10.12}) that 
\begin{equation}\label{10.23}
|a|^2=\frac{1+t^2+2t\cos\big(2(\varphi_1+\varphi_2)\big)}{1-t^2}~~\mbox{ and }
~~\overline{a}^2-a^2=2i\sin\big(2(\varphi_1+\varphi_2)\big).
\end{equation}
Since $|t|<1$, we can deduce from (\ref{10.23}) that 
\begin{equation}\label{10.24}
\frac{1+|t|}{1-|t|}\geq|a|^2\geq\frac{1-|t|}{1+|t|}. 
\end{equation}
From (\ref{10.23}), we find 
\begin{equation*}
4|a|+\frac{(\overline{a}^2-a^2)^2}{2|a|^3}
=\frac{8|a|^4+(\overline{a}^2-a^2)^2}{2|a|^3}=\frac{8|a|^4-4\sin^2\big(2(\varphi_1+\varphi_2))^2}{2|a|^3}.
\end{equation*}
This together with (\ref{10.24}) implies that
\begin{equation}\label{10.25}
 4|a|\geq 4|a|+\frac{(\overline{a}^2-a^2)^2}{2|a|^3}
  \geq\frac{4|a|^4-2}{|a|^3}\geq \frac{2(t^2-5|t|+1)}{|a|^3(1+|t|)^2}.
\end{equation}
From (\ref{10.24}) and (\ref{10.25}), we have
\begin{equation*}
\begin{split}
&|a|-\frac{t}{1-t^2}\left(4|a|+\frac{(\overline{a}^2-a^2)^2}{2|a|^3}\right)
\geq |a|- \frac{2t(t^2-5|t|+1)}{|a|^3(1+|t|)^2(1-t^2)}\\
&=\frac{|a|^4(1+|t|)^2(1-t^2)-2t(t^2-5|t|+1)}{|a|^3(1+|t|)^2(1-t^2)}\\
&\geq\frac{(1-|t|)^2(1-t^2)-2t(t^2-5|t|+1)}{|a|^3(1+|t|)^2(1-t^2)}~~\mbox{ whenever }t\leq 0.
\end{split}
\end{equation*}
Hence, we can conclude that 
\begin{equation}\label{10.26}
\begin{split}
&|a|-\frac{t}{1-t^2}\left(4|a|+\frac{(\overline{a}^2-a^2)^2}{2|a|^3}\right)>0\\
&\hspace{8mm}\mbox{ whenever } (1+t)^3(1-t)>2t(t^2+5t+1)\mbox{ and }t\leq 0. 
\end{split}
\end{equation}
Similarly, from (\ref{10.24}) and (\ref{10.25}), we have
\begin{equation*}
\begin{split}
&|a|-\frac{t}{1-t^2}\left(4|a|+\frac{(\overline{a}^2-a^2)^2}{2|a|^3}\right)\\
&\geq |a|-\frac{4t}{1-t^2}|a|
=\frac{1-4t-t^2}{1-t^2}|a|~~\mbox{ whenever }t\geq 0.
\end{split}
\end{equation*}
Hence, we can conclude that 
\begin{equation}\label{10.27}
|a|-\frac{t}{1-t^2}\left(4|a|+\frac{(\overline{a}^2-a^2)^2}{2|a|^3}\right)>0\mbox{ whenever } 1-4t-t^2>0\mbox{ and }t\geq 0. 
\end{equation}
Under the conditions in (\ref{10.26}) and (\ref{10.27}), 
it follows from (\ref{10.21}) that 
$H\to +\infty$ ($-\infty$ respectively) as $\rho_1\to 1$ ($\rho_1\to 0$ respectively). 
Combining this fact
with (\ref{10.26}) and (\ref{10.27}) , 
we conclude that 
$E_2(\Sigma_{[c]})$ tends to $+\infty$ ($-\infty$ respectively) as $\rho_1\to 1$ ($\rho_1\to 0$ respectively). 
From this, we prove the following: 

\begin{lem}
The functional $E_2$ is unbounded from below and above in the Rossi sphere 
$S^3_t$, if one of the following conditions holds:\\
(i) $(1+t)^3(1-t)>2t(t^2+5t+1)$ and $t\leq 0$,\\
(ii) $ 1-4t-t^2>0$ and $t\geq 0$.
\end{lem}

As mentioned before, it follows
from (\ref{10.6}) that the Rossi sphere $S^3_t$ has constant Webster scalar curvature and 
its torsion is constant and purely imaginary.
In particular, Theorem \ref{thm7.1} is applicable to 
the Rossi sphere $S^3_t$.

For the Clifford torus $\Sigma_{[\frac{\sqrt{2}}{2}]}$
in $S^3_t$, we have 
\begin{equation}\label{10.32}
\alpha\equiv 0~~\mbox{ and }~~H\equiv 0
\end{equation}
by (\ref{10.6A}) and (\ref{10.31}).
Substituting (\ref{10.6}) and (\ref{10.32}) into 
(\ref{7.34}), we find
\begin{equation*}
\begin{split}
\frac{9}{4}\mathcal{E}_2&=\frac{3}{4}W^2
+3(\mbox{Im}A_{11})^2-\frac{15}{4}W\mbox{Im}A_{11}\\
&=\frac{3}{(1-t^2)^2}(t^4+10t^3+18t^2+10t+1)\\
&=\frac{3}{(1-t^2)^2}(t+1)^2(t^2+8t+1)
=\frac{3}{(1-t)^2}(t^2+8t+1),
\end{split}
\end{equation*}
where we have used the assumption that $|t|<1$. 
In particular, since $|t|<1$, 
we see that $\mathcal{E}_2=0$ 
if and only if $t=-4+\sqrt{15}$. 
Therefore,
we can conclude from Theorem \ref{thm7.1} 
the following: 

\begin{lem}\label{lem5.4}
The Clifford torus
 $\Sigma_{[\frac{\sqrt{2}}{2}]}$ 
in the Rossi sphere $S^3_{t}$
is a critical point for the energy $E_2$ 
if and only if $t=-4+\sqrt{15}$.
\end{lem}

In particular,  according to Lemma \ref{lem5.1} and Lemma \ref{lem5.4}, 
the Clifford torus
 $\Sigma_{[\frac{\sqrt{2}}{2}]}$ in the Rossi sphere 
 $S^3_{t_0}$, where $t_0=-4+\sqrt{15}$, 
 is a critical point for both the area functionals $E_1$ and $E_2$. 
Moreover, when $t\neq -4+\sqrt{15}$, 
according to Lemma \ref{lem5.2} and Lemma \ref{lem5.4}, 
the Clifford torus $\Sigma_{[\frac{\sqrt{2}}{2}]}$ in the Rossi sphere $S^3_t$ 
 is a critical point for $E_1$, but not a critical point for $E_2$.

\section{Second variation of $E_1$ for the Clifford torus in the Rossi sphere}\label{section10}

As in Section \ref{section6}, we assume that the pseudohermitian $3$-manifold $(M,J,\theta)$
has constant Webster scalar curvature and 
its torsion is constant and purely imaginary. That is to say,
(\ref{7.1}) is satisfied. We will assume that (\ref{7.1}) holds throughout this section. 
Moreover, as we have mentioned before, the assumption (\ref{7.1}) is satisfied by the Rossi sphere
$S^3_t$ (see (\ref{10.6}). 

Let 
$\Sigma_s=F_s(\Sigma)$
be a family of immersions in $(M,J,\theta)$ such that 
\begin{equation}\label{8.1}
\frac{d}{ds}F_s=X=fe_2+gT.
\end{equation}
Here, we let $e_1$ be the unit vector in $T\Sigma_s\cap\xi$
and $e_1=Je_1$. 
Then the Euler-Lagrange equation for $E_1$ has been derived in Theorem \ref{thm5.1}, i.e., 
$\mathcal{E}_1=0$ with $\mathcal{E}_1$ given in \eqref{5.1}. 

We will see that \eqref{5.1} could be simplified further if 
 $(M,J,\theta)$
has constant Webster scalar curvature and 
its torsion is constant and purely imaginary, i.e. when (\ref{7.1}) holds.  
Indeed, it follows from (\ref{5.2}), (\ref{5.9}) and (\ref{7.1}) that 
\begin{equation}\label{8.2}
\begin{split}
|H_{cr}|\mathfrak{f}
&=e_1(H)\left(H_{cr}+\frac{1}{6}H^2\right)+HV(H)+\frac{3}{2}\left(H_{cr}-\frac{1}{6}H^2\right)\\
&\hspace{4mm}-\frac{7}{2}\alpha H e_1(\alpha)-\frac{5}{2}\alpha^3H-\frac{2}{3}\alpha H^3-\frac{5}{4}\alpha H W
+2\alpha H\mbox{Im} A_{11}\\
&=e_1(H) H_{cr}+\frac{3}{2}V(H_{cr})+\frac{1}{2}H e_1\left(H_{cr}-e_1(\alpha)-\frac{1}{2}\alpha^2\right)+\frac{1}{2}HV(H)\\
&\hspace{4mm}-\frac{7}{2}\alpha H e_1(\alpha)-\frac{5}{2}\alpha^3H-\frac{2}{3}\alpha H^3-\frac{5}{4}\alpha H W
+2\alpha H\mbox{Im} A_{11}.
\end{split}
\end{equation}
Substituting 
(\ref{7.17}) into (\ref{8.2}) 
and using (\ref{5.2}) again, 
we find  
\begin{equation}\label{8.3}
\begin{split}
|H_{cr}|\mathfrak{f}
&=e_1(H) H_{cr}+\frac{3}{2}V(H_{cr})+\frac{1}{2}H e_1(H_{cr})\\
&\hspace{4mm}-\frac{1}{2}H\Big(-5\alpha e_1(\alpha)-\alpha H^2-4\alpha^3-2W\alpha+2\alpha\mbox{Im}A_{11}\Big)\\
&\hspace{4mm}-\frac{7}{2}\alpha H e_1(\alpha)-\frac{5}{2}\alpha^3H-\frac{2}{3}\alpha H^3-\frac{5}{4}\alpha H W
+2\alpha H\mbox{Im} A_{11}\\
&=e_1(H) H_{cr}+\frac{3}{2}V(H_{cr})+\frac{1}{2}H e_1(H_{cr})\\
&\hspace{4mm}
-H\alpha\left(e_1(\alpha)+\frac{1}{2}\alpha^2+\frac{1}{6}H^2+\frac{1}{4}W-\mbox{Im} A_{11}\right)\\
&=e_1(H) H_{cr}+\frac{3}{2}V(H_{cr})+\frac{1}{2}H e_1(H_{cr})
-\alpha H H_{cr}. 
\end{split}
\end{equation}
On the other hand, it follows from (\ref{5.2}) and (\ref{7.1})
that we can rewrite (\ref{5.10}) as 
\begin{equation}\label{8.4}
\begin{split}
&9h_{00}+6h_{11}h_{10}+\frac{2}{3}h_{11}^3\\
&=9V(\alpha) +6He_1(\alpha)+3H\alpha^2-6H\mbox{Im} A_{11}+\frac{3}{2}HW+\frac{2}{3}H^3\\
&=9V(\alpha)+6H H_{cr}-\frac{1}{3}H^3, 
\end{split}
\end{equation}
where $V:=T+\alpha e_2$. 
Substituting (\ref{8.4}) into (\ref{5.1}) yields 
\begin{align}\label{8.5}
\begin{split}
\mathcal{E}_1&=\frac{1}{2}e_1(|H_{cr}|^{1/2}\mathfrak{f})+\frac{3}{2}|H_{cr}|^{1/2}\alpha\mathfrak{f}\\
&\hspace{4mm}+\frac{1}{2}sign(H_{cr})|H_{cr}|^{1/2}\left(9V(\alpha)+6H H_{cr}-\frac{1}{3}H^3\right)\\
&=|H_{cr}|^{-\frac{1}{2}}\Bigg[\frac{3}{2}|H_{cr}|\mathfrak{f}\alpha+\frac{1}{2}H_{cr}\left(9V(\alpha)+6H H_{cr}-\frac{1}{3}H^3\right)\\
&\hspace{8mm}+\frac{1}{2}e_1(|H_{cr}|\mathfrak{f})-\frac{1}{4}sign(H_{cr})\mathfrak{f}e_1(H_{cr})\Bigg].
\end{split}
\end{align}

To conclude, we have the following: 

\begin{theorem}
Suppose that  $(M,J,\theta)$
has constant Webster scalar curvature and 
its torsion is constant and purely imaginary, i.e. when \eqref{7.1} holds.
Let $F_s(\Sigma)$ be given as in \eqref{8.1}. There holds 
\begin{equation}\label{8.6}
\frac{d}{ds}E_1(\Sigma_s)=\frac{d}{ds}\int_{\Sigma_s}dA_1
=\int_{\Sigma_s}\mathcal{E}_1 h \theta\wedge e^1,
\end{equation}
where 
$$dA_1=|H_{cr}|^{\frac{3}{2}}\theta\wedge e^1,~~
H_{cr}=e_1(\alpha)+\frac{1}{2}\alpha^2-\mbox{\emph{Im}} A_{11}+\frac{1}{4}W+\frac{1}{6}H^2,$$
and $\mathcal{E}_1$ is given as in \eqref{8.5}. 
\end{theorem}

\subsection{Second variation of $E_1$.}
Now we assume that $\Sigma$ in $(M,J,\theta)$
satisfies $\mathcal{E}_1(\Sigma)=0$, i.e. 
$\Sigma$ is a critical point of $E_1$, 
and $H_{cr}\neq 0$. 
Let $\Sigma_s=F_s(\Sigma)$
be a family of surfaces in $(M,J,\theta)$ 
with $\Sigma_s\big|_{s=0}=\Sigma$ such that 
(\ref{8.1}) holds. 
As before, we denote $h:=f-\alpha g$ and $V:=T+\alpha e_2$. 
Then by (\ref{8.5}) and (\ref{8.6}), 
the second variation of $E_1$ for $\Sigma$ is given by 
\begin{equation}\label{8.7}
\begin{split}
\left.\frac{d^2}{ds^2}E_1(\Sigma_s)\right|_{s=0} &=
\int_\Sigma|H_{cr}|^{-\frac{1}{2}}
\frac{d}{ds}\Bigg[-\frac{1}{4}|H_{cr}|\mathfrak{f}\frac{1}{H_{cr}}e_1(H_{cr})+\frac{1}{2}e_1(|H_{cr}|\mathfrak{f})\\
&\hspace{4mm}+\frac{3}{2}|H_{cr}|\mathfrak{f}\alpha+H_{cr}\left(\frac{9}{2}V(\alpha)+3H H_{cr}-\frac{1}{6}H^3\right)\Bigg] h\theta\wedge e^1.
\end{split}
\end{equation}
We need some lemmas to compute the above second variation of $E_1$
for the Clifford torus in the Rossi sphere.

First, we have 
\begin{equation}\label{8.8}
[e_1,V]=-\alpha H e_1-2\alpha V.
\end{equation}
To see this, we compute
\begin{align*}
[e_1,V]
&=[e_1,T+\alpha e_2]=[e_1,T]+\alpha[e_1,e_2]+e_1(\alpha)e_2\\
&=-(\mbox{Im}A_{11}+\omega(T))e_2
-\alpha\big(2T+\omega(e_1)e_1+\omega(e_2)e_2\big)+e_1(\alpha)e_2\\
&=-\big(e_1(\alpha)-\alpha h^{-1}e_1(h)\big)e_2
-\alpha\big(2T+He_1+(h^{-1}e_1(h)+2\alpha)e_2\big)+e_1(\alpha)e_2\\
&=-\alpha He_1-2\alpha V,
\end{align*}
where we have used 
(\ref{7.2}), (\ref{7.10}), (\ref{7.11}), 
and (\ref{7.9}). 

\begin{lem}\label{lem7.5} We have 
\begin{align*}
\frac{d}{ds}V(\alpha)
&=V\left(\frac{d\alpha}{ds}\right)+h e_1(\alpha)e_1(\alpha)-\alpha e_1(h)e_1(\alpha)-2h\mbox{\emph{Im}}A_{11}e_1(\alpha)
-V(g)V(\alpha)
\end{align*}
and 
\begin{align*}
\frac{d}{ds}\big(e_1(|H_{cr}|\mathfrak{f})\big)
&=e_1\left(\frac{d}{ds}(|H_{cr}|\mathfrak{f})\right)+
\big(fHe_1 +2fV-e_1(g)V\big)(|H_{cr}|\mathfrak{f}).
\end{align*}
where 
\begin{align*}
\frac{d}{ds}(|H_{cr}|\mathfrak{f})
&= H_{cr}\frac{d}{ds}\big(e_1(H)\big)
+e_1(H)\frac{d}{ds}H_{cr}+\frac{3}{2}\frac{d}{ds}V(H_{cr})\\
&\hspace{4mm}+\frac{1}{2}e_1(H_{cr})\frac{dH}{ds}+\frac{1}{2}H\frac{d}{ds}\big(e_1(H_{cr})\big)\\
&\hspace{4mm}-\frac{d\alpha}{ds} H H_{cr}-\alpha\frac{dH}{ds}H_{cr}
-\alpha H\frac{d}{ds}H_{cr}.
\end{align*}
Here, 
\begin{align*}
\frac{d}{ds}e_1(H)
&=e_1\Bigg(e_1e_1(h)+2\alpha e_1(h)+4he_1(\alpha)+H^2h+4\alpha^2h+2Wh-2h\mbox{\emph{Im}}A_{11}\\
&\hspace{8mm}
-g\big(
2\alpha\mbox{\emph{Im}}A_{11}+h^{-1}e_1(h)\mbox{\emph{Im}}A_{11}-\mbox{\emph{Im}}A_{11}\big)\Bigg)\\
&\hspace{4mm}
+ge_1V(H)+fHe_1(H) +2fV(H)\\
\frac{d}{ds}V(H_{cr})
&=V\left(\frac{d}{ds}H_{cr}\right)+
h e_1(\alpha)e_1(H_{cr})-\alpha e_1(h)e_1(H_{cr})\\
&\hspace{4mm}-2h\mbox{\emph{Im}}A_{11}e_1(H_{cr})
-V(g)V(H_{cr})
\end{align*}
and 
\begin{align*}
&\frac{d}{ds}H_{cr}\\
&=e_1V(h)+\alpha V(h)+\frac{1}{3}H\big(e_1e_1(h)+2\alpha e_1(h)\big)\\
&\hspace{4mm}+\frac{4}{3}H\left(e_1(\alpha)+\alpha^2+\frac{1}{4}H^2+\frac{1}{2}W-\frac{1}{2}\mbox{\emph{Im}}A_{11}\right)h
+\big(2V(\alpha)+He_1(\alpha)\big)f\\
&\hspace{4mm}+\left(e_1V(\alpha)+\alpha V(\alpha)
+\frac{1}{3}H\big(V(H)-2\alpha\mbox{\emph{Im}}A_{11}-h^{-1}e_1(h)\mbox{\emph{Im}}A_{11}+\mbox{\emph{Im}}A_{11}\big)\right)g.
\end{align*}
\end{lem}

The proof is similar to that in \cite[Lemma 5.5]{CCYZ} 
with the torsion terms added. For the reader's convenience, we give the complete proof in the Appendix.

\subsection{The Clifford torus in the Rossi sphere.}
As in Section \ref{section8}, 
 we consider 
a closed surface $\Sigma_{[c]}$ defined by 
$\rho_1=c$ in the Rossi sphere $S^3_t$,
where 
$$z_1=\rho_1 e^{i\varphi_1},~~ z_2=\rho_2 e^{i\varphi_2}~~\mbox{ and }~~
\rho_1^2+\rho_2^2=1.$$
The Clifford torus is the surface $\Sigma_{[\frac{\sqrt{2}}{2}]}$, i.e. 
when $\rho_1=\rho_2=\frac{\sqrt{2}}{2}$.

We have already shown in Lemma \ref{lem5.1} that 
the Clifford torus
 $\Sigma_{[\frac{\sqrt{2}}{2}]}$ 
in the Rossi sphere $S^3_{t}$, where 
$t=-4+\sqrt{15}$, 
is a minimizer for the energy $E_1$ 
 with zero energy.    
Therefore, we are going to assume that $t\neq -4+\sqrt{15}$. 
It follows from Lemma \ref{lem5.2}
that the Clifford torus
 $\Sigma_{[\frac{\sqrt{2}}{2}]}$ 
in the Rossi sphere $S^3_{t}$
is a critical point for the energy $E_1$ with nonzero energy.

We are going to apply (\ref{8.7}) to compute the second variation 
of the Clifford torus $\Sigma_{[\frac{\sqrt{2}}{2}]}$ 
in the Rossi sphere $S^3_{t}$, where $t\neq -4+\sqrt{15}$.

\begin{theorem}\label{thm8.1}
Let $\Sigma_{[\frac{\sqrt{2}}{2}]}$ be the Clifford torus 
in the Rossi sphere $S^3_{t}$, where $t\neq -4+\sqrt{15}$.
If $f$ or $g$ are supported in a domain of $\Sigma_{[\frac{\sqrt{2}}{2}]}$
away from the singular set of  $\Sigma_{[\frac{\sqrt{2}}{2}]}$, there holds 
\begin{equation}\label{8.18}
\begin{split}
&\left.\frac{d^2}{ds^2}E_1(\Sigma_s)\right|_{s=0}\\ &=
\sqrt{\left|\frac{2(1-t^2)}{1+8t+t^2}\right|}\int_{\Sigma_{[\frac{\sqrt{2}}{2}]}}
\Bigg[\frac{1+8t+t^2}{4(1-t^2)}\left(e_1e_1(h)^2-\frac{4(1-t)}{1+t}e_1(h)^2\right)
+\frac{3}{2}Te_1(h)^2\\
&\hspace{8mm}+\frac{9(1+8t+t^2)}{4(1-t^2)}hTT(h)
+\frac{3(1+8t+t^2)^2}{4(1-t^2)^2}\left(-e_1(h)^2+\frac{4(1-t)}{1+t}h^2\right)
\Bigg]\theta\wedge e^1\\
&\hspace{4mm}-\sqrt{\left|\frac{2(1-t^2)}{1+8t+t^2}\right|} \int_{\Sigma_{[\frac{\sqrt{2}}{2}]}}
\Bigg[\frac{1}{2}H_{cr}e_1e_1\Big(g\big(h^{-1}e_1(h)\mbox{\emph{Im}}A_{11}-\mbox{\emph{Im}}A_{11}\big)\Big)\\
&\hspace{8mm}
+3H_{cr}^2g\big(h^{-1}e_1(h)\mbox{\emph{Im}}A_{11}-\mbox{\emph{Im}}A_{11}\big)\Bigg]h\theta\wedge e^1,
\end{split}
\end{equation}
where $h=f-\alpha g$. 
\end{theorem}

We remark that Theorem \ref{thm8.1} extends 
\cite[Theorem 1.3]{CCYZ}. Indeed, 
\cite[Theorem 1.3]{CCYZ} is the special case of Theorem \ref{thm8.1}
when $t=0$. 
 
\begin{proof}[Proof of Theorem \ref{thm8.1}]
On the Clifford torus $\Sigma_{[\frac{\sqrt{2}}{2}]}$, it follows from  (\ref{10.6A}), (\ref{10.8}), (\ref{10.14A}), (\ref{10.31})
and (\ref{10.30A}) that 
\begin{equation}\label{8.9}
\begin{split}
&\alpha=0,~~e_1=\frac{1}{|a|}\left(-\frac{\partial}{\partial\varphi_1}+\frac{\partial}{\partial\varphi_2}\right),
~~H=0,\\
&X_1X_2\cdots X_n(H)=0,
\end{split}
\end{equation}
as long as the vector fields $X_j=e_1$ or $T$ for all $1\leq j\leq n$. 
From this and the fact that $V=T-\alpha e_2$, we have 
\begin{equation}\label{8.10}
V=T
\end{equation}
 on the Clifford torus $\Sigma_{[\frac{\sqrt{2}}{2}]}$. 
It also follows from (\ref{8.8}) and (\ref{8.9}) that
\begin{equation}\label{8.11}
[e_1,V]=0. 
\end{equation}
Moreover, it follows from (\ref{10.28}), (\ref{8.3}) and (\ref{8.9}) that 
on the Clifford torus $\Sigma_{[\frac{\sqrt{2}}{2}]}$
\begin{equation}\label{8.12}
\begin{split}
H_{cr}&=\frac{1+8t+t^2}{2(1-t^2)},\\
|H_{cr}|\mathfrak{f}&=e_1(H) H_{cr}+\frac{3}{2}V(H_{cr})+\frac{1}{2}H e_1(H_{cr})
-\alpha H H_{cr}=0.
\end{split}
\end{equation}
Substituting (\ref{8.9})-(\ref{8.12}) into 
(\ref{8.7}), we find that the second variation of $E_1$ at $\Sigma_s\big|_{s=0}=\Sigma_{[\frac{\sqrt{2}}{2}]}$, 
the Clifford torus, is given by 
\begin{equation}\label{8.13}
\begin{split}
&\left.\frac{d^2}{ds^2}E_1(\Sigma_s)\right|_{s=0}\\ &=
|H_{cr}|^{-\frac{1}{2}}\int_{\Sigma_{[\frac{\sqrt{2}}{2}]}}
\Bigg[\frac{1}{2}\frac{d}{ds}e_1(|H_{cr}|\mathfrak{f})
+\frac{9}{2}H_{cr}\frac{d}{ds}V(\alpha)+3H_{cr}^2\frac{dH}{ds}\Bigg]h\theta\wedge e^1.
\end{split}
\end{equation}
From (\ref{7.22}) and (\ref{8.9})-(\ref{8.12}), 
we find
\begin{equation}\label{8.14}
\begin{split}
\frac{dH}{ds}
&=e_1e_1(h) +2h\big( W- \mbox{Im}A_{11}\big)
-g\big(h^{-1}e_1(h)\mbox{Im}A_{11}-\mbox{Im}A_{11}\big).
\end{split}
\end{equation}
Similarly, from (\ref{7.23}), (\ref{8.9})-(\ref{8.12}) and Lemma \ref{lem7.5}, 
we have 
\begin{align}\label{8.15}
\frac{d}{ds}V(\alpha)
&=V\left(\frac{d\alpha}{ds}\right)
=VV(h).
\end{align}
Furthermore, from (\ref{8.9})-(\ref{8.12}) and Lemma \ref{lem7.5}, 
we have 
\begin{equation}\label{8.16}
\begin{split}
&\frac{d}{ds}\big(e_1(|H_{cr}|\mathfrak{f})\big)\\
&=e_1\left(\frac{d}{ds}(|H_{cr}|\mathfrak{f})\right)=e_1\left(H_{cr}\frac{d}{ds}\big(e_1(H)\big)+\frac{3}{2}\frac{d}{ds}V(H_{cr})\right)\\
&=H_{cr}e_1\left(\frac{d}{ds}\big(e_1(H)\big)\right)+\frac{3}{2}e_1V\left(\frac{d}{ds} H_{cr}\right)\\
&=H_{cr}e_1\Bigg(e_1e_1e_1(h)+2(W-\mbox{Im}A_{11})e_1(h)-e_1\Big(g\big(h^{-1}e_1(h)\mbox{Im}A_{11}-\mbox{Im}A_{11}\big)\Big)\Bigg)\\
&\hspace{4mm}+\frac{3}{2}e_1Ve_1V(h).
\end{split}
\end{equation}
Substituting   (\ref{8.14})-(\ref{8.16}) into (\ref{8.13})
yields 
\begin{equation}\label{8.17}
\begin{split}
&\left.\frac{d^2}{ds^2}E_1(\Sigma_s)\right|_{s=0}\\ &=
|H_{cr}|^{-\frac{1}{2}}\int_{\Sigma_{[\frac{\sqrt{2}}{2}]}}
\Bigg[\frac{1}{2}H_{cr}\Big(e_1e_1e_1e_1(h)+2(W-\mbox{Im}A_{11})e_1e_1(h)\Big)
+\frac{3}{2}e_1V e_1V(h)\\
&\hspace{8mm}+\frac{9}{2}H_{cr}VV(h)
+3H_{cr}^2\Big(e_1e_1(h)+2h(W-\mbox{Im}A_{11})\Big)
\Bigg]h\theta\wedge e^1\\
&\hspace{4mm}-|H_{cr}|^{-\frac{1}{2}}\int_{\Sigma_{[\frac{\sqrt{2}}{2}]}}
\Bigg[\frac{1}{2}H_{cr}e_1e_1\Big(g\big(h^{-1}e_1(h)\mbox{Im}A_{11}-\mbox{Im}A_{11}\big)\Big)\\
&\hspace{8mm}
+3H_{cr}^2g\big(h^{-1}e_1(h)\mbox{Im}A_{11}-\mbox{Im}A_{11}\big)\Bigg]h\theta\wedge e^1.
\end{split}
\end{equation}
Since $\alpha=0$ on the Clifford torus $\Sigma_{[\frac{\sqrt{2}}{2}]}$
by (\ref{8.9}), it follows from Lemma \ref{lem7.4} that 
\begin{equation}\label{8.19}
\int_{\Sigma_{[\frac{\sqrt{2}}{2}]}} f_1 e_1(f_2)\theta\wedge e^1=-\int_{\Sigma_{[\frac{\sqrt{2}}{2}]}} e_1(f_1) f_2\theta\wedge e^1,
\end{equation}
and 
\begin{equation}\label{8.20}
\int_{\Sigma_{[\frac{\sqrt{2}}{2}]}} f_1 V(f_2)\theta\wedge e^1=-\int_{\Sigma_{[\frac{\sqrt{2}}{2}]}}V(f_1)f_2\theta\wedge e^1,
\end{equation}
if either $f_1$ and $f_2$ has compact support in the nonsingular domain of $\Sigma_{[\frac{\sqrt{2}}{2}]}$. 
Applying (\ref{8.19}) and (\ref{8.20}) to (\ref{8.17}) 
and using  (\ref{8.10}) and (\ref{8.11}), we 
obtain (\ref{8.18}). 
This proves the assertion.   
\end{proof}

Taking
$$g=0~~\mbox{ and }~~f=v_k(z)=\cos k(\varphi_1+\varphi_2),~~\mbox{ for }k=1,2,\cdots,$$
we find 
\begin{equation*}
e_1(h)=e_1(f)=0,~~TT(h)=TT(f)=-4k^2f
\end{equation*}
by using (\ref{10.3}) and (\ref{8.9}). 
Substituting these into (\ref{8.18}) yields 
\begin{equation}\label{8.21}
\begin{split}
&\sqrt{\left|\frac{1+8t+t^2}{2(1-t^2)}\right|}\left.\frac{d^2}{ds^2}E_1(\Sigma_s)\right|_{s=0}\\ &=
\int_{\Sigma_{[\frac{\sqrt{2}}{2}]}}
\Bigg[\frac{9(1+8t+t^2)}{4(1-t^2)}fTT(f)
+\frac{3(1+8t+t^2)^2}{4(1-t^2)^2}\left(\frac{4(1-t)}{1+t}f^2\right)
\Bigg]\theta\wedge e^1\\
&=\frac{3(1+8t+t^2)}{1-t^2}
\int_{\Sigma_{[\frac{\sqrt{2}}{2}]}}
\left(-3k^2
+\frac{1+8t+t^2}{(1+t)^2} 
\right)f^2\theta\wedge e^1.
\end{split}
\end{equation}
Recall that we always have $|t|<1$
for the Rossi sphere $S^3_t$. 
Moreover, if $t<-4+\sqrt{15}$, then we have 
$1+8t+t^2>0$. Therefore, if $t>-4+\sqrt{15}$, 
it follows from (\ref{8.21}) that 
$$\left.\frac{d^2}{ds^2}E_1(\Sigma_s)\right|_{s=0}<0$$
if $k$ is sufficiently large. From this, we have the following: 

\begin{cor}\label{cor8.2}
If $t>-4+\sqrt{15}$, then the Clifford torus $\Sigma_{[\frac{\sqrt{2}}{2}]}$
is not a minimizer for $E_1$ among all surfaces of torus type in the 
Rossi sphere $S^3_t$. 
\end{cor}

We remark that Corollary \ref{cor8.2} extends 
\cite[Corollary 1.4]{CCYZ}. Indeed, 
\cite[Corollary 1.4]{CCYZ} proved the special case of Corollary \ref{cor8.2}
when $t=0$.

\section{Three-dimensional tori}\label{section9}

Now we consider the
three-dimensional tori. 
Note that this has been studied by Dall'Ara and Son in \cite{DS}, in which 
they studied the Kohn Laplacian of the three-dimensional tori.
We first recall its construction. 

Let $\gamma:[0,L]\to\mathbb{R}^2$, 
given by 
$\gamma(s)=(\xi(s),\eta(s))$, 
be a smooth closed plane curve parametrized 
by arc-length $s$ such that 
its curvature $\kappa(s)$ is positive. 
We consider the $3$-dimensional torus $M$
such that $\gamma$ is the generating curve of the torus: 
$M$ is parametrized by 
$$(s,x,y)\mapsto(e^{\xi(s)+ix},e^{\eta(s)+iy})\in \mathbb{C}^2$$
where $s\in [0,L]$ and $(x,y)\in\mathbb{T}^2$.  
Then we can define 
a contact form $\theta$ on 
the $3$-dimensional torus
$M:=(\mathbb{R}/L\mathbb{Z})\times\mathbb{T}^2$
by 
\begin{equation}\label{6.1}
\theta=\eta' dx+\xi'dy.
\end{equation}
Hereafter, $'$ denotes the derivative with respect to $s$.
Note that 
$$d\theta=\kappa(\xi'ds\wedge dx-\rho' ds\wedge dy),~~
\theta\wedge d\theta=\kappa ds\wedge dx\wedge dy>0.$$
The Reeb vector field $T$ is given by 
\begin{equation}\label{6.2}
T=\eta'\frac{\partial}{\partial x}+\xi'\frac{\partial}{\partial y}.
\end{equation}
The CR structure on $M$ is defined by 
\begin{equation}\label{6.3}
Z_1=\frac{1}{\sqrt{2\kappa}}
\left(\frac{\partial}{\partial s}-i\xi'\frac{\partial}{\partial x}+i\eta'\frac{\partial}{\partial y}\right).
\end{equation}
The almost complex structure $J$ maps 
\begin{equation}\label{6.4}
J(Z_1)=iZ_1~~\mbox{ and }~~J(Z_{\overline{1}})=-iZ_{\overline{1}}, 
\end{equation}
where $Z_{\overline{1}}=\overline{Z_1}$.
If we define 
\begin{equation}\label{6.5}
\theta^1=\sqrt{\frac{\kappa}{2}}
\big(ds+i\xi'dx-i\eta'dy\big),
\end{equation}
then one can check, by using 
(\ref{6.1})-(\ref{6.3}) and (\ref{6.5}), 
that 
 $(\theta,\theta^1,\theta^{\overline{1}}=\overline{\theta^1})$
is the admissible coframe.
From (\ref{6.1}) and (\ref{6.5}), we compute
\begin{equation}\label{6.11}
d\theta=i\theta^1\wedge \theta^{\overline{1}}
\end{equation}
and 
\begin{equation*}
d\theta^1
=\theta^1\wedge\left(\frac{\kappa'}{\sqrt{8\kappa^3}}
(\theta^1-\theta^{\overline{1}})-\frac{i}{2}\kappa\theta\right)
+\theta\wedge\left(\frac{i}{2}\kappa\theta^{\overline{1}}\right).
\end{equation*}
This implies that 
\begin{equation}\label{6.6}
\omega_1^1=\frac{\kappa'}{\sqrt{8\kappa^3}}
(\theta^1-\theta^{\overline{1}})-\frac{i}{2}\kappa\theta
~~\mbox{ and }~~
A_{\overline{1}\,\overline{1}}=\frac{i}{2}\kappa 
~~\mbox{ and }~~
A_{11}=\overline{A_{\overline{1}\,\overline{1}}}
=-\frac{i}{2}\kappa.
\end{equation}
From (\ref{6.6}), we find 
\begin{equation*}
d\omega_1^1=
\left(\frac{\kappa}{2}-\frac{(\log\kappa)''}{2\kappa}\right)\theta^1\wedge\theta^{\overline{1}}
-\frac{i\kappa'}{\sqrt{2\kappa}}(\theta^1+\theta^{\overline{1}})\wedge\theta.
\end{equation*}
This implies that the Webster curvature is given by 
\begin{equation}\label{6.7}
W=\frac{\kappa}{2}-\frac{(\log\kappa)''}{2\kappa}
=\frac{\kappa^4-\kappa\kappa''+(\kappa')^2}{2\kappa^3}.
\end{equation}

It may be of interest to remark that, by a result of Sunada \cite{Sunada}, two
such $3$-dimensional tori  are biholomorphically equivalent if and only if
the generating curves are rigid equivalent.
Moreover, we would like to emphasize that 
 it follows from (\ref{6.6}) that the torsion never vanishes, 
 since $\kappa>0$ by assumption. 
In particular, the Euler-Lagrangian equations 
(\ref{EE_for_E1})-(\ref{EE_for_E2}) are not applicable. 
This shows the necessity of proving 
Theorem \ref{thm5.1} in Section \ref{section5} and Theorem \ref{thm7.1}
in Section \ref{section6} in order to study the critical points of $E_1$ and $E_2$ 
on the $3$-dimensional tori.

\subsection{Surfaces given by $s=c$.}
We consider the surface $\Sigma_c$ given by $s=c$ for some constant $c\in [0,L]$. 
Note that $\Sigma_c$ is nonsingular. To see this, note that 
$T\Sigma_c$ is spanned by $\displaystyle\frac{\partial}{\partial x}$ 
and $\displaystyle\frac{\partial}{\partial y}$. 
If $p$ is a singular point in $\Sigma_c$, then $\displaystyle\frac{\partial}{\partial x}\in\ker\theta$
and  $\displaystyle\frac{\partial}{\partial y}\in\ker\theta$
at $p$. 
Note that $\displaystyle\frac{\partial}{\partial x}\in\ker\theta$
(and  $\displaystyle\frac{\partial}{\partial y}\in\ker\theta$ respectively)
if and only if 
$\eta'=0$
 (and $\xi'=0$ respectively) by (\ref{6.1}).
Hence, if $p$ is a singular point, we would have 
$\eta'=0$ and $\xi'=0$ at $p$, which contradicts the fact that $\gamma$ is regular. 

By (\ref{6.2}), 
$T=\displaystyle\eta'\frac{\partial}{\partial x}+\xi'\frac{\partial}{\partial y}$, 
and hence $T\in T\Sigma_c$. 
Thus we have 
\begin{equation}\label{6.8}
\alpha\equiv 0
\end{equation}
on $\Sigma_c$. 
By inspection, we see that 
$\displaystyle\xi'\frac{\partial}{\partial x}-\eta'\frac{\partial}{\partial y}\in \ker\theta\cap T\Sigma$. 
Therefore, we have 
\begin{equation}\label{6.9}
e_1=\beta\left(\xi'\frac{\partial}{\partial x}-\eta'\frac{\partial}{\partial y}\right).
\end{equation}
for some constant $\beta$. 
To find $\beta$, we write 
$\displaystyle
\xi'\frac{\partial}{\partial x}-\eta'\frac{\partial}{\partial y}=aZ_1+bZ_{\overline{1}}. 
$
From (\ref{6.3}), we find
$a=\displaystyle i\sqrt{\frac{\kappa}{2}}$ and $b=\displaystyle-i\sqrt{\frac{\kappa}{2}}$. That is to say, 
\begin{equation}\label{6.10}
\xi'\frac{\partial}{\partial x}-\eta'\frac{\partial}{\partial y}=i\sqrt{\frac{\kappa}{2}}Z_1-i\sqrt{\frac{\kappa}{2}}Z_{\overline{1}}. 
\end{equation}
Since $e_1$ is a unit vector 
with respect to  $\frac{1}{2}d\theta(\cdot,J\cdot)$, 
we have 
\begin{equation*}
\begin{split}
1&=\frac{1}{2}d\theta(e_1,Je_1)=\frac{\beta^2}{2}d\theta\left(\xi'\frac{\partial}{\partial x}-\eta'\frac{\partial}{\partial y}, J\Big(\xi'\frac{\partial}{\partial x}-\eta'\frac{\partial}{\partial y}\Big)\right)\\
&=\frac{\beta^2}{2}d\theta\left(i\sqrt{\frac{\kappa}{2}}Z_1-i\sqrt{\frac{\kappa}{2}}Z_{\overline{1}}, J\Big(i\sqrt{\frac{\kappa}{2}}Z_1-i\sqrt{\frac{\kappa}{2}}Z_{\overline{1}}\Big)\right)\\
&=-\frac{\beta^2\kappa}{4}d\theta\big(Z_1-Z_{\overline{1}},J(Z_1-Z_{\overline{1}})\big)\\
&=-\frac{\beta^2\kappa}{4}id\theta\big(Z_1-Z_{\overline{1}},Z_1+Z_{\overline{1}})\big)
=-\frac{\beta^2\kappa}{2}id\theta\big(Z_1,Z_{\overline{1}})\big)\\
&=\frac{\beta^2\kappa}{2}(\theta^1\wedge \theta^{\overline{1}})\big(Z_1-Z_{\overline{1}},Z_1+Z_{\overline{1}})\big)=
\frac{\beta^2\kappa}{2},
\end{split}
\end{equation*}
where the second equality follows from (\ref{6.9}), 
the third equality follows from (\ref{6.10}), 
the fifth equality follows from (\ref{6.4}), 
and the second last equality follows from (\ref{6.11}). 
This gives $\beta=\displaystyle\sqrt{\frac{2}{\kappa}}$. 
Hence, it follows from (\ref{6.9}) and (\ref{6.10}) that 
\begin{equation}\label{6.12}
e_1=\sqrt{\frac{2}{\kappa}}\left(\xi'\frac{\partial}{\partial x}-\eta'\frac{\partial}{\partial y}\right)
=i(Z_1-Z_{\overline{1}}).
\end{equation}
It follows from (\ref{6.4}) and (\ref{6.12}) that
\begin{equation}\label{6.13}
e_2=J e_1=i(JZ_1-JZ_{\overline{1}})=-(Z_1+Z_{\overline{1}}).
\end{equation}
From (\ref{6.12}) and (\ref{6.13}), we find 
\begin{equation}\label{6.14}
e^1=\frac{i}{2}(\theta^{\overline{1}}-\theta^1)~~\mbox{ and }~~
e^2=-\frac{1}{2}(\theta^{\overline{1}}+\theta^1).
\end{equation}
To check (\ref{6.14}), one can see from (\ref{6.2}), (\ref{6.12})-(\ref{6.14}) that $e^j(e_k)=\delta_{jk}$
and $e^j(T)=0$ for all $j,k=1,2$. 
Substituting
(\ref{6.5})
into (\ref{6.14}) yields  
\begin{equation}\label{6.15}
e^1=i\sqrt{2\kappa}(-\xi'dx+\eta' dy)
~~\mbox{ and }~~e^2=-\sqrt{\frac{\kappa}{2}}ds.
\end{equation}
From (\ref{6.1}) and (\ref{6.15}), we compute 
\begin{equation}\label{6.16}
\theta\wedge e^1=-\sqrt{\frac{\kappa}{2}} dx\wedge dy
~~\mbox{ and }~~\theta\wedge e^1\wedge e^2=\frac{\kappa}{2}ds\wedge dx\wedge dy.
\end{equation}
Taking the exterior derivative of the first equation in (\ref{6.16}) gives 
\begin{equation}\label{6.17}
d(\theta\wedge e^1)=-\frac{\kappa'}{2\sqrt{2\kappa}}ds\wedge dx\wedge dy.
\end{equation}
Since $d(\theta\wedge e^1)=-H\theta\wedge e^1\wedge e^2$, 
we can deduce from (\ref{6.16}) and (\ref{6.17}) that 
\begin{equation}\label{6.18}
H=\frac{\kappa'}{\sqrt{2}\kappa^{\frac{3}{2}}}. 
\end{equation}

Substituting (\ref{6.6}), (\ref{6.7}), (\ref{6.8}) and (\ref{6.18}) into (\ref{5.2}), we find 
\begin{equation}\label{6.19}
H_{cr}=-\mbox{Im} A_{11}+\frac{1}{4}W+\frac{1}{6}H^2
=\frac{1}{2}\kappa+\frac{\kappa^4-\kappa\kappa''+(\kappa')^2}{8\kappa^3}
+\frac{(\kappa')^2}{12\kappa^3}.
\end{equation}

\subsubsection{Example 1.} 
Consider the $3$-dimensional torus $M$ such that 
the generating curve $\gamma: [0,2\pi r]\to\mathbb{R}^2$ is the circle of radius $r>0$, i.e. 
$\gamma(s)=(r\cos\frac{s}{r},r\sin \frac{s}{r})$. 
Then its curvature equals
\begin{equation}\label{6.20}
\kappa\equiv \frac{1}{r}.
\end{equation}
In particular, it satisfies the assumption that the curvature $\kappa>0$. 
Therefore, we can consider the $3$-dimensional torus $M$ with $\gamma$ as generating curve. 

Now, for the surfaces $\Sigma$ given by $s=c$, 
we substitute (\ref{6.20}) into (\ref{6.18}) and (\ref{6.19}) to get
\begin{equation}\label{6.21}
H=0
\end{equation}
and 
\begin{equation}\label{6.22}
H_{cr}=\frac{5}{8r}\neq 0. 
\end{equation}
Substituting (\ref{6.6}), (\ref{6.7}), (\ref{6.8}), (\ref{6.21}) and (\ref{6.22}) into
(\ref{5.9}) and (\ref{5.10}), we can conclude from (\ref{5.1})
that 
$$\mathcal{E}_1=0.$$
To summarize, we have proved the following: 

\begin{lem}\label{lem6.1}
Consider the $3$-dimensional torus $M$ such that 
the generating curve  $\gamma: [0,2\pi r]\to\mathbb{R}^2$
is the circle of radius $r$ given by
$\gamma(s)=(r\cos\frac{s}{r},r\sin \frac{s}{r}).$
Then for any $c\in [0,2\pi r]$, the surface $\Sigma_c$ defined by $s=c$ 
is a critical point for the functional $E_1$ with nonzero energy. 
\end{lem}

Similar to Lemma \ref{lem5.3}, 
we have the following: 

\begin{lem} 
Consider the $3$-dimensional torus $M$ such that 
the generating curve  $\gamma: [0,2\pi r]\to\mathbb{R}^2$
is the circle of radius $r$ given by
$\gamma(s)=(r\cos\frac{s}{r},r\sin \frac{s}{r}).$
If $\Sigma$ is a surface in $M$ without singular point, 
then it cannot have $E_1=0.$
\end{lem}
\begin{proof}
If $E_1=0$, then it follows from (\ref{5.2}), (\ref{6.6}), (\ref{6.7}) and  (\ref{6.20}) that 
\begin{equation*}
\begin{split}
e_1(\alpha)&=-\frac{1}{2}\alpha^2+\mbox{Im} A_{11}-\frac{1}{4}W-\frac{1}{6}H^2\\
&=-\frac{1}{2}\alpha^2-\frac{1}{2r}-\frac{1}{8r}-\frac{1}{6}H^2<0.
\end{split}
\end{equation*}
Therefore, the characteristic curve of $e_1$ is open 
and along the curve $\alpha$ goes to infinity, which contradicts to $\alpha$ being 
bounded on a nonsingular compact surface.   
\end{proof}

For the circle of radius $r$, i.e. $\gamma: [0,2\pi r]\to\mathbb{R}^2$ given by 
$\gamma(s)=(r\cos\frac{s}{r},r\sin \frac{s}{r})$, 
its curvature is given as in (\ref{6.20}). 
It follows from (\ref{6.6}), (\ref{6.7}) and  (\ref{6.20}) that the $3$-dimensional torus $M$ 
with $\gamma$ as generating curve 
satisfies 
\begin{equation}\label{6.29}
A_{11}=-\frac{i}{2r}~~\mbox{ and }~~W=\frac{1}{2r}.
\end{equation}
That is to say, $M$ has constant Webster scalar curvature and 
its torsion is constant and purely imaginary.
In particular, Theorem \ref{thm7.1} is applicable to $M$. 

Now, for the surface $\Sigma$ defined by $s=c$, 
(\ref{6.8}) and (\ref{6.21}) hold. 
Substituting (\ref{6.8}), (\ref{6.21}) and (\ref{6.29}) into 
 (\ref{7.34}), we find 
\begin{equation*}
\frac{9}{4}\mathcal{E}_2=\frac{3}{4}W^2
+3(\mbox{Im}A_{11})^2-\frac{15}{4}W\mbox{Im}A_{11}=\frac{15}{8r^2}>0. 
\end{equation*}
 In particular, it follows from Theorem \ref{thm7.1} that
  the surface $\Sigma$ defined by $s=c$
 is not a critical point 
for the functional $E_2$ in $M$.

\subsubsection{Example 2.} 
Consider the $3$-dimensional torus $M$ such that 
the generating curve  $\gamma: [0,2\pi]\to\mathbb{R}^2$
is the ellipse given by 
$$\gamma(t)=(a\cos t, b\sin t),$$
where $a$ and $b$ are fixed positive number. 
Note that $t$ may not be the arc-length $s$. Indeed, the arclength $s$ is given by 
\begin{equation}\label{6.23}
s=\int^t_0\sqrt{a^2\sin^2t+b^2\cos^2t}dt.
\end{equation}
The curvature $\kappa$ of $\gamma$ can be computed as
\begin{equation}\label{6.24}
\kappa=\frac{\ddot{(b\sin t)} \dot{(a\cos t)}-\ddot{(a\cos t)}\dot{(b\sin t)}}{(a^2\sin^2t+b^2\cos^2t)^{\frac{3}{2}}}
=\frac{ab}{(a^2\sin^2t+b^2\cos^2t)^{\frac{3}{2}}},
\end{equation}
where $\dot{}$ denotes derivatives with respect to $t$. 
From (\ref{6.24}), we see that $\kappa$ is positive. 
Recall that $'$ denotes derivatives with respect to $s$. 
It follows from (\ref{6.23}) and (\ref{6.24}) that 
\begin{equation}\label{6.25}
\kappa'=\frac{d}{dt}\left(\frac{ab}{(a^2\sin^2t+b^2\cos^2t)^{\frac{3}{2}}}\right)\left(\frac{ds}{dt}\right)^{-1}
=-\frac{3ab(b^2-a^2)\sin(2t)}{2(a^2\sin^2t+b^2\cos^2t)^{3}}.
\end{equation}
Also, it follows from (\ref{6.23}) and (\ref{6.25}) that 
\begin{equation}\label{6.26}
\begin{split}
  \kappa''&=-\frac{d}{dt}\left(\frac{3ab(b^2-a^2)\sin (2t)}{2(a^2\sin^2t+b^2\cos^2t)^{3}}\right)\left(\frac{ds}{dt}\right)^{-1}\\
  &=\frac{3ab(b^2-a^2)\cos (2t)}{(a^2\sin^2t+b^2\cos^2t)^{\frac{7}{2}}}+\sin(2t) g(t)
  \end{split}
\end{equation}
for some smooth function $g(t)$. 
Substituting (\ref{6.24})-(\ref{6.26}) into 
(\ref{6.19})
and evaluating it at $t=0$ and at $t=\frac{\pi}{2}$, we obtain 
\begin{equation}\label{6.27}
H_{cr}|_{t=0}=\frac{8a^2-3b^2}{8ab^2}~~\mbox{ and }~~
H_{cr}|_{t=\frac{\pi}{2}}=\frac{8b^2-3a^2}{8a^2b}.
\end{equation}
In particular, if 
$\frac{b^2}{a^2}<\frac{3}{8}$, then 
it follows from (\ref{6.27}) that 
\begin{equation}\label{6.28}
H_{cr}|_{t=0}>0~~\mbox{ and }~~
H_{cr}|_{t=\frac{\pi}{2}}<0.
\end{equation}
By intermediate value theorem, there exists $t_0$ between $0$ and $\frac{\pi}{2}$
such that 
$H_{cr}|_{t=t_0}=0$. It follows from (\ref{6.23}) that 
$s_0=\displaystyle\int^{t_0}_0\sqrt{a^2\sin^2t+b^2\cos^2t}dt.$
To conclude, we have proved the following: 

\begin{lem}\label{lem6.2}
Consider the $3$-dimensional torus $M$ such that 
the generating curve  $\gamma: [0,2\pi]\to\mathbb{R}^2$
is the ellipse given by 
$\gamma(t)=(a\cos t, b\sin t).$
If $\frac{b^2}{a^2}<\frac{3}{8}$, then 
there exists $t_0\in (0,\pi)$ such that 
the surface $\Sigma_{s_0}$ defined by $s=s_0$, where $s_0=\int^{t_0}_0\sqrt{a^2\sin^2t+b^2\cos^2t}dt$, 
is a minimizer for the functional $E_1$ with zero energy.
\end{lem}

\section{Appendix}

In this Appendix, we prove Lemmas \ref{lem7.2}, \ref{lem7.3},
\ref{lem7.5} and Theorem \ref{thm7.1}. 

\begin{proof}[Proof of Lemma \ref{lem7.2}]
Note that (\ref{7.17}) is defined on $\Sigma$. 
To prove it, we extend the frame $e_2, e_1=-Je_2$
at points of $\Sigma$ to a neighborhood $U$ of $\Sigma$
(still denoted by the same notation) such that
$e_2\in\xi$ and 
$$\nabla_{e_2}e_2=0~~\mbox{ and }~~
e_1:=-Je_2~~\mbox{ in }U.$$
Hence, we have 
\begin{equation}\label{7.18}
\omega(e_2)=0.
\end{equation}
Note that 
(\ref{7.18}) does not hold for $e_2$, $e_1$ defined on $F_t(\Sigma)$ canonically. 
Combining (2.16) in \cite{CYZ} with (\ref{7.1}) and (\ref{7.18}), we obtain 
\begin{equation}\label{7.19}
\begin{split}
   \omega(T) & =e_1(\alpha)+2\alpha^2-\mbox{Im}A_{11}\mbox{ and hence}\\
    e_1e_1(\alpha) &=e_1(\omega(T))-4\alpha e_1(\alpha). 
\end{split}
\end{equation}

By the first and third formulas in (\ref{7.9}), we find
\begin{equation*}
\begin{split}
0&=d\omega(e_1,T)=e_1(\omega(T))-T(\omega(e_1))-\omega([e_1,T])\\
&=e_1(\omega(T))-T(\omega(e_1))-\omega\big((\mbox{Im}A_{11}+\omega(T))e_2\big)
=e_1(\omega(T))-T(\omega(e_1)).
\end{split}
\end{equation*}
This together with  
(\ref{7.2}), (\ref{7.18}) and the last equation in (\ref{7.9})
implies that 
\begin{equation}\label{7.20}
\begin{split}
e_1(\omega(T))&=T(\omega(e_1))=(V-\alpha e_2)(\omega(e_1))\\
&=V(H)-\alpha e_2(\omega(e_1))\\
&=V(H)-\alpha \big(e_1(\omega(e_2))+2W+\omega(e_2)^2+\omega(e_1)^2+2\omega(T)\big)\\
&=V(H)-\alpha\big(2W+H^2+2\omega(T)\big)
\end{split}
\end{equation}
in $\Sigma$. 
Substituting (\ref{7.20}) into (\ref{7.19}), 
we have 
\begin{equation*}
\begin{split}
 e_1e_1(\alpha) &=e_1(\omega(T))-4\alpha e_1(\alpha)\\
 &=V(H)-\alpha\big(2W+H^2+2\omega(T)\big)-4\alpha e_1(\alpha)\\
 &=V(H)-2W\alpha-\alpha H^2
 -2\alpha\big(e_1(\alpha)+2\alpha^2-\mbox{Im}A_{11}\big)
 -4\alpha e_1(\alpha)\\
 &=V(H)-6\alpha e_1(\alpha)-\alpha H^2-4\alpha^3-2W\alpha+2\alpha \mbox{Im}A_{11}, 
\end{split}
\end{equation*}
as required. 
\end{proof}

\begin{proof}[Proof of Lemma \ref{lem7.3}]
It follows from (7.5) in \cite{CYZ} that
\begin{equation*}
\begin{split}
de^1&=-e^2\wedge\omega+\theta\wedge(a_1e^1-a_2 e^2),\\
de^2&=e^1\wedge\omega-\theta\wedge(a_2e^1+a_1 e^2), 
\end{split}
\end{equation*}  
where $a_1$ and $a_2$ are defined as in (\ref{7.4}). 
Combining this with (\ref{7.1}), we have 
\begin{equation}\label{7.25}
\begin{split}
de^1&=-e^2\wedge\omega-c_2\theta\wedge e^2,\\
de^2&=e^1\wedge\omega-c_2\theta\wedge e^1. 
\end{split}
\end{equation}  
We compute 
\begin{equation}\label{7.26}
\begin{split}
d(\theta\wedge e^1) &
=d\theta\wedge e^1+\theta\wedge de^1+\theta\wedge de^1\\
&=(2e^1\wedge e^2)\wedge e^1+\theta\wedge (-e^2\wedge\omega-c_2\theta\wedge e^2)\\
&=\theta\wedge e^2\wedge \omega=-H\theta\wedge e^1\wedge e^2
\end{split}
\end{equation}
where we have used (\ref{7.2}), (\ref{7.25}) and the fact that $d\theta=2e^1\wedge e^2$. 
Hence, we find  
\begin{equation*}
\begin{split}
\frac{d}{dt}[F_t^*(\theta\wedge e^1)]
&=F_t^*L_{fe_2+gT}(\theta\wedge e^1)~~\mbox{(pullback omitted)}\\
&=i_{fe_2+gT}(-H\theta\wedge e^1\wedge e^2)
+di_{fe_2+gT}(\theta\wedge e^1)\\
&=-H(f\theta\wedge e^1+ge^1\wedge e^2)+d(ge^1)\\
&=-H(f\theta\wedge e^1+ge^1\wedge e^2)+dg\wedge e^1+gde^1\\
&=-H(f\theta\wedge e^1+ge^1\wedge e^2)
+\big(e_1(g) e^1+e_2(g)e^2+T(g)\theta\big)\wedge e^1\\
&\hspace{4mm}+g(-e^2\wedge\omega-c_2\theta\wedge e^2)\\
&=-(f-\alpha g)H\theta\wedge e^1+\big(\alpha e_2(g)+T(g)\big)\theta\wedge e^1
-\alpha g H\theta\wedge e^1\\
&=\big(-fH+V(g)\big)\theta\wedge e^1,
\end{split}
\end{equation*}
where we have used (\ref{7.2}), (\ref{7.3}), (\ref{7.25}) and (\ref{7.26}). 
This proves (\ref{7.21}).

From the third equation of (\ref{7.9}),
we observe that 
\begin{equation}\label{7.27}
\begin{split}
\frac{d}{dt}
(F_t^*\omega)
&=L_{fe^2+gT}\omega\\
&=d\big(f\omega(e_2)+g\omega(T)\big)+i_{fe_2+gT}(d\omega)\\
&=d\big(f\omega(e_2)+g\omega(T)\big)+i_{fe_2+gT}(-2We^1\wedge e^2)\\
&=d\big(f\omega(e_2)+g\omega(T)\big)+2Wfe^1.
\end{split}
\end{equation}
We compute 
\begin{equation}\label{7.28}
\begin{split}
\frac{dH}{dt}&=\frac{d}{dt}\big(\omega(e_1)\big)\\
&=(L_{fe^2+gT}\omega)(e_1)+\omega([fe^2+gT,e_1])\\
&=d\big(f\omega(e_2)+g\omega(T)\big)(e_1)+2Wf\\
&\hspace{4mm}
+\omega\big(-e_1(f)e_2+f[e_2,e_1]-e_1(g)T+g[T,e_1]\big)\\
&=e_1\big(f\omega(e_2)+g\omega(T)\big) +2Wf\\
&\hspace{4mm}
+\omega\big(-e_1(f)e_2+f(2T+\omega(e_1)e_1+\omega(e_2)e_2)-e_1(g)T+g(\mbox{Im}A_{11}+\omega(T))e_2\big)\\
&=e_1(f)\omega(e_2)+fe_1(\omega(e_2))+e_1(g)\omega(T)+ge_1(\omega(T))+2Wf-e_1(f)\omega(e_2)\\
&\hspace{4mm}+2f\omega(T)+f\omega(e_2)^2+f\omega(e_1)^2
-e_1(g)\omega(T)+g(\mbox{Im}A_{11}+\omega(T))\omega(e_2)\\
&=f\big[2W+e_1(\omega(e_2))+\omega(e_1)^2+\omega(e_2)^2+2\omega(T)\big]\\
&\hspace{4mm}+g\big[e_1(\omega(T))+\omega(e_2)\omega(T)+\mbox{Im}A_{11}\big],
\end{split}
\end{equation}
where the third equality follows from  (\ref{7.27})
and
the fourth equality follows from (\ref{7.9}).
Substituting (\ref{7.2}), (\ref{7.10})
and (\ref{7.11}) into (\ref{7.28})
and using the assumption (\ref{7.1}), 
we find 
\begin{equation}\label{7.29}
\begin{split}
\frac{dH}{dt}&=
e_1e_1(h)+2\alpha e_1(h)
+f\big[4e_1(\alpha)+4\alpha^2+H^2+2W-2\mbox{Im}A_{11}\big]\\
&\hspace{4mm}
+g\big[e_1e_1(\alpha)+2\alpha e_1(\alpha)
-h^{-1}e_1(h)\mbox{Im}A_{11}-2\alpha \mbox{Im}A_{11}+\mbox{Im}A_{11}\big].
\end{split}
\end{equation}
Substituting (\ref{7.17}) into (\ref{7.29}), 
we obtain (\ref{7.22}).

From (\ref{7.25}) and $e^2\wedge e^1=\alpha \theta\wedge e^1$ on $F_t(\Sigma)$ by 
(\ref{7.3}), we compute 
\begin{equation*}
\begin{split}
\frac{d}{dt}(F_t^* e^1)
&=L_{fe_2+gT}e^1=i_{fe_2+gT}(de^1)\\
&=i_{fe_2+gT}(-e^2\wedge\omega-c_2\theta\wedge e^2)\\
&=-f\omega(e_1)e^1-f\omega(T)\theta+g\omega(T)e^2-c_2g e^2
+c_2f\theta
\end{split}
\end{equation*}
and 
\begin{equation*}
\begin{split}
\frac{d}{dt}(F_t^* e^2)
&=L_{fe_2+gT}e^2=i_{fe_2+gT}(de^2)+df\\
&=i_{fe_2+gT}(e^1\wedge\omega-c_2\theta\wedge e^1)+df\\
&=-\big(f\omega(e_2)+g\omega(T)\big)e^1-g c_2e^1+df, 
\end{split}
\end{equation*}
and hence
\begin{equation}\label{7.30}
\begin{split}
\frac{d}{dt}(F_t^*(e^2\wedge e^1))
&=(L_{fe_2+gT}e^2)\wedge e^1+e^2\wedge( L_{fe_2+gT}e^1)\\
&=df\wedge e^1+e^2\wedge (-f\omega(e_1)e^1)\\
&=\big[\alpha e_2(f)+T(f)-\alpha H f\big]\theta\wedge e^1.
\end{split}
\end{equation}
Here, we have used $e^2\wedge e^1=\alpha \theta\wedge e^1$, $e^2\wedge \theta=0$
and $\omega(e_1)=H$ on $\Sigma_t(\Sigma)$
by (\ref{7.2}) and (\ref{7.3}). 
On the other hand, we have 
\begin{equation}\label{7.31}
\begin{split}
\frac{d}{dt}(F_t^*(e^2\wedge e^1))
&=\frac{d}{dt}(F_t^*(\alpha\theta\wedge e^1))\\
&\frac{d\alpha}{dt}\theta\wedge e^1+\alpha \frac{d}{dt}(F_t^*(\theta\wedge e^1))\\
&\frac{d\alpha}{dt}\theta\wedge e^1+\alpha \big(-fH+V(g)\big)\theta\wedge e^1
\end{split}
\end{equation}
where we have used (\ref{7.21}).
Combining (\ref{7.30}) and (\ref{7.31}), we obtain 
\begin{equation*}
\frac{d\alpha}{dt}+\alpha \big(-fH+V(g)\big)
=\alpha e_2(f)+T(f)-\alpha H f,
\end{equation*}
which gives (\ref{7.23}), since $V=T+\alpha e_2$ and $h=f-\alpha g$. 

We now compute 
\begin{equation}\label{7.32}
\begin{split}
&\frac{d}{dt}e_1(\alpha)=\frac{d}{dt}[(d\alpha)(e_1)]\\
&=(L_{fe_2+gT}d\alpha)(e_1)+d\alpha([fe_2+gT,e_1])\\
&=e_1(fe_2+gT)(\alpha)+[fe_2+gT,e_1](\alpha)\\
&=fe_1e_2(\alpha)+ge_1T(\alpha)+f[e_2,e_2](\alpha)+g[T,e_1](\alpha)\\
&=fe_1e_2(\alpha)+ge_1T(\alpha)+f\big(2T+\omega(e_1)e_1+\omega(e_2)e_2\big)(\alpha)
+g\big(\mbox{Im}A_{11}+\omega(T)\big)e_2(\alpha)
\end{split}
\end{equation}
where we have used (\ref{7.9}). Substituting (\ref{7.2}), 
(\ref{7.10})
and (\ref{7.11}) into (\ref{7.32}) yields
\begin{equation*} 
\begin{split}
\frac{d}{dt}e_1(\alpha)
&=fe_1e_2(\alpha)+ge_1V(\alpha)-ge_1(\alpha)e_2(\alpha)-g\alpha e_1e_2(\alpha)\\
&\hspace{4mm}+2fV(\alpha)-2\alpha fe_2(\alpha)+fHe_1(\alpha)
+f\big(h^{-1}e_1(h)+2\alpha\big)e_2(\alpha)\\
&\hspace{4mm}+g(\mbox{Im}A_{11})e_2(\alpha)+g\big(e_1(\alpha)-\alpha h^{-1}e_1(h)-\mbox{Im}A_{11}\big)e_2(\alpha)\\
&=fe_1e_2(\alpha)+ge_1 V(\alpha)-\alpha g e_1e_2(\alpha)\\
&\hspace{4mm}+2fV(\alpha)+fHe_1(\alpha)+fh^{-1}e_1(h)e_2(\alpha)-\alpha gh^{-1}e_1(h)e_2(\alpha)\\
&=he_1e_2(\alpha)+ge_1V(\alpha)+fHe_1(\alpha)+e_1(h)e_2(\alpha)+2fV(\alpha)\\
&=e_1V(h)+ge_1V(\alpha)+2fV(\alpha)+fHe_1(\alpha),
\end{split}
\end{equation*}
where we have used (\ref{7.12}) in the last equality. 
This proves (\ref{7.24})
and completes the proof. 
\end{proof}

\begin{proof}[Proof of Theorem \ref{thm7.1}]
From (\ref{7.0}), for $f$, $g$, and hence $h=f-\alpha g$ having compact support in nonsingular 
domain of $\Sigma$, we compute 
\begin{align*}
\begin{split}
&\frac{d}{dt}\int_{F_t(\Sigma)}dA_2\\
&=\frac{d}{dt}\int_\Sigma F_t^*\left(\left[\frac{2}{3}e_1(\alpha)+\frac{4}{3}\alpha^2-\frac{2}{3}\mbox{Im} A_{11}
+\frac{1}{6}W+\frac{2}{27}H^2\right]H\right)F_t^*(\theta\wedge e^1)\\
&=\frac{2}{3}\int_\Sigma \left[\left(e_1(\alpha)+2\alpha^2-\mbox{Im} A_{11}+\frac{1}{4}W+\frac{1}{3}H^2\right)\frac{dH}{dt}
+H\frac{d}{dt}\big(e_1(\alpha)\big)+4\alpha H\frac{d\alpha}{dt}\right]\theta\wedge e^1\\
&\hspace{4mm}+\int_\Sigma\left[\frac{2}{3}e_1(\alpha)+\frac{4}{3}\alpha^2-\frac{2}{3}\mbox{Im} A_{11}
+\frac{1}{6}W+\frac{2}{27}H^2\right]H\frac{d}{dt}\big[F_t^*(\theta\wedge e^1)\big]\\
&:=I+II,
\end{split}
\end{align*}
where 
\begin{align*}
\begin{split}
I&:=\frac{2}{3}\int_\Sigma \Bigg[\left(e_1(\alpha)+2\alpha^2-\mbox{Im} A_{11}+\frac{1}{4}W+\frac{1}{3}H^2\right)\frac{dH}{dt}
+H\frac{d}{dt}\big(e_1(\alpha)\big)+4\alpha H\frac{d\alpha}{dt}\Bigg]\theta\wedge e^1\\
&=\frac{2}{3}\int_\Sigma  \left(e_1(\alpha)+2\alpha^2-\mbox{Im} A_{11}+\frac{1}{4}W+\frac{1}{3}H^2\right)\\
&\hspace{4mm}\cdot \Bigg\{e_1e_1(h)+2\alpha e_1(h)+4h\left(e_1(\alpha)+\alpha^2+\frac{1}{4}H^2+\frac{1}{2}W-\frac{1}{2}\mbox{Im}A_{11}\right)
+g V(H) \Bigg\}\theta\wedge e^1\\
&\hspace{4mm}+\frac{2}{3}\int_\Sigma H\Big[e_1V(h)+ge_1V(\alpha)+2fV(\alpha)+fHe_1(\alpha)\Big]\theta\wedge e^1\\
&\hspace{4mm}+\frac{8}{3}\int_\Sigma \alpha H\big[V(h)+gV(\alpha)\Big]\theta\wedge e^1
\end{split}
\end{align*}
by (\ref{7.22})-(\ref{7.24}) and the assumption (\ref{condition}), and 
\begin{equation*}
\begin{split}
II&:=\int_\Sigma\left[\frac{2}{3}e_1(\alpha)+\frac{4}{3}\alpha^2-\frac{2}{3}\mbox{Im} A_{11}
+\frac{1}{6}W+\frac{2}{27}H^2\right]H\frac{d}{dt}\big[F_t^*(\theta\wedge e^1)\big]\\
&=\int_\Sigma\left[\frac{2}{3}He_1(\alpha)+\frac{4}{3}\alpha^2H-\frac{2}{3}H\mbox{Im} A_{11}
+\frac{1}{6}WH+\frac{2}{27}H^3\right]
\big(-fH+V(g)\big)\theta\wedge e^1
\end{split}
\end{equation*}
by (\ref{7.21}). 
Let 
\begin{align*}
H_1&=\int_\Sigma\frac{2}{3}\left(e_1(\alpha)+2\alpha^2-\mbox{Im} A_{11}+\frac{1}{4}W+\frac{1}{3}H^2\right)
\big[e_1e_1(h)+2\alpha e_1(h)\big]\theta\wedge e^1\\
H_2&=\int_\Sigma\frac{8}{3}\left(e_1(\alpha)+2\alpha^2-\mbox{Im} A_{11}+\frac{1}{4}W+\frac{1}{3}H^2\right)\\
&\hspace{8mm}\cdot
\left(e_1(\alpha)+\alpha^2+\frac{1}{4}H^2+\frac{1}{2}W-\frac{1}{2}\mbox{Im}A_{11}\right)h \theta\wedge e^1\\
H_3&=\int_\Sigma\left(\frac{2}{3}H e_1V(h)+\frac{8}{3}\alpha H V(h)\right)\theta\wedge e^1\\
F&=\int_\Sigma\left[\frac{4}{3}HV(\alpha)-H\left(\frac{4}{3}\alpha^2H+\frac{2}{27}H^3
+\frac{1}{6}WH-\frac{2}{3}\mbox{Im}A_{11}H\right)\right]f\theta\wedge e^1\\
G_1&=\int_\Sigma\frac{2}{3} \left(e_1(\alpha)+2\alpha^2-\mbox{Im} A_{11}+\frac{1}{4}W+\frac{1}{3}H^2\right)
V(H)g\theta\wedge e^1\\
&\hspace{4mm}+\int_\Sigma\left[\frac{2}{3}He_1V(\alpha)+\frac{8}{3}\alpha H V(\alpha)\right]g\theta\wedge \theta^1\\
G_2&=\int_\Sigma
\left[\frac{2}{3}He_1(\alpha)+\frac{4}{3}\alpha^2H-\frac{2}{3}H\mbox{Im} A_{11}
+\frac{1}{6}WH+\frac{2}{27}H^3\right]
 V(g) \theta\wedge e^1.
\end{align*}
Thus we have 
\begin{equation}\label{7.36}
\frac{d}{dt}\int_{F_t(\Sigma)}dA_2
=I+II=H_1+H_2+H_3+F+G_1+G_2.
\end{equation}
We compute 
\begin{align*}
H_1&=\int_\Sigma\frac{2}{3}\left(e_1(\alpha)+2\alpha^2-\mbox{Im} A_{11}+\frac{1}{4}W+\frac{1}{3}H^2\right)
\big[e_1e_1(h)+2\alpha e_1(h)\big]\theta\wedge e^1\\
&=-\frac{2}{3}\int_\Sigma e_1\left(e_1(\alpha)+2\alpha^2-\mbox{Im} A_{11}+\frac{1}{4}W+\frac{1}{3}H^2\right)
e_1(h)\theta\wedge e^1\\
&=-\frac{2}{3}\int_\Sigma\left(e_1e_1(\alpha)+4\alpha e_1(\alpha)+\frac{2}{3}He_1(H)\right)e_1(h)\theta\wedge e^1\\
&=-\frac{2}{3}\int_\Sigma\left(\frac{2}{3}He_1(H)
-2\alpha e_1(\alpha)-4\alpha^3-\alpha H^2
V(H)-2W\alpha+2\alpha \mbox{Im}A_{11}
\right)e_1(h)\theta\wedge e^1\\
&=\frac{2}{3}\int_\Sigma \Bigg\{e_1\left(\frac{2}{3}He_1(H)
-2\alpha e_1(\alpha)-4\alpha^3-\alpha H^2
V(H)-2W\alpha+2\alpha \mbox{Im}A_{11}
\right)\\
&\hspace{4mm}+2\alpha\left(\frac{2}{3}He_1(H)
-2\alpha e_1(\alpha)-4\alpha^3-\alpha H^2
V(H)-2W\alpha+2\alpha \mbox{Im}A_{11}
\right)\Bigg\}h\theta\wedge e^1\\
&=\frac{4}{9}
\int_\Sigma\Bigg[He_1e_1(H)+e_1(H)^2-3e_1(\alpha)^2-\alpha He_1(H)-\frac{3}{2}H^2e_1(\alpha)\\
&\hspace{4mm}-6\alpha^2e_1(\alpha)+\frac{3}{2}e_1V(H)-3We_1(\alpha)+3\mbox{Im}A_{11}e_1(\alpha)\Bigg]h\theta\wedge e^1,
\end{align*}
where we have used (\ref{7.1}), (\ref{7.17}) and Lemma \ref{lem7.4}. 
We also compute 
\begin{align*}
H_3&=\int_\Sigma\left(\frac{2}{3}H e_1V(h)+\frac{8}{3}\alpha H V(h)\right)\theta\wedge e^1\\
&=-\int_\Sigma\frac{2}{3}\big[e_1(H)+2\alpha H\big]V(h)\theta\wedge e^1
+\int_\Sigma\frac{8}{3}\alpha HV(h)\theta\wedge e^1\\
&=-\int_\Sigma\frac{2}{3}\big[e_1(H)-2\alpha H\big]V(h)\theta\wedge e^1\\
&=\int_\Sigma \frac{2}{3}\Big[V\big(e_1(H)-2\alpha H\big)-\alpha H\big(e_1(H)-2\alpha H\big)\Big]h\theta\wedge e^1\\
&=\int_\Sigma\left[\frac{2}{3}e_1V(h)-\frac{4}{3}HV(\alpha)+\frac{4}{3}\alpha^2H^2\right] h\theta\wedge e^1,
\end{align*}
where we have used 
\begin{align}\label{7.35}
\begin{split}
[e_1,V]&=[e_1,T+\alpha e_2]=[e_1,T]+e_1(\alpha)e_2+\alpha[e_1,e_2]\\
&=-(\mbox{Im}A_{11}+\omega(T))e_2+e_1(\alpha)e_2+\alpha\big[-2T-\omega(e_1)e_1-\omega(e_2)e_2\big]\\
&=\alpha h^{-1}e_1(h)e_2-\alpha He_1-\alpha h^{-1} e_1(h) e_2-2\alpha^2 e_2-2\alpha T\\
&=-\alpha He_1-2\alpha V
\end{split}
\end{align}
by (\ref{7.9}) and (\ref{7.11}).
Therefore, we obtain 
\begin{align}\label{7.37}
\begin{split}
&H_1+H_2+H_3\\
&=\frac{4}{9}\int_\Sigma\left[He_1e_1(H)+3e_1V(H)+e_1(H)^2+\frac{1}{2}H^4\right]h\theta\wedge e^1\\
&\hspace{4mm}+\frac{4}{3}\int_\Sigma\big[e_1(\alpha)^2+4\alpha^2e_1(\alpha)+4\alpha^4\big]h\theta\wedge e^1\\
&\hspace{4mm}-\frac{4}{9}\int_\Sigma\big[\alpha H e_1(H)+3HV(\alpha)-2H^2e_1(\alpha)-8\alpha^2H^2\big]h\theta\wedge e^1\\
&\hspace{4mm}+\frac{2}{3}\int_\Sigma W\left(e_1(\alpha)+\frac{11}{12}H^2+5\alpha^2+\frac{1}{2} W\right)h\theta\wedge e^1\\
&\hspace{4mm}+\frac{8}{3}\int_\Sigma 
\mbox{Im}A_{11}\left[\frac{1}{2}\mbox{Im}A_{11}-e_1(\alpha)-2\alpha^2-\frac{5}{8}W-\frac{5}{12}H^2\right]h\theta\wedge e^1.
\end{split}
\end{align}
By (\ref{7.1}), (\ref{7.35}) and Lemma \ref{lem7.4},
we compute 
\begin{align*}
G_2&=\int_\Sigma
\left[\frac{2}{3}He_1(\alpha)+\frac{4}{3}\alpha^2H-\frac{2}{3}H\mbox{Im} A_{11}
+\frac{1}{6}WH+\frac{2}{27}H^3\right]
 V(g) \theta\wedge e^1\\
 &=-\int_\Sigma
V\left[\frac{2}{3}He_1(\alpha)+\frac{4}{3}\alpha^2H-\frac{2}{3}H\mbox{Im} A_{11}
+\frac{1}{6}WH+\frac{2}{27}H^3\right]
g\theta\wedge e^1\\
&\hspace{4mm}+\int_\Sigma
\alpha H\left[\frac{2}{3}He_1(\alpha)+\frac{4}{3}\alpha^2H-\frac{2}{3}H\mbox{Im} A_{11}
+\frac{1}{6}WH+\frac{2}{27}H^3\right]
g\theta\wedge e^1\\
 &=-\int_\Sigma
\Bigg[\frac{2}{3}HVe_1(\alpha)+\frac{2}{3}e_1(\alpha)V(H)+\frac{4}{3}\alpha^2V(H)+\frac{8}{3}\alpha HV(\alpha)\\
&\hspace{8mm}-\frac{2}{3}V(H)\mbox{Im} A_{11}
+\frac{1}{6}WV(H)+\frac{2}{9}H^2V(H)\Bigg]
g\theta\wedge e^1\\
&\hspace{4mm}+\int_\Sigma
\alpha H\left[\frac{2}{3}He_1(\alpha)+\frac{4}{3}\alpha^2H-\frac{2}{3}H\mbox{Im} A_{11}
+\frac{1}{6}WH+\frac{2}{27}H^3\right]
g\theta\wedge e^1\\
&=\int_\Sigma\Bigg[-\frac{2}{3}He_1V(\alpha)-\frac{2}{3}e_1(\alpha)V(H)-\frac{4}{3}\alpha^2V(H)-4\alpha HV(\alpha)-\frac{2}{9}H^2V(H)\\
&\hspace{8mm}+\frac{2}{3}V(H)\mbox{Im} A_{11}-\frac{1}{6}WV(H)+\frac{4}{3}\alpha^3H^2+\frac{2}{27}\alpha H^4\\
&\hspace{8mm}-\frac{2}{3}\alpha H^2\mbox{Im}A_{11}+\frac{1}{6}W\alpha H^2\Bigg]
g\theta\wedge e^1.
\end{align*}
Therefore, we have 
\begin{align}\label{7.38}
\begin{split}
G_1+G_2&=\int_{\Sigma}\left(-\frac{4}{3}HV(\alpha)+\frac{4}{3}\alpha^2H^2+\frac{2}{27}H^4
+\frac{1}{6}WH^2-\frac{2}{3} \mbox{Im}A_{11}H^2 \right)\alpha g\theta\wedge e^1, \\
F+G_1+G_2&=\int_{\Sigma}\left(\frac{4}{3}HV(\alpha)-\frac{4}{3}\alpha^2H^2-\frac{2}{27}H^4
-\frac{1}{6}WH^2+\frac{2}{3} \mbox{Im}A_{11}H^2 \right)h\theta\wedge e^1.
\end{split}
\end{align}
Now (\ref{7.34}) 
follows from combining (\ref{7.36}), (\ref{7.37}) and (\ref{7.38}).
This proves Theorem \ref{thm7.1}. 
\end{proof}

\begin{proof}[Proof of Lemma \ref{lem7.5}]
It follows from  (\ref{7.1}), (\ref{7.10}) and (\ref{7.9})  that
\begin{align*}
[f_2+gT,V]
&=f[e_2,V]+g[T,V]-V(f)e_2-V(g)T\\
&=f\big(e_2(\alpha)e_2+[e_2,T]\big)
+g\big(T(\alpha)e_2+\alpha[T,e_2]\big)-V(f)e_2-V(g)T\\
&=f\big(e_2(\alpha)e_2+(-\mbox{Im}A_{11}+\omega(T))e_1\big)\\
&\hspace{4mm}
+g\big(T(\alpha)e_2-\alpha (-\mbox{Im}A_{11}+\omega(T))e_1\big)-V(f)e_2-V(g)T\\
&=f\Big(e_2(\alpha)e_2+\big(-2\mbox{Im}A_{11}+e_1(\alpha)-\alpha h^{-1}e_1(h)\big)e_1\Big)\\
&\hspace{4mm}
+g\Big(T(\alpha)e_2-\alpha\big(-2\mbox{Im}A_{11}+e_1(\alpha)-\alpha h^{-1}e_1(h)\big)e_1\Big)-V(f)e_2-V(g)T\\
&=h e_1(\alpha)e_1-\alpha e_1(h)e_1-2h\mbox{Im}A_{11}e_1
-V(g)V.
\end{align*}  
Hence, we have 
\begin{align*}
\frac{d}{ds}V(\alpha)
&=V\left(\frac{d\alpha}{ds}\right)+[f_2+gT,V](\alpha)\\
&=V\left(\frac{d\alpha}{ds}\right)+h e_1(\alpha)e_1(\alpha)-\alpha e_1(h)e_1(\alpha)-2h\mbox{Im}A_{11}e_1(\alpha)
-V(g)V(\alpha).
\end{align*}
Similarly, we have
\begin{align*}
&\frac{d}{ds}V(H_{cr})=V\left(\frac{d}{ds}H_{cr}\right)+[f_2+gT,V](H_{cr})\\
&=V\left(\frac{d}{ds}H_{cr}\right)+
h e_1(\alpha)e_1(H_{cr})-\alpha e_1(h)e_1(H_{cr})-2h\mbox{Im}A_{11}e_1(H_{cr})
-V(g)V(H_{cr}).
\end{align*}

It follows from (\ref{7.2}), (\ref{7.10}), (\ref{7.11}), and 
(\ref{7.9}) 
that 
\begin{align*}
&[fe_2+gT,e_1]
=f[e_2,e_1]-e_1(f) e_2+g[T,e_1]-e_1(g)T\\
&=f\big(2T+\omega(e_1)e_1+\omega(e_2)e_2\big)-e_1(f) e_2
+g\big(\mbox{Im}A_{11}+\omega(T)\big)e_2-e_1(g)T\\
&=f\Big(2T+He_1+(h^{-1}e_1(h)+2\alpha)e_2\Big)-e_1(f) e_2\\
&\hspace{4mm}+g\Big(e_1(\alpha)-\alpha h^{-1}e_1(h)\Big)e_2-e_1(g)T\\
&=fHe_1 +2fV-e_1(g)V.
\end{align*}
This together with (\ref{7.22}) implies that 
\begin{align*}
&\frac{d}{ds}e_1(H)
=e_1\left(\frac{dH}{ds}\right)+[fe_2+gT,e_1](H)\\
&=e_1\Bigg(e_1e_1(h)+2\alpha e_1(h)+4h\left(e_1(\alpha)+\alpha^2+\frac{1}{4}H^2+\frac{1}{2}W-\frac{1}{2}\mbox{Im}A_{11}\right)\\
&\hspace{8mm}
+g\big(V(H)-2\alpha\mbox{Im}A_{11}-h^{-1}e_1(h)\mbox{Im}A_{11}+\mbox{Im}A_{11}\big)\Bigg)\\
&\hspace{4mm}+fHe_1(H) +2fV(H)-e_1(g)V(H)\\
&=e_1\Bigg(e_1e_1(h)+2\alpha e_1(h)+4he_1(\alpha)+H^2h+4\alpha^2h+2Wh-2h\mbox{Im}A_{11}\\
&\hspace{4mm}
-g\big(
2\alpha\mbox{Im}A_{11}+h^{-1}e_1(h)\mbox{Im}A_{11}-\mbox{Im}A_{11}\big)\Bigg)
+ge_1V(H)+fHe_1(H) +2fV(H).
\end{align*}
Similarly, we have 
\begin{align*}
\frac{d}{ds}\big(e_1(|H_{cr}|\mathfrak{f})\big)
&=e_1\left(\frac{d}{ds}(|H_{cr}|\mathfrak{f})\right)+
\big(fHe_1 +2fV-e_1(g)V\big)(|H_{cr}|\mathfrak{f}).
\end{align*}
Recall from (\ref{8.3})
that 
\begin{align*}
|H_{cr}|\mathfrak{f}
&=e_1(H) H_{cr}+\frac{3}{2}V(H_{cr})+\frac{1}{2}H e_1(H_{cr})
-\alpha H H_{cr}. 
\end{align*}
Therefore, we have 
\begin{align*}
\frac{d}{ds}(|H_{cr}|\mathfrak{f})
&= H_{cr}\frac{d}{ds}\big(e_1(H)\big)
+e_1(H)\frac{d}{ds}H_{cr}+\frac{3}{2}\frac{d}{ds}V(H_{cr})\\
&\hspace{4mm}+\frac{1}{2}e_1(H_{cr})\frac{dH}{ds}+\frac{1}{2}H\frac{d}{ds}\big(e_1(H_{cr})\big)\\
&\hspace{4mm}-\frac{d\alpha}{ds} H H_{cr}-\alpha\frac{dH}{ds}H_{cr}
-\alpha H\frac{d}{ds}H_{cr}. 
\end{align*}
Recall from (\ref{5.2}) that 
$$H_{cr}=e_1(\alpha)+\frac{1}{2}\alpha^2-\mbox{Im} A_{11}+\frac{1}{4}W+\frac{1}{6}H^2.$$
Differentiating it with respect to $t$ and using 
(\ref{7.1}) and (\ref{7.22})-(\ref{7.24}), we find 
\begin{align*}
&\frac{d}{ds}H_{cr}=\frac{d}{ds}e_1(\alpha)+\alpha\frac{d\alpha}{ds}+\frac{1}{3}H\frac{dH}{ds}\\
&=e_1V(h)+\alpha V(h)+\frac{1}{3}H\big(e_1e_1(h)+2\alpha e_1(h)\big)\\
&\hspace{4mm}+\frac{4}{3}H\left(e_1(\alpha)+\alpha^2+\frac{1}{4}H^2+\frac{1}{2}W-\frac{1}{2}\mbox{Im}A_{11}\right)h
+\big(2V(\alpha)+He_1(\alpha)\big)f\\
&\hspace{4mm}+\left(e_1V(\alpha)+\alpha V(\alpha)
+\frac{1}{3}H\big(V(H)-2\alpha\mbox{Im}A_{11}-h^{-1}e_1(h)\mbox{Im}A_{11}+\mbox{Im}A_{11}\big)\right)g.
\end{align*}
This proves the assertion. 
\end{proof}

\section*{Acknowledgement}

The author would like to thank Prof. Jih-Hsin Cheng
for several helpful discussions which lead to results 
in this paper, and would like to thank the anonymous referee
for his/her helpful comments and suggestions that improved the
quality of this paper. 
The author was supported  by the National Science and Technology Council (NSTC),
Taiwan, with grant Number: 114-2115-M-032 -003 -MY2

\bibliographystyle{amsplain}

\end{document}